\documentclass[a4paper,reqno, colorinlistoftodos]{amsart}

\usepackage{amsthm,amssymb,amsmath,mathtools,pdfpages,changes,exscale,cite,color,amsopn,color}
\usepackage[colorlinks=true,pdftex,unicode=true,linktocpage,bookmarksopen,hypertexnames=false]{hyperref}

\usepackage[capitalise]{cleveref}

\crefname{subsection}{Subsection}{Subsections}

\usepackage{soul}

\newcommand{\com}[1]{{\color{blue}#1}}

\renewcommand{\leq}{\leqslant}
\renewcommand{\geq}{\geqslant}

\DeclareMathOperator{\Aut}{Aut}

\DeclareMathOperator{\End}{End}

\DeclareMathOperator{\im}{Im}

\DeclareMathOperator{\id}{id}

\DeclareMathOperator{\Map}{Map}

\DeclareMathOperator{\Sym}{Sym}

\DeclareMathOperator{\cm}{\sigma}

\newtheorem{theorem}{Theorem}[section]
\newtheorem{lemma}[theorem]{Lemma}
\newtheorem{proposition}[theorem]{Proposition}
\newtheorem{corollary}[theorem]{Corollary}
\newtheorem{thmx}{Theorem}

\theoremstyle{definition}
\newtheorem{definition}[theorem]{Definition}

\newtheorem{example}[theorem]{Example}

\makeatletter
\providecommand\@dotsep{5}

\makeatother

\let\oldtocsection=\tocsection

\let\oldtocsubsection=\tocsubsection

\renewcommand{\tocsection}[2]{\hspace{0em}\oldtocsection{#1}{#2}}
\renewcommand{\tocsubsection}[2]{\hspace{2em}\oldtocsubsection{#1}{#2}}

\usepackage{xspace}
\newcommand{\ourstructure}{YB-semitruss\xspace}
\newcommand{\ourstructures}{YB-semitrusses\xspace}

\numberwithin{equation}{section}


\begin{document}

\title[YB-semitrusses]{Left non-degenerate set-theoretic solutions of the Yang-Baxter equation and semitrusses}
\author{I. Colazzo, E. Jespers, A. Van Antwerpen and C. Verwimp }
\date{}
\address[I. Colazzo (ORCID: 0000-0002-2713-0409)]{Department of Mathematics,
College of Engineering, Mathematics and Physical Sciences, University of Exeter, Exeter EX4 4QF, UK
}
\email{I.Colazzo@exeter.ac.uk}
\address[E. Jespers (ORCID: 0000-0002-2695-7949), A. Van Antwerpen (ORCID: 0000-0001-7619-6298), and C. Verwimp (ORCID: 0000-0003-3128-4854)]{Department of Mathematics and Data Science,
Vrije Universiteit Brussel, Pleinlaan 2, 1050 Brussel}
\email{eric.jespers@vub.be}
\email{arne.van.antwerpen@vub.be}
\email{charlotte.verwimp@vub.be}

\maketitle

\begin{abstract}
To determine and analyze arbitrary left non-degenerate set-theoretic solutions of the Yang-Baxter equation (not necessarily bijective), we introduce an associative algebraic structure, called a \ourstructure, that forms a subclass of the category of semitrusses as introduced by Brzezi\'nski.
Fundamental examples of \ourstructures are structure monoids of left non-degenerate set-theoretic solutions and (skew) left braces. Gateva-Ivanova and Van den Bergh introduced structure monoids and showed their importance (as well as that of the structure algebra) for studying involutive non-degenerate solutions. 
Skew left  braces were introduced by Guarnieri, Vendramin and Rump to deal with bijective non-degenerate solutions.
Hence, \ourstructures also yield a unified treatment of these different algebraic structures.
The algebraic structure of \ourstructures is investigated and as a consequence it is proven, for example, that
any finite left non-degenerate set-theoretic solution of the Yang-Baxter equation is 
right non-degenerate if and only if it is bijective. Furthermore, it is
shown that some finite left non-degenerate solutions can be reduced to non-degenerate
solutions of smaller size. The structure algebra of a finitely generated \ourstructure
is an algebra defined by homogeneous quadratic relations. We prove that it often
is a  left Noetherian algebra of finite Gelfand-Kirillov dimension  that satisfies a 
polynomial identity, but in general it is not right Noetherian.
\end{abstract}
\maketitle

\tableofcontents

\section{Introduction}
Let $V$ be a vector space over a field $K$. A linear map $R\colon V\otimes V\to V\otimes V$ is called a solution of the Yang--Baxter equation if \[(R\otimes{\id})({\id}\otimes R)(R\otimes{\id})=({\id}\otimes R)(R\otimes{\id})({\id}\otimes R).\] 
This  equation originates from papers by Baxter \cite{Ba} and Yang \cite{Ya} on statistical physics and the search for solutions has
attracted numerous studies both in mathematical physics and pure mathematics.
As the study of arbitrary solutions is complex, Drinfeld in \cite{Drinfeld} initiated the investigations of set-theoretic solutions of the Yang-Baxter equation. These are  solutions which are induced
by a linear extension of a map $r\colon X\times X\to X\times X$, where $X$ is a basis of $V$. In this case $r$ satisfies
\[(r\times{\id})({\id}\times r)(r\times{\id})=({\id}\times r)(r\times{\id})({\id}\times r),\]
and one says that $(X,r)$ is a set-theoretic solution of the Yang--Baxter equation (throughout the paper we will simply call this a solution).
For any $x,y \in X$, we put $r(x,y)=(\lambda_x(y),\rho_y(x))$. The solution
is said to be bijective if $r$ is bijective. If furthermore $r^{2}=\id$ then it is said to be an involutive solution. A solution is said to be 
left non-degenerate (respectively right non-degenerate)  if all maps $\lambda_x$ (respectively all maps $\rho_x$) are bijective; a non-degenerate 
solution is one that is a left and right non-degenerate solution.
Solutions of the Yang-Baxter equation also are of central importance for many mathematical topics, including quantum groups, Hopf algebras, Galois theory, group algebras, knot theory, radical rings, pre-Lie algebras, solvable groups and Garside groups (see for example \cite{AngionoVendramin,CJDR,CJORet,Chouraqui,Dehornoy,EPGS,ESS99,FCC,GIVdB98,KasselBook,LeVe17,RumpRR,RumpClassical,SmokPre,SmokPre1}).

In recent years several fundamental new algebraic structures have been introduced to describe certain classes of solutions.
Rump  \cite{Rump2007} introduced the notion of a (left) brace, that is a tuple $(B,+,\circ)$ with $(B,+)$ an abelian group and $(B,\circ)$ a group so that  
$a\circ (b+c) = (a\circ b) -a +(a\circ c)$ holds for all $a,b,c\in B$ (see also \cite{CJO14}).   If one assumes that $(B,+)$ is an arbitrary group then one obtains the
notion of a skew left brace as introduced by Guarnieri and Vendramin \cite{GV17}. 
It has been shown by Bachiller, Ced\'o and Jespers  in \cite{BCJ2016} that all 
involutive non-degenerate solutions can be explicitly constructed from the description of all left braces, and later this has been extended  by Bachiller  in \cite{B2018}
to all  bijective non-degenerate  solutions from all  skew left braces.
This reduces  the classification of all bijective non-degenerate solutions to classify all  skew left braces.

The starting points of all these investigations is in the work of Gateva-Ivanova and Van den Bergh \cite{GIVdB98}, Etingoff, Schedler and 
Soloviev \cite{ESS99},
where a beautiful monoid and group theory translation was given of
 involutive non-degenerate solutions $(X,r)$, associating a monoid and group, denoted by $M(X,r)$ and $G(X,r)$ respectively. These are   called the structure monoid and structure group.
Recall that, for a solution $(X,r)$, the structure monoid  $M(X,r)$ is the monoid (with operation written multiplicatively as $\circ$) generated by the elements 
of $X$ 
subject to the relations $x\circ y =\lambda_x (y) \circ \rho_y(x)$. If one considers the group generated by $X$ subject to these 
relations, considered as group relations, then one obtains the structure group $G(X,r)$. 
For a field $K$, the structure algebra is the $K$-algebra defined by these relations, and hence, since the relations are homogeneous, (quadratic), this is the monoid algebra $K[M(X,r)]$. It turns out \cite{ESS99,JOItype} that if $(X,r)$ is an involutive non-degenerate solution then $M(X,r)$ is a 
regular submonoid of 
the holomorph of the free abelian monoid $FaM_{X}$  
of rank $|X|$ (with basis $X$) and $G(X,r)$ is the group of fractions of this cancellative monoid. The group $G(X,r)$ is a Bieberbach group  and it is a regular subgroup of 
the holomorph of the free abelian group $Fa_{X}$ (with basis $X$) of rank $|X|$.
So, $M(X,r) =\{ (a, \lambda_a) \mid a \in FaM_{X} \}$ is a submonoid of the natural semidirect product $FaM_{X}\rtimes \Sym (X)$, 
where $\lambda : FaM_{X} \rightarrow \Aut (FaM_{X}): a\mapsto \lambda_a$  
is a map so 
that $\lambda_1=\id$ and $\lambda_x$ is the mapping obtained from the defining solution $(X,r)$.
Furthermore, as shown by Gateva-Ivanova and Van den Bergh \cite{GIVdB98}, in this case, the  structure algebra $K[M(X,r)]$ has a rich algebraic structure that resembles that of a polynomial algebra in finitely many commuting variables. In particular,
if  $X$ is finite then it is left and right Noetherian domain that  satisfies a polynomial identity (and it is a maximal order). For details and a survey we refer the reader to \cite{JOBook}.

Lu, Yan and Zhu \cite{LuYaZhu00} and Soloviev \cite{So00} considered arbitrary bijective non-degenerate solutions and studied their structure groups. In contrast to the involutive case, the set $X$ is not necessarily canonically embedded into the structure group
$G(X,r)$ see \cite{SmVe18, LeVe2019}. This problem can be avoided by working with the structure monoid and many of the above mentioned results  have been extended to arbitrary bijective (left) non-degenerate solutions \cite{JeKuVA19,JeKuVA19Cor}. Furthermore, it has been shown that the algebraic structure of the  algebra $K[M(X,r)]$ determines crucial information of the solutions, for example, it is a domain (or equivalently it is a prime algebra) if and only if $(X,r)$ is involutive.
The structure algebra $R=K[M(X,r)]=\bigoplus_{n\in \mathbb{N}} R_n$ has a natural gradation and  it is a locally finite connected graded algebra, i.e. $\dim_{K}(R_n) < \infty$ and $R_0=K$. Moreover,  it is generated by 
the elements of degree 1. In \cite{GIComApp}, Gateva-Ivanova  posed several interesting questions 
on $\dim_{K}(R_2)$ for non-degenerate finite solutions $(X,r)$. This dimension is at 
most ${|X| \choose 2}$ and  this bound   is reached precisely when  $(X,r)$ is involutive (see  \cite{JeKuVA19}
and \cite{GIVdB98}).
In \cite{CJOMinCond} the minimal possible values of the dimension of $R_2$ 
are determined and one completely classifies  non-degenerate finite solutions $(X, r)$ for which these bounds are attained.

Rump   introduced some other structures to study bijective  (left) non-degenerate solutions $(X,r)$, such as cycle sets  \cite{Ru05}, $q$-cycle sets and $q$-braces \cite{Ru19}. 
Cycle sets and $q$-cycle sets are in a one-to-one correspondence with left non-degenerate involutive, respectively left non-degenerate solutions. Moreover, $q$-braces are groups equipped with two binary operations that give a $q$-cycle set structure. Such structures  describe non-degenerate $q$-cycle sets, i.e. non-degenerate bijective solutions.
Skew left braces are an intermediate step between left braces and q-braces.  

Brzezi\'nski in \cite{Br18} initiated another notion, called a left semitruss, this to put the algebraic structures of skew left braces and structure monoids into  a larger context, purely from an algebraic point of view, without making the link with solutions.
A left semitruss is  a tuple $(A,+,\circ , \lambda)$ such that $(A,+)$ and $(A,\circ)$ are  semigroups and $\lambda : A\rightarrow  \Map(A,A)$ is a function so that
$a\circ (b+c) =(a\circ b) + \lambda_a(c)$, for all $a,b,c\in A$. A  skew left brace $(A,+,\circ)$ is a left semitruss with $(A,+)$ and $(A,\circ)$ groups and 
$\lambda_a(c) =-a + (a\circ c)$.
Based on a fundamental result of Gateva-Ivanova and Majid \cite{GIMa08}, showing that any solution $(X,r)$  can be lifted to a solution on the structure monoid $M(X,r)$,  Ced\'o, Jespers and Verwimp showed in \cite{CJV2020} that $M(X,r)$ is a left semitruss (by considering it as a regular submonoid of the structure monoid of an associated shelf solution)  for any left non-degenerate solution $(X,r)$.

The aim of this paper is to study  arbitrary left non-degenerate solutions via the study of certain, strongly related, associative structures. 
For this, we consider within the category of left semitrusses those that  yield left non-degenerate solutions, and we therefore will call these objects \ourstructures. It turns out that  such a \ourstructure  is a tuple $(A,+,\circ, \lambda , \cm)$, with 
$(A,+,\circ,\lambda )$  a left semitruss, so that $a\circ b =a+\lambda_a (b)$ and $\cm:A \rightarrow \Map (A,A):a\mapsto \cm_a$ is a map that determines the addition:
$a+b=b+\cm_b (a)$. Furthermore $\lambda$ is a semigroup homomorphism $(A,\circ) \rightarrow \Aut (A,+)$, $\cm$ is a semigroup anti-homomorphism $(A,+) \rightarrow \End (A,+)$, and
    $\cm_{\lambda_a(d)}\lambda_a(b) = \lambda_a \cm_d(b)$, for all $a,b\in A$.
The map $\lambda$ is called the $\lambda$-map of the semitruss and the map 
$\cm$ is called the $\cm$-map of the \ourstructure.

We give a brief outline of the paper and mention parts of some of the main results.
In Section~\ref{sec_LNDsAndYBsemitrusses},
we set up the machinery of these \ourstructures and we show 
that this object determines all left non-degenerate solutions. In particular, in Corollary~\ref{CorEpiGraded} we show the following result.

\begin{thmx}
Left non-degenerate solutions are restrictions of solutions associated to  \ourstructures.
\end{thmx}

 Hence, the study of such solutions is being reduced to the study of \ourstructures. 
It turns out that a \ourstructure is an epimorphic image of a structure semigroup $S(X,r)=M(X,r)\setminus \{ 1\}$ (considered as a \ourstructure)  of a left non-degenerate solution, where $1$ is the identity of $M$.
Further it is shown that all associative structures associated to 
left non-degenerate solutions that appear in the literature are special cases and correspond to \ourstructures with additional assumptions.

In Section \ref{sec_bijective},  we characterize when the solution $(A,r_A) $ associated to a \ourstructure is bijective. 
In particular, we obtain the following result (Theorem~\ref{maintheorembijective}).

\begin{thmx}
If $(X,r)$ is a finite left non-degenerate solution, then  $r$ is bijective
if and only if $(X,r)$ is right non-degenerate.
\end{thmx}

That the right non-degeneracy is a necessary condition  already has been shown by Castelli, Catino and Stefanelli in \cite{CaCaSt21}. Our proof also clarifies which  data  of the tuple defining a \ourstructure   controls specific properties of the associated  solution.

In order to study solutions one would like techniques to reduce solutions to solutions of smaller size. For bijective non-degenerate solutions this for example has been successfully done via the retract of a solution (of a skew left brace); to do so one makes fundamentally use of the fact there are two group operations involved.  In Section \ref{sec_retraction}, 
we show that for non-degenerate solutions this can be done in general (Theorem~\ref{retractassoc} and Corollary~\ref{permutationyb}).

\begin{thmx}
Let $(A,+,\circ, \lambda, \cm)$ be a non-degenerate \emph{\ourstructure}. Then 
$ \mathcal{G}(A)=\{ f(a)= (\cm_a,\lambda_a, \rho_a)\mid a\in A\}$ is a \ourstructure
for the operations
$$f(a)\circ f(b)= (\cm_{a\circ b},\lambda_{a\circ b},\rho_{a\circ b}), \quad
f(a)+f(b) = (\cm_{a+b}, \lambda_{a+b},\rho_{a+b}),$$   $\lambda$-map
 $$\lambda_{f(a)}f(b) =(\cm_{\lambda_a(b)},\lambda_{\lambda_{a}(b)},\rho_{\lambda_{a}(b)})=f(\lambda_a (b)),$$
 and $\cm$-map
  $$\cm_{f(b)}f(a) = (\cm_{\cm_{b}(a)}, \lambda_{\cm_{b}(a)}, \rho_{\cm_{b}(a)})=f(\cm_b(a)).$$ 
The map $f: A\rightarrow \mathcal{G}(A): a\mapsto f(a)$ is a \ourstructure epimorphism
and  the associated solution of $\mathcal{G}(A)$ is non-degenerate.

Furthermore,  $(\mathcal{G}(A),\circ)$ is a cancellative monoid (and thus also $(\mathcal{G}(A),+)$ is left cancellative) that satisfies the left and right Ore condition.
\end{thmx}

The \ourstructure  $ \mathcal{G}(A)$ is thus  ``close'' to being  a left cancellative semi-brace as introduced by Catino, Colazzo and Stefanelli \cite{CCS2017}. If this multiplicative semigroup is a group then it turns out that $ \mathcal{G}(A)$  is a skew left brace and that the associated solution is bijective. In particular, it follows that if $(X,r)$ is a non-degenerate solution so that
for every $a\in M(X,r)$ there exists $b\in M(X,r)$ such that
$\lambda_a \lambda_b =\id$ and   $\rho_b \rho_a =\id $ (for example of all $\lambda_a$ and $\rho_b$ are of finite order) then $r$ is bijective.
This is another application of the machinery of \ourstructures.

Finally, in Section \ref{sec_algebraic}, we study the additive and multiplicative structure of  a \ourstructure $A$ that is an epimorphic image of the structure monoid $M(X,r)$ of a finite left non-degenerate solution $(X,r)$. To do so  we exploit the fact that in the finite semigroup $\mathcal{C}(A):=\{ \cm_a \mid a\in A\}$
every left ideal is a two-sided ideal. It turns out that $(A,+)$ is a finite ``module'' over a subsemigroup that satisfies the ascending chain condition on left ideals and it has an ideal chain of which the Rees factors are finite unions of commutative semigroups.
As a consequence we obtain the following results (Theorem~\ref{ThmNoetherianPI} and Corollary~\ref{corollary:pinoetherian}).

\begin{thmx}
 Let $(X,r)$ be a finite left non-degenerate solution. Let $K$ be a field.
Then the algebra $K[(M(X,r),+)]$  is a left Noetherian PI-algebra of finite Gelfand-Kirillov dimension bounded by $|X|$. 
 
 If, furthermore,  the diagonal map $\mathfrak{q}:X\rightarrow X:x\mapsto \lambda_x^{-1}(x)$ is bijective    then also
 the structure algebra $K[(M(X,r),\circ)]$ is a connected  graded left Noetherian representable algebra.
\end{thmx}

Note that $K[M(X,r)]$ is a $K$-algebra defined by homogeneous quadratic relations. Such algebras  have received a lot of interest in recent years in several contexts, such as arithmetical orders, 
 classification of four dimensional noncommutative projective surfaces, Artin-Shelter regular algebras, regular languages and constructions of Noetherian algebras
(see for example \cite{CedoOkn2012,GIQuad,JOBook,JVC,YZ}). In case $r$ is bijective, Theorem D has been proven in \cite{JeKuVA19,JeKuVA19Cor}. It furthermore   was shown that in this case  $K[M(X,r)]$ is a prime algebra (or equivalently a domain) if and only if $r$ is involutive. Hence showing that the algebraic structure of the algebra $K[M(X,r)]$ determines properties of the solution. For involutive finite solutions  Gateva-Ivanova and Van den Bergh showed that homologically $K[M(X,r)]$ shares many properties with polynomial algebras in commuting variables. This rich algebraic structure  of the algebra $K[M(X,r)]$ 
gives perspectives for developing and applying a variety of algebraic methods in the context of solutions of the Yang-Baxter equation.

With a \ourstructure $(A,+,\circ, \lambda , \cm)$ one associates another relevant solution, called the associated derived solution and which is denoted $(A,s_A)$. It is the associated solution of the \ourstructure  $(A,+,+, \iota, \cm)$, where $\iota$ is the map $A\to \Aut(A,+)$ defined by $\iota_a=\id_A$. It is this structure that is fundamental in the proof of Theorem D.

We end the paper with describing the algebraic structure of \ourstructures for which $\mathcal{C}(A)$ is finite and has a principal ideal chain of length one, that is $\mathcal{C}(A)$ is a left simple semigroup. We show that their associated derived solutions are determined by finitely many bijective non-degenerate solutions.  In particular we prove the following result
(see Theorem~\ref{descrleftsimple}).

 \begin{thmx}
 Let $A=(A,+,\circ, \lambda , \cm)$  be a \ourstructure. Assume $\mathcal{C}=\mathcal{C}(A)$ is a finite left simple semigroup, let $E=E(\mathcal{C})$ be its set of idempotents and $G_e$ 
 the maximal subgroup of $\mathcal{C}$ containing $e \in E$. Then, $A=\cup_{e\in E} A_e$, a disjoint union of left ideals $A_e=\cm^{-1}(G_e)$, so that $s_A$ restricts to a bijective non-degenerate solution on $G_e(A_e)$
 and, for $a_e\in A_e$, $a_f\in A_f$, we have $s_A(a_e,a_f) =( a_f, \cm_{f(a_f)}(f(a_e)))$, 
 i.e. the derived solution is determined by bijective non-degenerate solutions. 
 \end{thmx}

\section{Left non-degenerate solutions and \ourstructures}\label{sec_LNDsAndYBsemitrusses}

In order to study the algebraic essence of several associative structures associated to solutions,
Brzezi\'nski in \cite{Br18} introduced  the notion of  a left semitruss.
A left semitruss is a quadruple $(A,+,\circ,\lambda )$ such that $(A,+)$ and $(A,\circ)$ are non-empty  semigroups and $\lambda \colon
A\longrightarrow \Map(A,A): a \mapsto \lambda_a$ is a function 
such that
\begin{equation}\label{eq:leftsemitrussrelation}
    a\circ (b+ c)= (a\circ b)+ \lambda_a (c),
\end{equation}
for all $a,b,c\in A$. 
One calls $(A,+)$ the additive semigroup and  $(A,\circ)$ the multiplicative semigroup of the left  semitruss and the relation \eqref{eq:leftsemitrussrelation} is called the left semitruss identity.

Since there is no natural construction to associate a (left non-degenerate)
solution to an arbitrary  semitruss, in \cite{Br18,CVA21} special subclasses of left
semi-trusses are considered that yield (left non-degenerate) solutions, 
but not every (left non-degenerate) solution can be obtained in this way. 
In order to resolve this for all left non-degenerate solutions,
we  define the following  subclass. Strictly speaking these should be called
left non-degenerate YB-left semitrusses, but for simplicity reasons 
we simply will call them
\ourstructures. In Subsection~\ref{sec_StructureMonoid} we will then show 
that they determine all left non-degenerate solutions.

\begin{definition}\label{def:qsemitruss}
A tuple $(A,+,\circ, \lambda , \cm)$ is said to be a \emph{\ourstructure} if  $(A,+,\circ,\lambda )$ is a left semitruss
and  $\cm:A \rightarrow \Map (A,A):a\mapsto \cm_a$ is map such that, for each $a,b,d\in A$, the following conditions are satisfied:
 \begin{align}
      &\lambda_a \in \Aut (A,+) \mbox{ and } \lambda_a \lambda_b =\lambda_{a\circ b}\label{trusslambda},
   \\
   &a+\lambda_a(b) = a\circ b \label{eq:sumcirc},\\
    &a+b=b+\cm_b(a) \label{eq:csum},\\
    & \cm_a \in \End (A,+) \mbox{ and } \cm_{a+b} =\cm_{b}\cm_{a}\label{eq:condc},  
    \\
    &\cm_{\lambda_a(d)}\lambda_a(b) = \lambda_a \cm_d(b). \label{eq:clambda}
 \end{align}
The map   $\lambda: (A,\circ) \rightarrow \Aut (A,+): a \mapsto \lambda_a$ is thus a semigroup homomorphism, it is called the \emph{$\lambda$-map} of the semitruss and the map 
$\cm: (A,+) \rightarrow \End(A,+): a \mapsto \cm_a$ is a semigroup anti-homomorphism, called the \emph{$\cm$-map} of the \ourstructure.
\end{definition}

If $(A,+)$ is a semigroup and $\lambda: A \rightarrow \End (A,+): a\mapsto \lambda_a$ is a mapping so that $\lambda_{a} \lambda_b  =\lambda_{a+\lambda_a (b)}$, for all $a,b\in A$, then $(A,\circ)$ is a semigroup for the operation $a\circ b =a+\lambda_a (b)$
and  $A$ is a left semitruss. Clearly  $(A,\circ) \rightarrow (A,+)\rtimes \End (A,+): a \mapsto (a,\lambda_a) $ is a semigroup monomorphism. 
Of course one can transfer the addition on $A$ to the image $\{ (a,\lambda_a )\mid a\in A\}$, i.e. $(a,\lambda_a) + (b,\lambda_b) := (a+b, \lambda_{a+b})$. Hence, if $(A,+)$ is a semigroup then subsemigroups of the type  $\{ (a,\lambda_a) \mid a\in A\}$ of $(A,+)\rtimes \End (A,+)$ yield examples of left semitrusses
and one has the stronger condition; $a\circ b =a+\lambda_a (b)$ (here we identify $a$ with $(a,\lambda_a)$). Note that these have a similar flavour to regular subgroups of the holomorph of a group.

Examples of \ourstructures are regular subsemigroups of $(A,+)\rtimes \Aut (A,+)$ with $(A,+)$  
an abelian semigroup (here each map $\cm_a =\id$). Specific examples of   such semigroups are semigroups of I-type, studied by Gateva-Ivannova and Van den Bergh 
\cite{GIVdB98} and, more general,  semigroups of IG-type, studied by Goffa, Jespers and Okni\'nski in \cite{GJ,GJO}.  
Here $(A,+)$ is  a 
free abelian semigroup, respectively an arbitrary   cancellative abelian semigroup. It is shown in \cite{GIVdB98} that the former determine the involutive non-degenerate solutions; the latter have been studied in \cite{GJ,GJO} only from a pure algebraic point of view.

Now, we focus on the case in which the additive and the multiplicative structure of a \ourstructure are monoids. Firstly, we prove that they share the same identity, say 1, then show that $\lambda_1$ and $\cm_1$ are the identity map and $\lambda_a(1)=1$ (for any $a\in A$). Later, we introduce the notion of unital \ourstructure requiring in addition that $ \cm_a(1)=1$ (for any $a\in A$). In Subsection~\ref{sec_StructureMonoid} we show that structure monoids of solutions are examples of unital \ourstructure.
\begin{lemma}\label{lemma:unital}
    Let $(A,+,\circ,\lambda,\cm)$ be a \ourstructure and assume $(A,+)$ and $(A,\circ)$ are monoids.  
    Then, condition \eqref{eq:sumcirc} is equivalent with  $(A,+)$ and $(A,\circ)$ sharing  the same identity, say $1$. Moreover, $\lambda_1=\cm_1=\id$ and $\lambda_a(1)=1$, for any $a\in A$.
    \begin{proof}
        Assume that $(A,+,\circ,\lambda,\cm)$ is a \ourstructure such that $(A,+)$ and $(A,\circ)$ are monoids.
        Let $0$ denote the identity of $(A,+)$ and let $1$ denote the identity of $(A,\circ)$. By \eqref{trusslambda},  $\lambda_1=\lambda_{1\circ 1} = \lambda_1\lambda_1$ and thus, because $\lambda_1$ is bijective, $\lambda_1=\id$. Hence,  by \eqref{eq:sumcirc}, $0 = 1\circ 0 = 1+\lambda_1(0) = 1+0 = 1$.
        Conversely, assume that $(A,\circ)$ and $(A,+)$ share the same identity, say $1$. It then follows from the left semitruss identity  that $a+\lambda_a(b) = a\circ 1 +\lambda_a(b) = a \circ (1+b) = a\circ b$, as desired.
        
        Moreover, for any $a\in A$, because of \eqref{eq:csum} we have $ a = a+1 = 1+ \cm_1(a) = \cm_1(a)$, i.e. $\cm_1 =\id$. Also, for any $a,b\in A$, we have $b + \lambda_a(1) = \lambda_a(\lambda^{-1}_a(b))  + \lambda_a(1)=  \lambda_a(\lambda^{-1}_a(b)+1) = \lambda_a(\lambda^{-1}_a(b))=b$. This shows that $\lambda_a(1)$ is a right identity and thus $\lambda_a(1) = 1$.
    \end{proof}
\end{lemma}

\begin{definition}
    A \ourstructure $(A,+,\circ,\lambda, \cm)$ is said to be \emph{unital}, if $(A,+)$ and $(A,\circ)$ are monoids and, for any $a\in A$,
    \begin{align}\label{eq:condunital}
        \cm_a(1)=1,
    \end{align}
    where $1$ denotes the identity in $(A,+)$ (and also in $(A,\circ)$).
\end{definition}

Note that if $(A,+)$ is left cancellative monoid, with identity say $1$, then \eqref{eq:condunital} follows from the definition of \ourstructure. Indeed $a+1=a=1+a\overset{\eqref{eq:csum}}{=}a+\cm_a(1)$, and by left cancellativity, we get $1 =\cm_a(1)$. 

However, the above does not hold in  general.
Indeed, consider for example  the multiplicative monoid  $A=\{ 0,1\}$. This can be made into a \ourstructure by defining $(A,+)=(A,\circ)$ and 
$\lambda_1=\lambda_0=\id_A$, $\cm_1=\id$, and $\cm_0(x)=0$, for $x\in A$. 
Note that $A$ is not a unital \ourstructure since \eqref{eq:condunital} does not hold in $A$ because by definition $\cm_0(1) = 0\neq 1$.

Note that condition \eqref{eq:csum} implies, for all $a,b\in A$,
\begin{align}\label{eq:condc1}
   & \cm_a \cm_b = \cm_{\cm_a(b)} \cm_a.
\end{align}
Indeed, by \eqref{eq:condc} and \eqref{eq:csum},  $\cm_{a}\cm_{b}= \cm_{b+a}=\cm_{a+  \cm_a(b)} = \cm_{\cm_a(b)}\cm_a$.

Condition \eqref{eq:sumcirc} links the multiplicative with the additive structure. Condition \eqref{eq:csum} implies that $A+b\subseteq b+A$, so right ideals in the semigroup $(A,+)$ are two-sided ideals. If, furthermore, each $\cm_b$ is bijective then $A+b=b+A$, i.e. every element of $A$ is a normal element. As shown in \cite{JeKuVA19} the latter is an essential property to describe the structure monoids of bijective non-degenerate solutions. Condition \eqref{eq:condc} can be avoided 
to associate a left non-degenerate solution to the considered \ourstructures (see Proposition~\ref{prop:semitruss}). However, to deal
with an arbitrary left non-degenerate solution we will show in Proposition~\ref{SolFromSemiTruss}  that, without loss of generality, we may assume that \eqref{eq:condc} holds.

Because of the equations \eqref{eq:YBE1}, \eqref{eq:YBE2} and \eqref{eq:YBE3}, equation \eqref{eq:condc1} means that the map 
$$s_A : A^2 \rightarrow A^2 : (a,b)\mapsto (b,\cm_b(a)),$$ is a solution. We call it the {\it  associated 
derived solution of  the \ourstructure $A$}
(in the terminology of \cite{Le18} one says that  $(A,\triangleleft)$ is  a shelf, where
$a\triangleleft b =\cm_b(a)$, and equality \eqref{eq:condc1} says that $\triangleleft$ is self distributive, 
i.e $(a \triangleleft b) \triangleleft  d = (a\triangleleft d) \triangleleft (b\triangleleft d)$). 
Clearly $s_A$ is bijective if and only if all $\cm_a$ are bijective, that is $s_A$ is a non-degenerate bijective solution 
(or in the language of \cite{Kamada}, $(A,\triangleleft )$ is a rack).
Note that $s_A$ is the associated solution of the \ourstructure $(A,+,+,\iota , \cm)$, where $\iota$ is the map $A\rightarrow \Aut(A,+)$ defined by $\iota_a =\id_{A}$.

Recall (see for example \cite{CedoAgta}) that a mapping 
 $$r:X^2 \rightarrow X^2: (x,y)\mapsto (\lambda_x(y), \rho_y(x)),$$
defines a solution if and only if, for all $a,b,c\in X$, 
\begin{align} 
    \lambda_a\lambda_b = \lambda_{\lambda_a(b)}\lambda_{\rho_b(a)},\label{eq:YBE1}\\
    \lambda_{\rho_{\lambda_b(c)}(a)}\rho_c(b) = \rho_{\lambda_{\rho_b(a)}(c)}\lambda_a(b),\label{eq:YBE2}\\
    \rho_c\rho_b = \rho_{\rho_c(b)}\rho_{\lambda_b(c)}.\label{eq:YBE3}\
\end{align}

The meaning of the other requirements in Definition~\ref{def:qsemitruss} is clarified in the following proposition. For this we  introduce the so-called $\rho$-map of a left semitruss
 $(A,+,\circ , \lambda)$ for which there exists a map  $\cm :A \rightarrow \Map (A,A)$ so that $a+b=b+\cm_b(a)$, for all $a,b\in A$ (i.e. condition \eqref{eq:csum} holds) and  $\lambda: (A,\circ) \rightarrow \Aut (A,+): a \mapsto \lambda_a $
is a homomorphism (for example $A$ is a \ourstructure).
Then 
  $$\rho: A \rightarrow \Map (A,A): a\mapsto \rho_a,$$
with, for all $a,b\in A$,
 $$\rho_b (a)=\lambda^{-1}_{\lambda_{a}(b)}\cm_{\lambda_a(b)}(a),$$
is called its \emph{$\rho$-map}.

We now prove a generalization of \cite[Proposition 5.4]{CJV2020}. The proof follows the same main idea as in \cite[Proposition 2.2]{JeKuVA19} and \cite[Theorem 2.3]{So00}.

\begin{proposition}\label{prop:semitruss}
Let $(A,+,\circ , \lambda)$ be a left semitruss and let  $\cm :A \rightarrow \Map (A,A)$ be a map so that $a+b=b+\cm_b(a)$, for all $a,b\in A$ (i.e. condition \eqref{eq:csum} holds).
Assume that $\lambda: (A,\circ) \rightarrow \Aut (A,+): a \mapsto \lambda_a $
is a homomorphism and condition \eqref{eq:sumcirc} is
satisfied.
Then 
  \begin{align}\label{eq:solsemitruss}
  r_A:A^2 \rightarrow A^2: (a,b) \mapsto (\lambda_a (b),\rho_b (a)),
  \end{align}
is a (left non-degenerate) solution if and only if
$\cm_{\lambda_a(b)} \lambda_a = \lambda_a \cm_b$, for all $a,b\in A$ (i.e. condition \eqref{eq:clambda} holds). 
 
Furthermore, if $(A,+,\circ , \lambda ,\cm)$ is a \ourstructure, then $r_A =\varphi^{-1} s_A \varphi$ where $\varphi:A^2 \rightarrow A^2: (a,b)\mapsto (a,\lambda_a(b))$. In particular, $r_A$ is bijective if and only if $s_A$ is bijective, 
or equivalently $s_A$ is (right) non-degenerate (i.e. all maps $\cm_a$ are bijective).
\end{proposition}
\begin{proof}
Let $J: A^3\to A^3$ be the map defined by $J(a, b, d) = (a, \lambda_a(b), \lambda_a\lambda_b(d))$. Clearly $J$ is bijective and $J^{-1}(a, b, d) = (a, \lambda_a^{-1}(b), \lambda^{-1}_{\lambda^{-1}_a(b)}\lambda^{-1}_a(d))$, for all $a, b, d \in A$.  We have
\begin{align*}
    J^{-1}(s_A \times \id) J(a,b,d) 
    &=J^{-1}(s_A\times  \id) (a,\lambda_a(b),\lambda_a\lambda_b(d))\\
    &=J^{-1}(\lambda_a(b), \cm_{\lambda_a(b)}(a), \lambda_a\lambda_b(d)) \\
    &=(\lambda_a(b), \lambda^{-1}_{\lambda_a(b)}\cm_{\lambda_a(b)}(a), \lambda^{-1}_{\lambda^{-1}_{\lambda_a(b)}\cm_{\lambda_a(b)}(a)}\lambda^{-1}_{\lambda_a(b)}\lambda_a\lambda_b(d))\\
    &=(\lambda_a(b), \lambda^{-1}_{\lambda_a(b)}\cm_{\lambda_a(b)}(a),
    \lambda^{-1}_{\lambda_a(b)\circ\lambda^{-1}_{\lambda_a(b)}\cm_{\lambda_a(b)}(a)}\lambda_{a\circ b}(d)) \\
    &=(\lambda_a(b), \lambda^{-1}_{\lambda_a(b)}\cm_{\lambda_a(b)}(a),\lambda^{-1}_{a\circ b}\lambda_{a\circ b}(d))\\
    &=(\lambda_a(b), \lambda^{-1}_{\lambda_a(b)}\cm_{\lambda_a(b)}(a), d)\\
    &=(\lambda_a (b) , \rho_b(a), d),
\end{align*}
since $\lambda_a(b)\circ \lambda^{-1}_{\lambda_a(b)}\cm_{\lambda_a(b)}(a) =\lambda_a(b)+ \cm_{\lambda_a(b)}(a)= a+ \lambda_a(b)=a\circ b$. 
Hence, $J^{-1}(s_A \times \id )J = r_A \times \id $. Moreover,  if condition \eqref{eq:clambda} holds then
\begin{align*} 
    J^{-1}(\id \times s_A) J(a,b,d)
    &=J^{-1}(\id \times s_A)(a,\lambda_a(b),\lambda_a\lambda_b(d)) \\
    &=J^{-1}(a,\lambda_a\lambda_b(d), \cm_{\lambda_a\lambda_b(d)}\lambda_a(b))\\
    &=(a, \lambda^{-1}_a\lambda_a\lambda_b(d), \lambda^{-1}_{\lambda^{-1}_a\lambda_a\lambda_b(d)}\lambda^{-1}_a \cm_{\lambda_a\lambda_b(d)}\lambda_a(b))\\
    &=(a, \lambda_b(d), \lambda^{-1}_{\lambda_b(d)}\lambda^{-1}_a\lambda_a \cm_{\lambda_b(d)}(b))\\
    &=(a, \lambda_b(d), \lambda^{-1}_{\lambda_b(d)} \cm_{\lambda_b(d)}(b))\\
    &=(a, \lambda_b (d), \rho_d(b)),
\end{align*}
yields $J^{-1}(\id \times s_A) J=\id \times r_A$. Because  $(A,s_A)$ is a solution (as mentioned earlier the condition \eqref{eq:clambda} is not needed for this) we get that  $(A,r_A)$ is a solution.
This proves the necessity of condition \eqref{eq:clambda}.

Conversely, assume that $(A,r_A)$ is a solution.
As  $\cm_b(a)=\lambda_{b}\rho_{\lambda^{-1}_a(b)}(a)$ we get that
\begin{align*}
    \cm_{\lambda_a(d)}(\lambda_a(b)) &= \lambda_{\lambda_a(d)}\rho_{\lambda^{-1}_{\lambda_a(b)}\lambda_a(d)}\lambda_a(b)\\
    &=\lambda_{\lambda_a(d)}\rho_{\lambda_{\rho_b(a)}\lambda^{-1}_b(d)}\lambda_a(b) &\mbox{by \eqref{eq:YBE1}}\\
    &= \lambda_{\lambda_a(d)} \lambda_{\rho_{\lambda_{b}\lambda^{-1}_b(d)}(a)}\rho_{\lambda^{-1}_b(d)}(b) &\mbox{by \eqref{eq:YBE2}}\\
    &=\lambda_{\lambda_a(d)}\lambda_{\rho_{d}(a)}\rho_{\lambda^{-1}_b(d)}(b)\\
    &= \lambda_a\lambda_d\rho_{\lambda^{-1}_b(d)}(b) &\mbox{by \eqref{eq:YBE1}}\\
    &=\lambda_a \cm_d(b).
\end{align*}
Hence, $\lambda_a \cm_d(b)=\cm_{\lambda_a(d)}(\lambda_a(b))$.

To finish the proof it is enough to observe that $s_A=\varphi r_A \varphi^{-1}$.
\end{proof}

The previous proposition states that a left non-degenerate solution $r_A$ can be determined by a conjugate of the derived solution $s_A$, by a mapping say $\varphi$.
However, $\varphi$ is not an arbitrary permutation of $A^2$. In the case of bijective solutions it has to be a so-called Drinfeld twist on solutions, see for example \cite{Ghobadi}. One can extend the latter definition to arbitrary left non-degenerate solutions and with the same strategy as in \cite[Example~2.1]{Ghobadi} one proves that $r_A =\varphi^{-1} s_A \varphi$ for some permutation $\varphi$ of $A^{2}$ so that $(\varphi, \Psi, \Phi)$ is a Drinfeld twist on the solution $(A,r_A)$.

For a  \ourstructure $A$ we call $r_A$ the {\it associated solution}. It is said to be the \emph{trivial solution} if $r_A(a,b)=(b,a)$ for all $a,b\in A$.
It turns out that if an associated solution is bijective then its inverse is the associated solution of a well-described associated \ourstructure.

For a semigroup $(A,+)$ we denote by $(A,+^{op})$ its opposite semigroup, i.e. $a+^{op}b= b+a$, for $a,b\in A$.

\begin{proposition}\label{prop:opposite}
Let $(A,+,\circ, \lambda, \cm)$ be a \ourstructure with $\cm_a$ bijective for any $a\in A$, that is,  the associated solution $r_A$ is bijective.
Then, 
\begin{align*}
    r_A^{-1}(a,b)=(\cm^{-1}_a\lambda_a(b), \lambda^{-1}_{\cm^{-1}_a\lambda_a(b)}(a)),
\end{align*}
and  $(A,+^{op},\circ, \bar{\lambda},\bar{\cm})$ is a \ourstructure, with $\bar{\lambda}_a=\cm^{-1}_a\lambda_a$ and $\bar{\cm}_a=\cm^{-1}_a$,  and its  associated solution is $r_A^{-1}$. We call $(A,+^{op},\circ, \bar{\lambda},\bar{\cm})$ the \emph{opposite \ourstructure}.

In particular, for a \ourstructure $(A,+,\circ, \lambda, \cm)$, the associated solution $r_A$ is involutive if and only if the associated derived solution $s_A$ is trivial (that is, all maps $\cm_a =\id$).
\begin{proof}
Put $r=r_A$. So, 
$
    r(a,b)=(\lambda_a(b),\lambda^{-1}_{\lambda_a(b)}\cm_{\lambda_a(b)}(a)).
$
Moreover, put $r'(a,b)=(\cm^{-1}_a\lambda_a(b), \lambda^{-1}_{\cm^{-1}_a\lambda_a(b)}(a))$. Then
\begin{align*}
    &rr'(a,b)=r(\cm^{-1}_a\lambda_a(b), \lambda^{-1}_{\cm^{-1}_a\lambda_a(b)}(a)) \\
    &=(\lambda_{\cm^{-1}_a\lambda_a(b)}\lambda^{-1}_{\cm^{-1}_a\lambda_a(b)}(a),\lambda^{-1}_{\lambda_{\cm^{-1}_a\lambda_a(b)}\lambda^{-1}_{\cm^{-1}_a\lambda_a(b)}(a)}\cm_{\lambda_{\cm^{-1}_a\lambda_a(b)}\lambda^{-1}_{\cm^{-1}_a\lambda_a(b)}(a)}{\cm^{-1}_a\lambda_a(b)})\\
    &=(a, \lambda^{-1}_a\cm_a{\cm^{-1}_a\lambda_a(b)}) =(a,b),
\end{align*}
and
\begin{align*}
    &r'r(a,b) =r'(\lambda_a(b),\lambda^{-1}_{\lambda_a(b)}\cm_{\lambda_a(b)}(a))\\
    &=(\cm^{-1}_{\lambda_a(b)}\lambda_{\lambda_a(b)}\lambda^{-1}_{\lambda_a(b)}\cm_{\lambda_a(b)}(a), \lambda^{-1}_{\cm^{-1}_{\lambda_a(b)}\lambda_{\lambda_a(b)}\lambda^{-1}_{\lambda_a(b)}\cm_{\lambda_a(b)}(a)}{\lambda_a(b)})\\
    &=(a,b).
\end{align*}
So indeed $r'=r^{-1}$.

Now, assume that $(A,+^{op},\circ, \bar{\lambda},\bar{\cm})$, with $\bar{\lambda}_a=\cm^{-1}_a\lambda_a$ and $\bar{\cm}_a=\cm^{-1}_a$, is a \ourstructure. The associated solution to such a \ourstructure is given by
\begin{align*}
    \bar{r}(a,b)&=(\bar{\lambda}_a(b),\bar{\lambda}^{-1}_{\bar{\lambda}_a(b)}\bar{\cm}_{\bar{\lambda}_a(b)}(a))\\
    &=(\cm^{-1}_a\lambda_a(b), \lambda^{-1}_{\cm^{-1}_a\lambda_a(b)}\cm_{\cm^{-1}_a\lambda_a(b)}\cm^{-1}_{\cm^{-1}_a\lambda_a(b)}(a))\\
    &=(\cm^{-1}_a\lambda_a(b), \lambda^{-1}_{\cm^{-1}_a\lambda_a(b)}(a))\\
    &=r^{-1}(a,b).
\end{align*}

So it remains to prove that $(A,+^{op},\circ, \bar{\lambda},\bar{\cm})$ is a \ourstructure.
Clearly $\bar{\lambda}_a\in \Aut (A,+^{op})$ for each $a\in A$. That $(A,\circ ) \rightarrow \Aut (A,+^{op}): a\mapsto \bar{\lambda}_a$ is a homomorphism follows from the following:
\begin{align*}
    \bar{\lambda}_a\bar{\lambda}_b (d) 
    &=\cm_a^{-1}\lambda_a \cm_b^{-1}\lambda_b(d)\\
    &= \cm_a^{-1}\cm^{-1}_{\lambda_a(b)}\lambda_a\lambda_b(d) &\mbox{by \eqref{eq:clambda}}\\
    &= (\cm_{\lambda_a(b)}\cm_a)^{-1}\lambda_{a\circ b}(d) &\mbox{by \eqref{trusslambda}}\\
    &=\cm_{a+\lambda_a(b)}^{-1}\lambda_{a\circ b}(d) &\mbox{by \eqref{eq:condc}}\\
    &=\cm_{a\circ b}^{-1}\lambda_{a\circ b}(d) &\mbox{by \eqref{eq:sumcirc}}\\
    &= \bar{\lambda}_{a\circ b}(d).
\end{align*}
Also, for each $a\in A$, it is clear that $\bar{\cm}_a \in \Aut (A,+^{op})$ and thus also that 
$(A,+^{op}) \rightarrow (\text{Map} (A,A),\circ): a \mapsto \bar{\cm}_a$ is  semigroup anti-homomorphism.

Since  $(A,+,\circ,\lambda, \cm)$ satisfies \eqref{eq:csum}, we have that  $a+b = b+\cm_b(a)$ for all $a,b\in A$. Hence, 
 also $\cm^{-1}_b(a)+b = b+a$, or equivalently, $b+^{op}\cm^{-1}_b(a) = a +^{op}b$,
i.e. \eqref{eq:csum} holds for $(A,+^{op},\circ, \bar{\lambda}, \bar{\cm})$.
Since $(A,+,\circ,\lambda, \cm)$ satisfies equation \eqref{eq:sumcirc}, we obtain
\begin{align*}
    a +^{op} \bar{\lambda}_a(b) &= a+^{op}\cm^{-1}_a\lambda_a(b)\\
    &=a+^{op}\bar{\cm}_a\lambda_a(b)\\
    &= \lambda_a(b)+^{op}a&\mbox{(by \eqref{eq:csum} for $(A,+^{op},\circ, \bar{\lambda}, \bar{\cm})$)}\\
    &=a+ \lambda_a(b) \\
    &=a\circ b.
\end{align*}
So, \eqref{eq:sumcirc} holds for $(A,+^{op},\circ, \bar{\lambda}, \bar{\cm})$. 
Finally, \eqref{eq:clambda} yields
$
    \cm_{\lambda_a(b)}\lambda_a = \lambda_a \cm_b $
and thus $\lambda_a \cm^{-1}_b=\cm^{-1}_{\lambda_a(b)}\lambda_a$. Hence,
$\cm^{-1}_a\lambda_a \cm^{-1}_b = \cm^{-1}_a\cm^{-1}_{\lambda_a(b)}\lambda_a$.
Because of \eqref{eq:condc1} we know that
$\cm^{-1}_a\cm^{-1}_{\lambda_a(b)} = \cm^{-1}_{\cm^{-1}_{a}\lambda_a(b)}\cm^{-1}_{a}$. Hence, 
$\bar{\lambda}_a\cm^{-1}_b     = \cm^{-1}_{\bar{\lambda}_a(b)}\bar{\lambda}_a $.
So, $(A,+^{op},\circ, \bar{\lambda}, \bar{\cm})$ is a \ourstructure.
\end{proof}
\end{proposition}

To deal with involutive solutions that are not  necessarily non-degenerate, Rump introduced recently in \cite{Ru22} the algebraic structure called   a  strong semibrace. A strong semibrace is a monoid $(A,\circ)$ with neutral element $0$ and 
an additional binary operation $\cdot$ such that the following are satisfied for all $a,b,c\in A$:
\begin{align*}
    &(a\circ b)\cdot c = a\cdot (b \cdot c),\qquad &0 \cdot a = a,\\
    &a \cdot (b\circ b) = ((c \cdot a)\cdot b) \circ (a\cdot c), \qquad &a\cdot 0 =0,\\
    &(a\cdot b)\circ a = (b\cdot a) \circ b.
\end{align*}
If in addition we assume that, for any $a\in A$, the left multiplication map  $\ell_a: A\to A: b \mapsto a\cdot b$ is bijective 
then $(A,\cdot)$ is a cycle set (with associated left non-degenerate solution $r(a,b)=(\ell^{-1}_a(b), \ell_{\ell^{-1}_a(b)}(a))$).
One can note that a strong semibrace with $(A,\cdot)$ a cycle set is an instance of a \ourstructure. Indeed $(A,+,\circ^{op},\lambda,\iota)$ is a \ourstructure, where $a+b = (a \cdot b)\circ a$, the $\lambda$-map is the inverse of the left multiplication and $\iota_a = \id$ (so $(A,+)$ is abelian).

The $\lambda$ and $\rho$ maps provide relations in \ourstructures.

\begin{proposition}\label{epitruss}
Let $(A,+,\circ, \lambda, \cm)$ be a \ourstructure.
Then for all $a,b\in A$ and for the $\lambda$-map and $\rho$-map we have
\begin{align}
    &a\circ b = \lambda_a (b) \circ \rho_b(a), \label{eq:truss1}\\
    &\lambda_b(a\circ d) = \lambda_b (a) \circ \lambda_{\rho_a(b)}(d), \label{eq:truss2}\\
    &\rho_b (d\circ a) =\rho_{\lambda_a (b)}(d) \circ \rho_b (a). \label{eq:truss3}
\end{align}
\end{proposition}
\begin{proof}
Let $a,b,d\in A$. Then,
\begin{align*}
    \lambda_a(b)\circ \rho_b(a) &= \lambda_a(b)\circ \lambda^{-1}_{\lambda_{a}(b)}\cm_{\lambda_a(b)}(a)\\
    &=
    \lambda_a(b)+\cm_{\lambda_a(b)}(a)= a+ \lambda_a(b)= a\circ b.
\end{align*}
Hence, we get 
\begin{align*}
\lambda_b (a\circ d) &= \lambda_b (a+\lambda_a (d)) = \lambda_b (a) + \lambda_b \lambda_a (d)\\
 &= \lambda_b (a) + \lambda_{b\circ a} (d) = \lambda_b (a) +\lambda_{\lambda_b (a) \circ \rho_a (b)}(d)\\
 &= \lambda_b (a) + \lambda_{\lambda_b (a)} \lambda_{\rho_a (b)}(d)\\
 &= \lambda_b(a) \circ \lambda_{\rho_a(b)}(d).
\end{align*}
To prove \eqref{eq:truss3} we first claim that
\begin{align}
\rho_{\lambda^{-1}_{\lambda_a^{-1} (b)} \lambda_{a}^{-1}(d)} \lambda_{a}^{-1} (b) = \lambda^{-1}_{\rho_{\lambda^{-1}_a (d)}(a)} \rho_{\lambda_{b}^{-1} (d)}(b). \label{eq:lem4}
\end{align}
Indeed by setting in \eqref{eq:YBE2} $c=\lambda^{-1}_y(t)$, $b=y$ and $a=z$, we get
 \begin{align*}
     \lambda_{\rho_{\lambda_y(\lambda^{-1}_y(t))}(z)}\rho_{\lambda^{-1}_y(t)}(y) = \rho_{\lambda_{\rho_y(z)}\lambda^{-1}_y(t)}\lambda_z(y).
 \end{align*}
 With $t=\lambda^{-1}_{z}(u)$ and $y=\lambda^{-1}_{z}(v)$ and because of \eqref{eq:YBE1} this yields
 \begin{align*}
     \lambda_{\rho_{\lambda^{-1}_{z}(u)}(z)}\rho_{\lambda^{-1}_{\lambda^{-1}_{z}(v)}\lambda^{-1}_{z}(u)}\lambda^{-1}_{z}(v) &= \rho_{\lambda_{\rho_{\lambda^{-1}_{z}(v)}(z)}\lambda^{-1}_{\lambda^{-1}_{z}(v)}\lambda^{-1}_{z}(u)}\lambda_z\lambda^{-1}_{z}(v)\\
     &=\rho_{\lambda^{-1}_v(u)}(v),
 \end{align*}
as desired.

Moreover, we notice that
\begin{align*}
    \cm_a(b+d) 
    &= \lambda_a \rho_{\lambda^{-1}_{b+d}(a)}(b+d)
    = \lambda_a \rho_{\lambda^{-1}_{b\circ \lambda^{-1}_b(d)}(a)}(b\circ \lambda^{-1}_b(d)),
\end{align*}
and
\begin{align*}
    \cm_a(b) + \cm_a(d)&= \cm_a(b) \circ \lambda^{-1}_{\cm_a(b)}(\cm_a(d))\\
    &= \lambda_a( \rho_{\lambda^{-1}_b(a)}(b)) \circ \lambda^{-1}_{\lambda_a \rho_{\lambda^{-1}_b(a)}(b)}(
    \lambda_a(\rho_{\lambda^{-1}_d(a)}(d)))\\
    &=  \lambda_a( \rho_{\lambda^{-1}_b(a)}(b)) \circ \lambda_{\rho_{\rho_{\lambda^{-1}_b(a)}(b)}(a)}(
    \lambda^{-1}_{\rho_{\lambda^{-1}_b(a)}(b)}(\rho_{\lambda^{-1}_d(a)}(d))) \quad &\mbox{ by \eqref{eq:YBE1}}\\
    &= \lambda_a( \rho_{\lambda^{-1}_b(a)}(b)) \circ \lambda_{\rho_{\rho_{\lambda^{-1}_b(a)}(b)}(a)}(
    \rho_{\lambda^{-1}_{\lambda^{-1}_b(d)}\lambda^{-1}_b(a)}(\lambda^{-1}_b(d))) \quad &\mbox{ by \eqref{eq:lem4}} \\
    &= \lambda_a ( \rho_{\lambda^{-1}_b(a)}(b) \circ \rho_{\lambda^{-1}_{\lambda^{-1}_b(d)}\lambda^{-1}_b(a)}(\lambda^{-1}_b(d))  )
      \quad &\mbox{ by \eqref{eq:truss2}}.
\end{align*}
As  $\cm_a$ is an additive homomorphism and $\lambda_a$ is bijective, we thus obtain that 
\begin{align*}
     \rho_{\lambda^{-1}_{b\circ \lambda^{-1}_b(d)}(a)}(b\circ \lambda^{-1}_b(d))
    &= \rho_{\lambda^{-1}_b(a)}(b) \circ \rho_{\lambda^{-1}_{\lambda^{-1}_b(d)}\lambda^{-1}_b(a)}(\lambda^{-1}_b(d)).
\end{align*}
Put $e=\lambda^{-1}_b (d)$ and $x=\lambda^{-1}_{b\circ e}(a)$. Then the previous equality becomes, for arbitrary $x,b,e\in A$,
 $$\rho_x (b\circ e) =\rho_{\lambda_b^{-1} \lambda_{b\circ e} (x)} (b) \circ \rho_x (e) =\rho_{\lambda_e (x)}(b) \circ \rho_x (e),$$
as desired.
\end{proof}

The requirement that the $\cm$-map is an anti-homomorphism implies that the $\rho$-map also is an anti-homomorphism. 

\begin{lemma}\label{chomomorphism}
Let $(A,+,\circ, \lambda, \cm)$ be a \ourstructure with $\rho$-map  $\rho_{b}(a)= \lambda^{-1}_{\lambda_a(b)}\cm_{\lambda_a(b)}(a)$.
Then $\rho:(A,\circ)\to \Map(A,A): a\mapsto \rho_a$ is an anti-homomorphism.
\begin{proof} 
Let $a,b\in A$. By \eqref{eq:truss2} and because the $\cm$-map is an anti-homomorphism, we get for all $x\in A$ that
\begin{align*}
    \rho_{b\circ a} (x) &= \lambda^{-1}_{\lambda_x(b\circ a)}\cm_{\lambda_x(b\circ a)}(x)
    =\lambda^{-1}_{\lambda_x(b)\circ\lambda_{\rho_b(x)}(a)}\cm_{\lambda_x(b+\lambda_b(a))}(x)\\
    &=\lambda^{-1}_{\lambda_{\rho_b(x)}(a)}\lambda^{-1}_{\lambda_x(b)} \cm_{\lambda_x(b)+\lambda_x\lambda_b(a)}(x)\\
    &=\lambda^{-1}_{\lambda_{\rho_b(x)}(a)}\lambda^{-1}_{\lambda_x(b)}\cm_{\lambda_x\lambda_b(a)} \cm_{\lambda_x(b)}(x)\\
    &= \lambda^{-1}_{\lambda_{\rho_b(x)}(a)}\cm_{\lambda^{-1}_{\lambda_x(b)}\lambda_x\lambda_b(a)} \lambda^{-1}_{\lambda_x(b)}\cm_{\lambda_x(b)}(x)\\
    &= \lambda^{-1}_{\lambda_{\rho_b(x)}(a)}\cm_{\lambda_{\rho_{b}(x)}(a)}\rho_b(x)\\
    &=\rho_a\rho_b(x).
\end{align*}
   Hence, $\rho_{b\circ a}(x) =\rho_a\rho_b(x)$, as desired.
\end{proof}
\end{lemma}

In fact, if we do not assume that a  \ourstructure $(A,+,\circ,\lambda,\cm)$ satisfies $\cm_{a+b} = \cm_b \cm_a$ for all $a,b \in A$, then we can prove that this condition follows from the $\rho$-map being an anti-homomorphism. Indeed as  $\cm_b(a)=\lambda_{b}\rho_{\lambda^{-1}_a(b)}(a)$ we get that, for any $a,b,x \in A$, using \eqref{eq:YBE1}, \eqref{eq:YBE2} and that the $\rho$-map is an anti-homomorphism,
   \begin{align*}
       \cm_{a+b}(x)
       &= \lambda_{a+b} \rho_{\lambda^{-1}_x(a+b)}(x)
       \\ &= \lambda_a \lambda_{\lambda^{-1}_a(b)} \rho_{\lambda^{-1}_x(a) + \lambda^{-1}_x(b)}(x)
       \\ &= \lambda_b \lambda_{\rho_{\lambda^{-1}_a(b)}(a)} \rho_{\lambda^{-1}_x(a) \circ \lambda^{-1}_{\lambda^{-1}_x(a)} (\lambda^{-1}_x(b))}(x)
       \\ &= \lambda_b \lambda_{\rho_{\lambda^{-1}_a(b)}(a)} \rho_{\lambda^{-1}_{\lambda^{-1}_x(a)} (\lambda^{-1}_x(b))} \rho_{\lambda^{-1}_x(a)}(x)
       \\ &= \lambda_b \lambda_{\rho_{\lambda^{-1}_a(b)}(a)} \rho_{\lambda^{-1}_{\rho_{\lambda^{-1}_x(a)}(x)} (\lambda^{-1}_a(b))} \rho_{\lambda^{-1}_x(a)}(x)
       \\ &= \lambda_b \lambda_{\rho_{\lambda_{\rho_{\lambda^{-1}_x(a)}(x)} (\lambda^{-1}_{\rho_{\lambda^{-1}_x(a)}(x)} \lambda^{-1}_a(b))}(a)} \rho_{\lambda^{-1}_{\rho_{\lambda^{-1}_x(a)}(x)} \lambda^{-1}_a(b)} (\rho_{\lambda^{-1}_x(a)}(x))
       \\ &= \lambda_b \rho_{\lambda_{\rho_{\rho_{\lambda^{-1}_x(a)}(x)} (a)}(\lambda^{-1}_{\rho_{\lambda^{-1}_x(a)}(x)} \lambda^{-1}_a(b))} \lambda_{a} (\rho_{\lambda^{-1}_x(a)}(x))
       \\ &= \lambda_b \rho_{\lambda^{-1}_{\lambda_a\rho_{\lambda^{-1}_x(a)}(x)} \lambda_a \lambda^{-1}_a(b)} \lambda_{a} (\rho_{\lambda^{-1}_x(a)}(x))
       \\ &= \lambda_b \rho_{\lambda^{-1}_{\lambda_a\rho_{\lambda^{-1}_x(a)}(x)}(b)} \lambda_{a} (\rho_{\lambda^{-1}_x(a)}(x))
       \\ &= \cm_b (\lambda_a\rho_{\lambda^{-1}_x(a)}(x))
       \\ &= \cm_b \cm_a(x).
   \end{align*}
   Hence, the map $\cm:(A,+) \to \Map(A,A): a \mapsto \cm_a$ is an anti-homomorphism.

Obviously, a homomorphism of \ourstructures should be a mapping  that preserves the operations and is compatible with the $\lambda$-map and the $\cm$-map.

\begin{definition}\label{def:homomorphism}
A mapping $f:A\rightarrow B$ between \ourstructures $(A,+,\circ, \lambda, \cm)$ and $(B,+',\circ', \lambda', \cm')$ is said the be a \emph{homomorphism} if
\begin{enumerate}
    \item 
$f(a+b)=f(a)+'f(b)$, $f(a\circ b) =f(a) \circ' f(b)$, 
  \item $f(\lambda_a (b)) =\lambda'_{f(a)}(f(b))$ and $f(\cm_b(a))=\cm'_{f(b)}(f(a))$, for all $a,b\in A$.
 \end{enumerate}
\end{definition}

Note that a homomorphism $f$ of \ourstructures also is compatible with the inverses of the $\lambda$-maps, i.e.  
$\lambda'^{-1}_{f(a)} f = f\lambda^{-1}_a$.

\begin{lemma}\label{homomorphismqsemitruss}
Let $f:A\rightarrow B$ be a epimorphism of \ourstructures $(A,+,\circ, \lambda, \cm)$ and $(B,+',\circ', \lambda', \cm')$.
Then $\rho'_{f(a)}f(b) =f\rho_{a}(b)$ for all $a,b\in A$, i.e. the $\rho$-map is compatible with $f$.

If, for all $a\in A$, $\rho_a$ is surjective then for all $x\in B$, $\rho'_{x}$ is surjective. 

If,  for all $a\in A$, $\rho_a$ is bijective and the mapping $B\rightarrow B: f(b) \mapsto f\rho_{a}^{-1}(b)$ is well-defined, that is $f\rho^{-1}_a(b)=f\rho^{-1}_{a'}(b')$ for any $a,a',b,b' \in A$ such that $f(a)=f(a')$ and $f(b)=f(b')$, then all $(\rho')_{f(a)}$ are bijective
and $(\rho')_{f(a)}^{-1} f=f\rho_a^{-1}$.
\end{lemma}
\begin{proof}
    By definition  we have that $\rho_b(a) =\lambda^{-1}_{\lambda_a(b)}\cm_{\lambda_a(b)}(a)$ and $\rho'_y(x) =(\lambda')^{-1}_{\lambda'_x(y)}\cm'_{\lambda'_x(y)}(x)$, for all $a,b\in A$, $x,y\in B$. 
    It follows that
    \begin{align*}
        \rho'_{f(b)}f(a) &= (\lambda')^{-1}_{\lambda'_{f(a)}f(b)}\cm'_{\lambda'_{f(a)}f(b)}f(a)
        =(\lambda')^{-1}_{f(\lambda_{a}(b))}\cm'_{f(\lambda_{a}(b))}f(a)\\
        &=(\lambda')^{-1}_{f(\lambda_{a}(b))}f(\cm_{\lambda_{a}(b)}(a))
        =f(\lambda^{-1}_{\lambda_{a}(b)}\cm_{\lambda_{a}(b)}(a))\\
        &=f\rho_b(a),
    \end{align*}
    for any $a,b\in A$.
    This proves the first part.
    Obviously this implies that $\rho'_{f(b)}$ is surjective if both $f$ and $\rho_b$ are surjective. Hence the second part also follows.
    
    For the third part,
    assume $\rho_a$ bijective, for any $a\in A$, and that the mapping 
    $g:B\rightarrow B: f(b) \mapsto f\rho_{a}^{-1}(b)$ is well-defined. 
     Then
    \begin{align*}
        g\rho'_{f(a)}f(b) =g f\rho_{a}(b)= f\rho^{-1}_a\rho_a(b) =f(b),
    \end{align*}
    and 
    \begin{align*}
    \rho'_{f(a)}gf(b) = \rho'_{f(a)}f\rho^{-1}_a(b) = f\rho_a\rho^{-1}_a(b) = f(b).
    \end{align*}
    Therefore $\rho'_{f(a)}$ is bijective with inverse $g$.
\end{proof}

Recall that a solution $(X',r')$ is said to be an \emph{epimorphic image} of a solution $(X,r)$ if there exists a surjective map
$f:X\rightarrow X'$ so that $(f\times f) r =r' (f\times f)$ on $X^{2}$. If, moreover, $f$ is bijective then one says that the solutions are isomorphic.

\begin{proposition}
Let $(A,+,\circ,\lambda,\cm)$ and $(A',+',\circ',\lambda',\cm')$ be \ourstructures. Then the associated solution $r_{A'}$ is an epimorphic (respectively isomorphic) image of the associated solution $r_A$ 
if and only if there exists a surjective (respectively bijective)  map $f:A\to A'$ such that $f\lambda_a = \lambda'_{f(a)}f$ and $f\cm_a=\cm'_{f(a)}f$, for any $a \in A$ (that is condition (2) of Definition~\ref{def:homomorphism} holds).
\end{proposition}
\begin{proof}
    Assume that $r_{A'}$ is an epimorphic image of  $r_{A}$. Then there exists a surjective map $f:A\to A'$ such that $(f\times f)r_A = r_{A'}(f\times f)$. It follows that 
    \begin{align*}
        (f\lambda_a(b),f\lambda^{-1}_{\lambda_a(b)}\cm_{\lambda_a(b)}(a))=(\lambda'_{f(a)}f(b),(\lambda')^{-1}_{\lambda'_{f(a)}f(b)}
        \cm'_{\lambda'_{f(a)}f(b)}(f(a))).
    \end{align*}
    Hence, in particular $f\lambda_a(b)=\lambda'_{f(a)}f(b)$ for every $a,b \in A$. It yields $f\lambda_a = \lambda'_{f(a)}f$ and, since $\lambda_a$ and $\lambda'_{f(a)}$ are bijective, $f\lambda_a^{-1}=(\lambda')^{-1}_{f(a)}f$. Hence, we get 
    \begin{align*}
        (\lambda')^{-1}_{\lambda'_{f(a)}f(b)}\cm'_{\lambda'_{f(a)}f(b)}(f(a))&=f\lambda^{-1}_{\lambda_a(b)}\cm_{\lambda_a(b)}(a) \\
        &=(\lambda')^{-1}_{f\lambda_a(b)}f \cm_{\lambda_a(b)}(a)\\
        &=(\lambda')^{-1}_{\lambda'_{f(a)}f(b)}f \cm_{\lambda_a(b)}(a).
    \end{align*}
    It follows that $f \cm_{\lambda_a(b)}(a)=\cm'_{\lambda'_{f(a)}f(b)}(f(a)) =\cm'_{f\lambda_a(b)}f(a)$, for every $a,b\in A$. Hence $f \cm_b(a)=\cm_{f(b)}'f(a)$.
    
    Conversely, assume that $f\lambda_a = \lambda'_{f(a)}f$ and $f \cm_a=\cm'_{f(a)}f$, for any $a \in A$. Then
    \begin{align*}
        r_{A'}(f\times f)(a,b)&=(\lambda'_{f(a)}f(b),(\lambda')^{-1}_{\lambda'_{f(a)}f(b)}\cm'_{\lambda'_{f(a)}f(b)}f(a)) \\
        &=(f\lambda_a(b), (\lambda')^{-1}_{f\lambda_a(b)}\cm'_{f\lambda_a(b)}f(a))\\
        &=(f\lambda_a(b), (\lambda')^{-1}_{f\lambda_a(b)}f \cm_{\lambda_a(b)}(a))\\
        &=(f\lambda_a(b), f\lambda^{-1}_{\lambda_a(b)}\cm_{\lambda_a(b)}(a))\\
        &=(f\times f)r_A(a,b).
    \end{align*}
Hence the result follows.
\end{proof}

So, isomorphic \ourstructures yield isomorphic associated solutions. The converse however is not true, take for example two non-isomorphic abelian groups of the same size $(G,+)$ and $(H,\cdot)$. Then $(G,+,+,\iota, \cm)$ with $\iota_g = \cm_g = \id_G$ for all $g \in G$, and $(H,\cdot,\cdot,\iota, \cm')$ with $\iota_h = \cm'_h = \id_H$ for all $h \in H$, are non-isomorphic \ourstructures with isomorphic associated solutions.

\subsection{Structure monoids as \ourstructures}\label{sec_StructureMonoid}\hfill

In this section we show that all left non-degenerate solutions are determined by \ourstructures.
We do this by putting a \ourstructure structure on the structure monoid associated to a left non-degenerate solution,  by specifically defining the $\cm$-map using the derived solution. At the end of this section we show that isomorphic solutions have isomorphic (unital) structure \ourstructures and vice versa.

We begin by recalling results of Gateva-Ivanova and Majid \cite{GIMa08} and Ced\'o, Jespers and Verwimp \cite{CJV2020} using the same notation as in \cite{CJV2020}. 
For completeness' sake we include all the necessary background.
The notation $\langle X \rangle$ is used to denote the monoid generated by the set $X$.

Let $(X,r)$ be a solution. As before, write  $r(x,y)=(\lambda_x (y), \rho_y(x))$, for $x,y\in X$. The \emph{(left) derived monoid} 
$A(X,r)$ of $(X,r)$ is defined as the monoid with the following presentation
$$
    A=A(X,r)= \langle x\in X \mid x+\lambda_x (y) =\lambda_x (y) +\lambda_{\lambda_x(y)} \rho_{y}(x),\; x,y \in X\rangle, 
$$ 
and the \emph{structure monoid} $M(X,r)$ of $(X,r)$ is the monoid with presentation 
$$
    M=M(X,r)=\langle x \in X \mid x\circ y =\lambda_x (y) \circ \rho_y (x), \; x,y\in X\rangle .
$$
The subsemigroup  $M\setminus\{ 1\}$ is called the {\it structure semigroup}  of $(X,r)$ and is denoted by $S(X,r)$.
In \cite[Proposition 2.2]{JeKuVA19} it is shown that if $(X,r)$ is left non-degenerate then  
  $$s: X^2 \rightarrow X^2: (x,y)\mapsto (y,\sigma_y (x)),$$
also is a solution, called the \emph{derived solution} of $(X,r)$, where
\begin{eqnarray}\label{existc}
\sigma_y (x) =\lambda_y \rho_{\lambda^{-1}_{x}(y)}(x).
\end{eqnarray}
Note that in this case $x+y=y+\sigma_y (x)$ for all $x,y\in X$ (explaining the use of the notation $\sigma_y$ in \eqref{existc}),
and furthermore, $s$ is bijective if and only if $r$ is bijective (or equivalently, $s$ is right non-degenerate).

It was shown by Gateva-Ivanova  and Majid \cite[Theorem 3.6]{GIMa08} that  for an arbitrary solution $(X,r)$
the maps $\lambda$ and  $\rho$ can be extended uniquely to $M$ (for simplicity we use the same notation for the extension)
so that one  obtains a monoid homomorphism
 $\lambda : (M,\circ)  \rightarrow \Map (M,M): a\mapsto \lambda_a,$
 and
  a monoid anti-homomorphism
  $\rho : (M,\circ)  \rightarrow \Map (M,M): a\mapsto \rho_a.$
Furthermore, the map
  $r_M\colon M\times M\rightarrow M\times M,$
defined by
 $r_M(a,b)=(\lambda_a(b),\rho_b(a))$, for all $a,b\in M$
is a solution. 
For $a,b\in M$ we have
\begin{eqnarray}\label{eqproduct}
a\circ b=\lambda_{a}(b)\circ\rho_{b}(a),\label{eq:strmon1}\\
\rho_b(c \circ  a) = \rho_{\lambda_a (b)}(c) \circ  \rho_b(a),\label{eq:strmon2}\\
 \lambda_b (a \circ  c) = \lambda_b (a) \circ  \lambda_{\rho_a(b)}(c).\label{eq:strmon3}
\end{eqnarray}
Also, 
$r$ is left (respectively right)  non-degenerate if and only if $r_M$ is left (respectively right) non-degenerate; and 
$r$ is bijective if and only if $r_M$ is bijective.
Note that the derived solution of $(M,r_M)$ of a left non-degenerate solution 
is the solution
  $$s_M: M^2 \rightarrow M^2 : (a,b) \mapsto (a,\sigma_b (a)),$$
where $\sigma_a (b) = \lambda_a \rho_{\lambda^{-1}_{b}(a)} (b)$.
From the above we know that $r_M$ is bijective if and only if $s_M$ is bijective.

Actually one obtains more as shown by Ced\'o, Jespers and Verwimp in \cite[Proposition 3.1]{CJV2020}.  Each $\lambda$ induces a  monoid homomorphism
  $$\lambda' : (M,\circ)  \rightarrow \End(A,+): a\mapsto \lambda_a', $$
with $\lambda'_x(y) =\lambda_x (y)$, for $x,y\in X$.

In \cite[Proposition 3.2]{CJV2020} it is  shown that there exists a $1$-cocycle
  $$\pi : M \rightarrow A,$$
with respect to the homomorphism $\lambda'$, such that $\pi (x) =x$ for all $x\in X$, and hence one obtains a homomorphism
  $$f:M \rightarrow A \rtimes \text{Im} (\lambda') : m\mapsto (\pi (m),\lambda'_m).$$
Furthermore, if $r$ is left non-degenerate then  it follows from Proposition 3.6 and Proposition 3.7 in \cite{CJV2020} that $\pi$ is bijective and thus $$M\cong \{ (\pi (m),\lambda'_m)\mid m\in M\} \subseteq A \rtimes \text{Im} (\lambda').$$
For left non-degenerate solutions we will identify $M=M(X,r)$ with $A=A(X,r)$ via the mapping $\pi$
and we thus denote $\lambda'$ (respectively $\rho'$) by $\lambda$ (respectively $\rho$).
Hence $M$ also has an additive structure $(M,+)$ and we have $a\circ b =a+\lambda_a (b)$ and $a+b =b+\sigma_b(a)$ for all $a,b\in A$.
Also, 
\begin{align*}
    \lambda_a\sigma _d(b) 
    &= \lambda_a\lambda_d\rho_{\lambda^{-1}_b(d)}(b)
    \\ &= \lambda_{\lambda_a(d)} \lambda_{\rho_d(a)}\rho_{\lambda^{-1}_b(d)}(b)
    \\ &= \lambda_{\lambda_a(d)} \lambda_{\rho_{\lambda_b(\lambda^{-1}_b(d))}(a)}\rho_{\lambda^{-1}_b(d)}(b)
    \\ &= \lambda_{\lambda_a(d)} \rho_{\lambda_{\rho_b(a)}(\lambda^{-1}_b(d))}\lambda_a(b)
    \\ &= \lambda_{\lambda_a(d)} \rho_{\lambda^{-1}_{\lambda_a(b)}(\lambda_a(d))}\lambda_a(b)
    \\ &= \sigma_{\lambda_a(d)}\lambda_a(b).
\end{align*}
So equation (\ref{eq:clambda}) holds. 
Moreover, because of equation \eqref{eq:strmon2} one can show (with a reversed argument to that in the last part of the  proof of  Proposition~\ref{epitruss})  that each $\sigma_a \in \End(M,+)$. Finally, since $\rho$ is an anti-homomorphism we can prove that $\sigma_{a+b} = \sigma_b \sigma_a$, for all $a,b\in M(X,r)$.

To summarize we have shown the following.

\begin{proposition}\label{SolFromSemiTruss}
 Let $(X,r)$ be a left non-degenerate solution. The  structure monoid $M=M(X,r)$ is a unital \ourstructure
 $(M,+,\circ , \lambda , \cm)$ (with $\lambda$-map and $\cm$-map as defined above)
 and with the structure semigroup $S(X,r)$ as a sub-\ourstructure (not unital).  
 The associated solution of the unital  structure \ourstructure $M$ (respectively associated derived solution of the \ourstructure $M$) (as defined in Proposition~\ref{prop:semitruss}) is precisely the solution $(M,r_M)$ (respectively $(M,s_M)$; hence justifying the terminology in Section 1).
\end{proposition}

We refer to the \ourstructure $(M(X,r), +,\circ ,\lambda, \cm)$ as the \emph{unital structure \ourstructure} associated to $(X,r)$ and to the sub-\ourstructure $S(X,r)$ as the \emph{structure \ourstructure} of $(X,r)$.

We note that for specific  solutions $(X,r)$ the monoid  $M(X,r)$ possibly can be made into a \ourstructure in several manners. The following example shows that one can consider another  $\cm$-map. This shows that given a set $A$ with two operations $+,\circ$ and a $\lambda$-map, it is possible to have more than one $\cm$-map that makes $(A,+,\circ,\lambda,\cm)$ into a \ourstructure.

\begin{example}\label{ex:candc'}
Let $X=\{1,2\}$. Let  $r:X\times X \rightarrow X \times X$ denote the left non-degenerate solution with $\lambda_1=\lambda_2=\rho_2 = \id_{X}$ and $\rho_1: X \rightarrow X: x \mapsto 1$. The associated left derived structure monoid is
    $A = A(X,r) = \langle X \mid 1 + 2 = 2 + 1 = 1 + 1 \rangle$ and the $\cm$-map is such that $\cm_1$ is the constant map onto $1$, and $\cm_2=\id$.
Note that $M=M(X,r)=A$ and $r_A=r_M$. One can also consider another $\cm$-map, denoted $\cm'$, and defined by $\cm'_1$  the constant mapping onto $2$ and $\cm'_2=\id$. Notice that the $\cm'$-map yields a left non-degenerate solution
$r':X^2\rightarrow X^2$ with $\lambda$-map (respectively $\rho$-map), denoted $\lambda'$ (respectively $\rho'$), and defined by 
$\lambda'_1=\lambda'_2 = \rho'_2 = \id_X$ and $\rho'_1$ is the constant map onto $2$.
We get that $M(X,r') = A(X,r') = A=M(X,r)$.
However, it easily is verified that the solutions $r$ and $r'$ are not isomorphic and thus also $r_M$ and $r'_M$ are not isomorphic.
\end{example}

Let $\mathbb{N}$ denote the non-negative integers.
Note that  $M(X,r)$  has a natural strongly $\mathbb{N}$-gradation (its generators $x\in X$ are the elements of degree $1$) in the following sense.

\begin{definition}\label{definition:graded}
One says that a \ourstructure $(A,+,\circ , \lambda , \cm)$ is \emph{$\mathbb{N}$-graded} (respectively \emph{strongly $\mathbb{N}$-graded})
if $A=\bigcup_{n\in \mathbb{N}} A_n$, a disjoint union  of subsets indexed by the non-negative integers, so that $A_n+ A_m \subseteq  A_{n+m}$ (respectively $A_n + A_m =A_{n+m}$) and all maps $\lambda_a$ and $\cm_a$ are graded, i.e. degree preserving.
In particular, if $A$ is strongly $\mathbb{N}$-graded then the subsemigroup $A\setminus A_0$   is additively (and thus  also multiplicatively) generated by its elements of degree $1$. Note that the derived solution $r_A : A^2 \rightarrow A^2$ restricts to a solution $r_{A_1} : A_1^2 \rightarrow A_1^2$.
\end{definition}

\begin{corollary}\label{CorEpiGraded}

Let $(A,+,\circ , \lambda , \cm)$ be a \ourstructure and let $S(A,r_A)$ be the structure semigroup of its associated solution.
Then the \ourstructure $A$ is an epimorphic image of the  structure \ourstructure $S(A,r_A)$.
If $A$ is a unital \ourstructure then $A$ is an epimorphic image of the unital structure \ourstructure $M(A,r_A)$.
Furthermore, if $A$ is strongly $\mathbb{N}$-graded with $A_0=\emptyset$  then $A$ is a graded  epimorphic image of the $\mathbb{N}$-graded structure \ourstructure $S(A_1,r_{A_1})$. If, moreover, $A$ is unital and $A_0=\{ 1\}$ 
 then $A$ is a graded  epimorphic image of the $\mathbb{N}$-graded unital structure \ourstructure $M(A_1,r_{A_1})$.

In particular, left non-degenerate solutions are restrictions of solutions associated to  \ourstructures.
\end{corollary}

\begin{proof}
Of course there is a natural map of the generating set $A$ of $S(A,r_A)$ to $A$. Because of the equations $\eqref{eq:strmon1}$, $\eqref{eq:strmon2}$ and $\eqref{eq:strmon3}$ in $S(A,r_A)$ and the equations $\eqref{eq:truss1}$, $\eqref{eq:truss2}$ and $\eqref{eq:truss3}$
this map can be extended to a map $f:S(A,r_A)\rightarrow A$. Because of Lemma~\ref{chomomorphism} and $\eqref{existc}$ it follows
that $f$ is an epimorphism of \ourstructures.
In case $A$ is unital, we have that $\lambda_1=\id$ in both \ourstructures and also $\lambda_a(1) =1$ for all $a$.  Hence we can extend $f$ to $M = M(A,r_A)$. 
This proves the first part of the statement. The graded statements for $A_0=\emptyset$ follow at once as $A$ is generated by $A_1\cup A_0$ and the natural map preserves degrees.
\end{proof}

In the previous section it was shown that isomorphic \ourstructures yield isomorphic associated solutions. The converse is however not true. Yet, isomorphic left non-degenerate solutions have isomorphic (unital) structure \ourstructures and vice versa.

\begin{proposition}
    Let $(X,r)$ and $(Y,r')$ be left non-degenerate solutions. Then $(X,r) \cong (Y,r')$ if and only if the associated (unital) structure \ourstructures $S(X,r)$ and $S(Y,r')$ ($M(X,r)$ and $M(Y,r')$) are isomorphic as \ourstructures.
\end{proposition}
\begin{proof}
Let $(X,r)$ and $(Y,r')$ be two isomorphic left non-degenerate solutions and denote $r(x,y)=(\lambda_x(y), \rho_y(x))$, for all $x,y \in X$, and $r'(x',y') = (\lambda'_{x'}(y'), \rho'_{y'}(x'))$, for all $x',y' \in Y$. Then, there exists a bijection $f:X \to Y$ such that $(f \times f) r = r' (f \times f)$, i.e. for any $x,y \in X$, $f(\lambda_x(y))= \lambda'_{f(x)}(f(y))$ and $f(\rho_y(x))= \rho'_{f(y)}(f(x))$. 
Defining the $\cm$-map and $\cm'$-map as \eqref{existc} for $M(X,r)$ and $M(Y,r')$ respectively,
we get for any $x,y \in X$,
    \begin{align*}
        f(\cm_y(x)) &= f(\lambda_y(\rho_{\lambda^{-1}_x(y)}(x)))
        \\ &= \lambda'_f(y)(f(\rho_{\lambda^{-1}_x(y)}(x)))
        \\ &= \lambda'_f(y)(\rho'_{f(\lambda^{-1}_x(y))}(f(x)))
        \\ &= \lambda'_f(y)(\rho'_{\lambda^{-1}_{f(x)}(f(y))}(f(x)))
        \\ &= \cm'_{f(y)}(f(x)).
    \end{align*}
We will prove that the associated unital structure \ourstructures $M=M(X,r)$ and $M'=M(Y,r')$ are isomorphic. The proof for the associated structure \ourstructures is similar. 
Extend $f$ to $M$ by defining $f(x \circ a) = f(x) \circ' f(a)$ for all $x \in X, a \in M$. Hence, $f(a\circ b) = f(a) \circ' f(b)$, for all $a,b \in M$. One can check that $f:M \to M'$ is a bijection. To prove that it is well-defined, let $x,y \in X$. Then, 
    \begin{align*}
        f(\lambda_x(y) \circ \rho_y(x))
        &= f(\lambda_x(y)) \circ' f(\rho_y(x))
        \\ &= \lambda'_{f(x)}(f(y)) \circ' \rho'_{f(y)}(f(x))
        \\ &= f(x) \circ' f(y)
        \\ &= f(x \circ y).
    \end{align*}
By induction on the length of the elements in $M$, one can show that $f(\lambda_a(b))= \lambda'_{f(a)}(f(b))$ and $f(\cm_b(a)) = \cm'_{f(b)}(f(a))$, for any $a,b \in M$. Finally, since $a+b = a \circ \lambda^{-1}_a(b)$ and $f(a) +' f(b) = f(a) \circ' \lambda'^{-1}_{f(a)}(f(b))$ for all $a,b \in M$, we obtain
    \begin{align*}
        f(a+b) &= f(a \circ \lambda^{-1}_a(b))
        \\ &= f(a) \circ' f(\lambda^{-1}_a(b))
        \\ &= f(a) \circ' \lambda'^{-1}_{f(a)}(f(b))
        \\ &= f(a) +' f(b).
    \end{align*}
    Hence, by Definition \ref{def:homomorphism}, $(M(X,r),+,\circ,\lambda,\cm) \cong (M(Y,r'),+',\circ',\lambda',\cm')$ as \ourstructures.
    
    Conversely, assume that $M=M(X,r)$ and $M'=M(Y,r')$ are isomorphic as \ourstructures, i.e. there exists a bijection $f:M \to M'$ such that $(f \times f)r_M = r_{M'} (f \times f)$, i.e. for any $a,b \in M$, $f(\lambda_a(b))= \lambda'_{f(a)}(f(b))$ and $f(\rho_b(a))= \rho'_{f(b)}(f(a))$. Hence, for any $x,y \in X$, $f(\lambda_x(y))= \lambda'_{f(x)}(f(y))$ and $f(\rho_y(x))= \rho'_{f(y)}(f(x))$. Furthermore, since the set $X$ is the unique minimal set of generators of the monoid  $(M,\circ)$ and $f$ is a monoid homomorphism, we get that $f(X)=Y$, the unique minimal set of generators of the monoid $M'$. Hence, $(X,r) \cong (Y,r')$.
\end{proof}

Two structure monoids that are isomorphic as monoids do not necessarily imply isomorphic solutions. In \cite[Example 1.1]{JeKuVA19}, two non-isomorphic bijective non-degenerate solutions are given with the same structure monoid. In case both solutions are involutive non-degenerate one has a satisfactory answer. This already has  been pointed out in \cite[page 6]{JeKuVA19}. 
If one of the solutions is bijective while the other is still involutive, not necessarily non-degenerate, the result is still true.

\begin{proposition}\label{IsoStr}
    Let $(X,r)$ be an involutive solution. Suppose $(Y,r')$ is a bijective solution such that $M(X,r) \cong M(Y,r')$ as monoids. Then, $(X,r) \cong (Y,r')$ as solutions.
\end{proposition}
\begin{proof}
The monoid $M(X,r)$ has a unique minimal generating set $X$, which holds analogously for $Y$. Hence, a monoid isomorphism $f$ between $M(X,r)$ and $M(Y,r')$ maps $X$ to $Y$. Hence, this can be restricted to a map $f: X \longrightarrow Y$. Moreover, considering that $r$ is involutive, one obtains that at most two words of length $2$ (in the alphabet $X$) become equal in the structure monoid of $M(X,r)$. Hence, the same can be said about words of length $2$ (in alphabet $Y$) in $M(Y,r')$. In particular, $r'$ is involutive. 
Let $x_1,x_2 \in X$. Then, either $x_1\circ x_2$ can not be rewritten or $x_1\circ x_2 = u\circ v$ with $r(x_1,x_2) = (u,v)$. In the former case, as $f$ is a monoid isomorphism, it follows that $f(x_1\circ x_2) = f(x_1)\circ f(x_2)$ can also not be rewritten, implying that $\lambda_{f(x_1)}(f(x_2)) = f(x_1)$ and $\rho_{f(x_2)}(f(x_1))=f(x_2)$. In the latter case, $f(x_1)\circ f(x_2) = f(u)\circ f(v)$ are the two only ways to write this element. Considering the defining relation of $M(Y,s)$, it follows that $\lambda_{f(x_1)}(f(x_2)) = f(u)$ and $\rho_{f(x_2)}(f(x_1)) = f(v)$. 
In conclusion, in both cases one obtains that $f\lambda_x = \lambda_{f(x)}f$ and $f\rho_x = \rho_{f(x)}f$ for all $x \in X$, which shows that $f$ is an isomorphism of solutions. 
\end{proof}

The following example shows that the assumption that $(Y,r')$ is bijective is essential.
Let $X=\{1,2\}$ and $(X,r)$  the trivial  solution, i.e. $r(x,y)=(y,x)$, for $x,y \in X$. Its  structure monoid is $M(X,r) = \langle X \mid 1 \circ 2 = 2 \circ 1 \rangle$, the free abelian monoid of rank 2. On $X$ one can also have a degenerate idempotent solution $(X,r')$  with $r'(x,y) = (\lambda_x(y), \rho_y(x))$ and $\lambda_1: x  \mapsto 1, \lambda_2 = \rho_1 = \id$ and $\rho_2: x \mapsto 2$, i.e. $r'(2,1)=(1,2)$ and the other pairs are fixed under $r'$.
Then, $M(X,r') = \langle X \mid 2 \circ 1 = 1 \circ 2 \rangle = M(X,r)$. However, $(X,r) \not\cong (X,r')$.
Note that the example $(X,r')$ is not left non-degenerate and that, because of Theorem~\ref{maintheorembijective} no counterexamples exists in case $(Y,r')$ is finite non-degenerate.

We end this section with a class of examples that yield left non-degenerate idempotent solutions.

\begin{example}\label{exidempotent}

If $(X,r)$ is an idempotent solution, i.e. $r^2=r$, then $\lambda_{\lambda_x(y)}(\rho_y(x)) = \lambda_x(y)$ and $\rho_{\rho_y(x)}(\lambda_x(y)) = \rho_y(x)$. If furthermore the solution is left non-degenerate, the associated $\cm$-map on the unital structure \ourstructure $M(X,r)$ satisfies $\cm_b(a) = \lambda_b(\rho_{\lambda^{-1}_a(b)}(a)) =  \lambda_{\lambda_a(t)}(\rho_{t}(a)) = \lambda_a(t) = \lambda_a(\lambda^{-1}_a(b)) = b$, with $t=\lambda^{-1}_a(b)$.
Conversely, assume $(A, +, \circ, \lambda, \cm)$ is a \ourstructure  with $\cm_b(a) = b$ for all $a,b \in A$, then one easily verifies that $(r_A)^2 = r_A$, i.e. the solution $(A,r_A)$ is idempotent.
Hence, idempotent left non-degenerate solutions correspond with \ourstructures $(A,+,\circ, \lambda, \cm)$ with $\cm_b$ the constant mapping onto $b$, for every $b\in A$.
\end{example}

\subsection{Left cancellative \ourstructures}\label{sec_leftcancellative}\hfill

In this section we show that the known algebraic associative structures that determine left non-degenerate solutions, and that are mentioned in the introduction, are special classes of \ourstructures.

Let $(A,+,\circ, \lambda ,\cm)$ be a \ourstructure. Because $a\circ b =a+\lambda_a (b)$ for all $a,b\in A$, it easily is verified that
$(A,+)$ being left cancellative (i.e. $a+b=a+d$ implies $b=d$) is equivalent to $(A,\circ)$ being left cancellative. 
So if $(A,+)$ is a left cancellative semigroup, then the mappings $\cm_b$
are uniquely determined as $a+b=b+\cm_b(a)$ for any $a,b\in A$ (in the non-cancellative case this is not necessarily true, see \cref{ex:candc'}). 
Furthermore, since  (by Proposition~\ref{epitruss}) $a\circ b =\lambda_a (b) \circ \lambda^{-1}_{\lambda_a (b)}\cm_{\lambda_a(b)} (a)$, there exists a unique $x\in A$ so that $a\circ b =\lambda_a (b) \circ x$. It turns out that several of the requirements to be a \ourstructure are then redundant. 

\begin{proposition}\label{cancellativetruss}
Let $(A,+,\circ,\lambda)$ be a left semitruss with $(A,+)$ left cancellative (and thus also $(A,\circ)$ is left cancellative) and $\lambda_a$ bijective, for any $a\in A$ such that for all $a,b\in A$, the equation
\begin{align*}
    a\circ b =a +\lambda_a(b),
\end{align*}
holds and there exists an $x\in A$ such that
\begin{align*}
    a\circ b =\lambda_a(b)\circ x.
\end{align*}
We put $\rho_b(a) = x$. Then there exists a unique  semigroup anti-morphism $\cm:(A,+)\to \End(A,+): a \mapsto \cm_a$  such that $(A,+,\circ,\lambda,\cm)$ is a \ourstructure. 
For $a,b\in A$ we have $\cm_b(a)= \lambda_b\rho_{\lambda^{-1}_a(b)}(a)$.
\begin{proof} 
By \cite[Proposition 2.1]{Br18}, we have that  $\lambda:(A,\circ)\to \Aut (A,+): a\mapsto \lambda_a$ is a semigroup morphism. By the previous and the  assumptions, for any $a,b\in A$, there is a unique element $x\in A$ such that $a\circ b= \lambda_a(b)\circ x$.  We  denote  $x$ by $\rho_b(a)$. Now, set $\cm_b(a)= \lambda_b\rho_{\lambda^{-1}_a(b)}(a)$. Then
\begin{align*}
   b+\cm_b(a)&= b+\lambda_b\rho_{\lambda^{-1}_a(b)}(a) = b\circ\rho_{\lambda^{-1}_a(b)}(a) = \lambda_a\lambda^{-1}_a(b) \circ \rho_{\lambda^{-1}_a(b)}(a) \\&= a\circ\lambda^{-1}_a(b) = a + \lambda_a\lambda^{-1}_a(b) =a+b.
\end{align*}
Hence, \eqref{eq:csum} holds. Moreover,
\begin{align*}
    \lambda_a(b) + \cm_{\lambda_a(b)}\lambda_a(d) &= \lambda_a(d) + \lambda_a(b) = \lambda_a(d+b) = \lambda_a(b+\cm_b(d)) 
    \\&=\lambda_a(b) +\lambda_a \cm_b(d),
\end{align*}
and, since $(A,+)$ is left cancellative, we get $\cm_{\lambda_a(b)}\lambda_a(d) =\lambda_a\cm_b(d)$, i.e. condition \eqref{eq:clambda} is satisfied. Furthermore, note that
\begin{align*}
    d+ \cm_d(a+b) = a+(b+d) = a+ d+ \cm_d(b) = d+\cm_d(a)+\cm_d(b),
\end{align*}
i.e. $\cm_d(a+b) = \cm_d(a)+\cm_d(b)$ and 
\begin{align*}
    b+a+\cm_a  \cm_b(d) = b +\cm_b(d) + a = d + (b+a) = b+a + \cm_{b+a}(d),
\end{align*}
i.e. $\cm_a \cm_b =\cm_{b+a}$. Hence $\cm: (A,+)\to \End(A,+)$ is an anti-morphism,
i.e. \eqref{eq:condc} holds. 
Finally, since \eqref{eq:sumcirc} holds by assumption, $(A,+,\circ, \lambda,\cm)$ is a \ourstructure.
\end{proof}
\end{proposition}

We call a \ourstructure $(A,+,\circ,\lambda,\cm)$ with $(A,+)$ left cancellative (and so also $(A,\circ)$ left cancellative) a \emph{left cancellative \ourstructure}.

\begin{example}
   Let $(X,r)$ be a left non-degenerate solution and let $\eta$ be the left cancellative congruence on $(M(X,r),+)$. In \cite[Lemma 5.8, Corollary 5.9]{CJV2020}, it is shown that the left cancellative image $\bar{M}=M/\eta$ of $(M(X,r),+)$ has a left semitruss structure, denoted by $(\bar{M}, +, \circ, \bar{\lambda})$, with a unique $\bar{\cm}$-map such that $a+\bar{\lambda}_a(b)=a\circ b$ and $a+b= b+\bar{\cm}_b(a)$. We claim that $(\bar{M}, +, \circ, \bar{\lambda}, \bar{\cm})$ is a left cancellative \ourstructure. Indeed $(\bar{M}, +)$ and $(\bar{M}, \circ)$ are left cancellative and
   \begin{align*}
       a\circ b = a+\bar{\lambda}_a(b) = \bar{\lambda}_a(b) +\bar{\cm}_{\bar{\lambda}_a(b)}(a) =\bar{\lambda}_a(b)\circ \bar{\lambda}^{-1}_{\bar{\lambda}_a(b)}\bar{\cm}_{\bar{\lambda}_a(b)}(a),
   \end{align*}
   i.e. $(\bar{M}, +, \circ, \bar{\lambda})$ is a left semitruss satisfying the assumptions of \cref{cancellativetruss}. Therefore $(\bar{M}, +, \circ, \bar{\lambda}, \bar{\cm})$ is a left cancellative \ourstructure.
\end{example}

\begin{example}\label{excsb}
 As defined by Catino,  Colazzo and Stefanelli in \cite{CCS2017}, a \emph{left cancellative semi-brace} is a triple $(A,+,\circ)$ such that $(A,\circ)$ is a group, $(A,+)$ is a left cancellative semigroup and the condition $a\circ (b+d) = a \circ b + a\circ (\bar{a}+d)$ holds for any $a,b,d \in A$, where $\bar{a}$ denotes the inverse of $a$ in $(A,\circ)$. 
 Any left cancellative semi-brace is a \ourstructure with $(A,\circ)$ a group,  $\lambda_a(b)=a\circ(\bar{a}+b)$ and $\cm_b(a) = \lambda_b(\bar{b}\circ (a+b))=\lambda_b(\bar{b})+a+b$ for all $a,b\in A$.
 Conversely,  a  \ourstructure with $(A,\circ)$ a group is a left cancellative semi-brace.
    \begin{proof}
     Let $(A,+,\circ)$ be a left cancellative semi-brace. Let $1$ denote the identity of $(A,\circ)$. Then $1+1 = 1\circ(1+1) = 1\circ 1 + 1\circ (\bar{1}+1) = 1+1+1$ and thus the left cancellativity of  $(A,+)$  yields $1=1+1$. So, for any $a\in A$ we obtain that $1+a = 1+1+a$ and thus $a= 1+a$, i.e. $1$ is a left identity of $(A,+)$. Put $\lambda_a(b)=a\circ(\bar{a}+b)$, by \cite[Proposition 3]{CCS2017} the map $\lambda: (A,\circ) \to \Aut(A,+): a\mapsto \lambda_a$ is a group homomorphism and $a\circ(b+d) = a \circ b +\lambda_a(d)$. In addition, $a+\lambda_a(b)=a\circ 1+\lambda_a(b) = a\circ(1+b) = a\circ b$.
     Because $(A,\circ)$ is a group, for any $a,b\in A$ there exists a unique $x\in A$ such that $a\circ b =\lambda_a (b)\circ x$.
     Hence, from Proposition~\ref{cancellativetruss} we obtain that indeed $(A,+,\circ, \lambda ,\cm)$ is a \ourstructure.
    
    Conversely,
    let $(A,+,\circ,\lambda, \cm)$ be a \ourstructure with $(A,\circ)$ a group. Then,  $(A,+)$ is a left cancellative semigroup. Moreover, \eqref{eq:sumcirc} yields $\lambda^{-1}_a(b) = \bar{a}\circ(a+b)$ and, since $\lambda_a^{-1}=\lambda_{\bar{a}}$, we get $\lambda_a(b)=a\circ(\bar{a}+b)$. Hence, $a\circ(b+d) = a\circ b +\lambda_a(d)= a\circ b +a \circ (\bar{a}+d)$. Therefore, $(A,+,\circ)$ is a left cancellative semi-brace.
    \end{proof}
\end{example}

\begin{example}\label{exsb}
As defined by Guarnieri and Vendramin in  \cite{GV17}, a \emph{skew left brace}  is a triple $(A,+,\circ)$, with  $(A,+)$ and $(A,\circ)$ groups, so that
$a\circ (b + d)=a\circ b - a +a\circ d$, for all $a,b,d\in A$ (here $-a$ denotes  the inverse of $a$ in the additive group). 
If, furthermore, $(A,+)$ is abelian, then this is called a \emph{left brace}, as introduced by Rump in \cite{Rump2007} (see also \cite{CJO14}).
 Skew left braces are (unital)  \ourstructures $(A,+,\circ, \lambda , \cm)$ where both  $(A,+)$ and $(A,\circ)$ are groups.
In this case $\lambda_a (b) =-a+a\circ b$ and $\cm_b(a) =-b+a+b$. Furthermore, skew left braces are left cancellative semi-braces for which also $\rho_a$ is bijective, for any $a\in A$.
\end{example}
\begin{proof}
Note that if $(A,+,\circ, \lambda, \cm)$ is a \ourstructure with $(A,+)$ and $(A,\circ)$ groups, they share the same identity, say $1$, and by Example~\ref{excsb}, $(A,+,\circ)$  it is a left cancellative semi-brace. Then  $\lambda_a (b)  = a\circ (\bar{a} +b) =-a +(a\circ 1) + (a\circ (\bar{a} +b)) =-a + a\circ (1+b) =-a + (a\circ b) $ and  $\cm_b(a)= \lambda_b (\bar{b}) +a +b
=-b + (b\circ \bar{b}) +a+b =-b+a+b$.  Hence $(A,+,\circ)$ is a skew left brace. The previous equalities also show that a skew left brace is a left cancellative semi-brace with $(A,+)$ also a group.

The last statement of the example has been proven in  \cite[p. 167]{CCS2017}. For completeness' sake we include a short proof.
So, let  $(A,+,\circ)$ be a left cancellative semi-brace and assume that $\rho_a$ is bijective for any $a\in B$.
We know that $(A,+,\circ,\lambda,\cm )$ is  a \ourstructure and thus  $\lambda:(A,\circ)\to \Aut(A,+)$ is a monoid  homomorphism and, by \cref{chomomorphism}, $\rho: (A,\circ) \to \Map (A,A): b\mapsto \rho_b$ is a monoid anti-homomorphism. In particular, since $\lambda_a$ and $\rho_a$ are bijective for any $a\in A$ we get $\lambda_1=\id$ and $\rho_1=\id$ where $1$ denotes the identity of $(A,\circ)$.
Moreover, by \cref{cancellativetruss} we get $\rho_b(a) = \overline{(\lambda_a(b))}\circ a\circ b = \overline{a\circ (\bar{a}+b)}\circ a\circ b = \overline{(\bar{a}+b)}\circ \bar{a}\circ a\circ b = \overline{(\bar{a}+b)}\circ b$, which yields 
\begin{align}\label{eq:semibraceplusrho}
    a+b = b\circ \overline{\rho_b(\bar{a})}.
\end{align}
It follows that $1$ also is an identity for $(A,+)$. Indeed, for any $a\in A$, we have 
    $1+ a =1+\lambda_1(a)= 1\circ a = a$ and 
    $a+1 \overset{\text{\eqref{eq:semibraceplusrho}}}{=} 1 \circ \overline{\rho_1(\bar{a})} = \bar{\bar{a}} = a$.
Furthermore, put $b= \lambda_a(\bar{a})$. Then $a+b = a+ \lambda_a(\bar{a}) = a \circ \bar{a} = 1$. 
Since  $\overline{\rho_{\bar{a}}(a)} = \overline{\overline{(\bar{a}+\bar{a})}\circ \bar{a}}= a\circ (\bar{a}+\bar{a})=\lambda_a(\bar{a})= b$
we also get that
$b+a \overset{\text{\eqref{eq:semibraceplusrho}}}{=} a \circ \overline{\rho_a(\bar{b})} = a\circ \overline{\rho_a\rho_{\bar{a}}(a)} = a\circ \overline{\rho_{\bar{a} \circ a}(a)} = a \circ \overline{\rho_1(a)} = a \circ \bar{a} = 1$. Therefore, $(A,+)$ is a group and thus $(A,+,\circ)$ is skew left brace and the result follows.
\end{proof}

Note that if $(B,+,\circ)$ is a skew left brace, the opposite \ourstructure as defined in \cref{prop:opposite} is a skew left brace and the definition coincides with the definition of an opposite skew left brace given in \cite[Proposition 3.1]{KoTr20}.

In order to deal with degenerate and non-bijective  solutions the notion of a left semi-brace has been introduced  in \cite{JVA19}.
The definition is as in Example~\ref{excsb} but without the restriction that $(A,+)$ has to be left cancellative.
In general this is not a \ourstructure as, for example, $A+b$ is not necessarily contained in $b+A$. It is easy to see that a left semi-brace is a \ourstructure if and only if it is a left cancellative semi-brace.

Catino, Colazzo and Stefanelli in \cite{CaCoSt21} and Catino, Mazzotta and Stefanelli in \cite{CaMaSt} introduced a generalization of left semi-braces. In general these do not yield left non-degenerate solutions.

We conclude this section with an example of a left cancellative \ourstructure that is finite but is not a left cancellative semi-brace.

\begin{example}\label{ex:circnotgroup}
Let $A = \{1,2,3,4\}$. Define for any $a,b \in A$, $a+b=b$, $\cm_b(a) = b$ and $a \circ b = \lambda_a(b)$, where $\lambda_1 = \lambda_2 = \id_A$ and $\lambda_3 = \lambda_4 = (13)(24)$. Then $(A,+,\circ,\lambda,\cm)$ is a \ourstructure with  $(A,+)$ left cancellative and $\rho$-map given by
$\rho_1=\rho_3: t \mapsto 1$ and $\rho_2 = \rho_4: t \mapsto 2$. In particular $1\circ 1 =2\circ 1$ and thus $(A,\circ)$ is not a group.
\end{example}

\subsection{Idempotents in \ourstructures}\label{sec_idempotent}\hfill

For a semigroup $(S,*) $ we denote by $E_{*}(S)$ its subset of idempotents. In case $(A,+,\circ, \lambda ,\cm)$ is a \ourstructure  we have two such subsets $E_{+}(A)$ and $E_{\circ}(A)$.
It has been shown in \cite{CCS2017} that if $A$ is a left cancellative semi-brace then $E_{+}(A)$ is a subsemigroup of $(A,+)$ and, furthermore $A=A + 1 + E_{+}(A)$, where $\{1 \} =E_{\circ}(A)$, and $A+1$ is a maximal subgroup of $(A,+)$ and it is a sub-\ourstructure that is skew left brace. We will prove that for an arbitrary \ourstructure $E_{+}(A)$ is subsemigroup of $(A,+)$ and also $E_{\circ}(A)$ is a subsemigroup of $(A,\circ)$.
In the terminology of \cite{MR1142743}, $(A,+)$ and $(A,\circ)$ are E-semigroups and they thus both have a so-called unitary cover.
Actually we will show that $E_{+}(A)$ is a sub-\ourstructures and $E_{\circ}(A)$ is a sub-semitruss of $A$.

\begin{proposition}\label{prop:idempotents}
    Let $(A,+,\circ,\lambda,\cm)$ be a \ourstructure. The following properties hold.
    \begin{enumerate}
        \item  $\lambda_e=\id$ for $e\in E_{\circ}(A)$.
        \item  $\cm_e^{2}=\cm_e$ for $e\in E_{+}(A)$.
        \item $E_{\circ}(A)\subseteq E_+(A)$.
        \item $E_{+}(A)$ is a sub-\ourstructure  of $A$.
        \item $E_{\circ}(A)$ is a  sub-semitruss of $A$.
        \item If the associated solution $(A,r_A)$ is bijective, then $E_+(A)$ and $E_{\circ}(A)$ are   commutative semigroups.
    \end{enumerate}
    \begin{proof}
      (1) Follows from the fact that $\lambda_e$ is bijective and $\lambda_e^{2}=\com{\lambda_{e\circ e}}=\lambda_e$.
      (2) Follows from the fact that the $\cm$-map is a anti-homomorphism with domain $(A,+)$.
      To prove  (3), let $e\in E_{\circ}(A)$ then $e=e\circ e =e+\lambda_e (e) =e+e$ and thus $e\in E_{+}(A)$.

      (4) Let $e,f \in E_{+}(A)$. Then,
      \begin{align*}
          (e+f)+(e+f) \overset{\eqref{eq:csum}}{=} e+ f + f +\cm_f(e) =  e + f + \cm_f(e) \overset{\eqref{eq:csum}}{=} e+e + f = e+f,
      \end{align*}
      i.e. $e+f\in E_{+}(A)$. Hence  $E_{+}(A)$ is a subsemigroup. Since $e\circ f =e+\lambda_e (f)$ and because $\lambda_e (f)\in E_{+}(A)$,
      as $\lambda_e \in \End (A,+)$, we get that $e\circ f\in E_{+}(A)$. So, $E_{+}(A)$ is a subsemigroup of $(A,\circ)$.
      Since $E_{+}(A)$ is $\lambda$- and $\cm$-invariant it follows that $E_{+}(A)$ is a sub-\ourstructure  of $A$.
      
      (5) Let $e,f\in E_{\circ}(A)$. Then,
      \begin{align*}
          (e\circ f)\circ (e\circ f) &= (e\circ f)+\lambda_{e\circ f} (e\circ f) = e+f +\lambda_e\lambda_f(e+f)\\
          &\overset{(1)}{=} e+f+e+f \overset{(4)}{=} e+f \overset{(1)}{=} e\circ f,
      \end{align*}
      i.e. $e\circ f \in E_{\circ}(A)$. Hence, also $e+f=e\circ \lambda_e^{-1}(f)\overset{(1)}{=} e\circ f \in E_{\circ}(A)$. Moreover, $\lambda_e(f) \overset{(1)}{=} f\in E_{\circ}(A)$.
      It follows that $E_{\circ}(A)$ is a sub-semitruss. 
      
      (6)  Assume $r_A$ is bijective, i.e. $\cm_x$ is bijective for any $x\in A$. Hence, by part  (2) it follows that $\cm_e=\id$ for any $e\in E_+(A)$. Therefore,
        $e+f \overset{\eqref{eq:csum}}{=} f+\cm_f(e) = f+ e$.
    Hence, $E_+(A)$ is commutative. In case $e,f\in E_{\circ}(A)$ then we get that $e\circ f =e+f =f+e =f\circ e$. So also $E_{\circ}(A)$
    is commutative.
    \end{proof}
\end{proposition}

Recall that if $S$ is a right simple semigroup, i.e. $S$ is the  only right ideal, or a left cancellative semigroup then any idempotent is a left identity.
For both classes of semigroups we obtain additional information. We begin with the class of left cancellative semigroups.

\begin{proposition}
 Let $(A,+,\circ,\lambda,\cm)$ be a \ourstructure. If $e\in E_+(A)$ is a left identity of $(A,+)$, then $\cm_e(a) = a+e$. In particular, if $(A,+,\circ,\lambda,\cm)$ is left cancellative, then  $E_{\circ}(A)$ is a sub-\ourstructure.
 \begin{proof}
 Let $e\in E_+(A)$ be a left identity of $(A,+)$ and $a\in A$. It follows that
 \begin{align*}
     \cm_e(a) = e+ \cm_e(a) \overset{\eqref{eq:csum}}{=}  a+e. 
 \end{align*}
Hence, if $(A,+)$ is left cancellative  then, for any $e,f\in E_{\circ}(A)$, we get that  $\cm_{e}(f) = f+e = e \in E_{\circ}(A)$.
So $E_{\circ}(A)$ is $\lambda$- and $\cm$-invariant. Hence, by \cref{prop:idempotents}.(5), $E_{\circ}(A)$ is a  sub-\ourstructure. 
 \end{proof}
\end{proposition}

\begin{proposition}
Let  $(A,+,\circ,\lambda,\cm)$ be a  left cancellative \ourstructure. If the associated solution   $(A,r_A)$ is bijective then
\begin{enumerate}
    \item $|E_+(A)|\leq 1$, and 
    \item $(A,+)$ is  right cancellative.
\end{enumerate}
In particular, in this case, if  furthermore $A$ is finite, then $(A,+)$ is a group.
\end{proposition}
\begin{proof}
(1) Let $e,f\in E_+(A)$. Since $(A,+)$ is left cancellative, $e$ and $f$ also are  left identities. Moreover, since $r_A$ is bijective, we have that $\cm_e$ is bijective and thus  $\cm_e=\id$. It follows that
$
    e = f+e = e+\cm_e(f) = e+f = f
$.

(2) Let $a,b,d\in A$. Assume that $a+b = d+ b$. Equation \eqref{eq:csum} yields $b+\cm_b(a) = b+ \cm_b(d)$. Since $(A,+)$ is left cancellative, it follows that $\cm_b(a)=\cm_b(d)$. Hence, since $r_A$ is bijective, and thus  $\cm_b$ bijective,  $a= d$. Therefore $(A,+)$ is right cancellative.
\end{proof}

\begin{proposition}\label{prop:sumgroup}
    Let $(A,+,\circ, \lambda, \cm)$ be a \ourstructure. If $(A,+)$ is a group, then $(A,\circ)$ is a group. 
    \begin{proof}
    First note that \eqref{eq:sumcirc} and $(A,+)$ a group yield that $(A,\circ)$ is left cancellative. Let $0$ denote the identity of $(A,+)$. We claim that $0\in E_{\circ}(A)$ so in particular, $0$ is a left identity. First note that $\lambda_0(0) = \lambda_0(0+0) = \lambda_0(0) +\lambda_0(0)$, i.e. $\lambda_0(0)=0$. It yields $0\circ 0 \overset{\eqref{eq:sumcirc}}{=} 0+\lambda_0(0) = 0$ and thus  indeed  $0\in E_{\circ}(A)$. Similarly,  for $a\in A$ we have  $\lambda_a(0) = \lambda_a(0+0)= \lambda_a(0) + \lambda_a(0)$, i.e. $\lambda_a(0) = 0$. 
    It follows that $a\circ 0 = a + \lambda_a(0) = a + 0 = a$. Hence $0$ is also a right identity and $(A,\circ)$ is a monoid. Now set $b=\lambda^{-1}_a(-a)$, then we get $a\circ b = a\circ \lambda^{-1}_a(-a)\overset{\eqref{eq:sumcirc}}{=} a +(-a) = 0$. Therefore, any $a\in A$ has a right inverse in $(A,\circ)$. It follows that $(A,\circ)$ is a group.
    \end{proof}
\end{proposition}

\begin{corollary}
Let $(A,+,\circ, \lambda, \cm)$ be a \ourstructure. If $(A,+)$ is a group, then $(A,+,\circ)$ is a skew left brace and the associated solution $(A,r_A)$ is bijective non-degenerate.
\begin{proof}
 By \cref{prop:sumgroup}, $(A,\circ)$ is a group. Moreover, by \cref{exsb}, a \ourstructure with $(A,+)$ and $(A,\circ)$ groups is a skew left brace. 
\end{proof}
\end{corollary}

Recall that an inverse semigroup is a regular semigroup $S$ (i.e. for any $a\in S$ there exists $b\in S$ so that $aba=a$) such that any   two idempotents of $S$ commute \cite[Theorem 1.17]{ClPr}.

\begin{lemma} \label{invsemigroup}
Let $(A,+,\circ,\lambda, \cm)$ be a \ourstructure. If $a\in A$ is regular in $(A,\circ)$ then a is regular in $(A,+)$.
\begin{proof}
    Let $x\in A$ such that $a\circ x\circ a = a$. It follows
    \begin{align*}
        a&= a\circ x \circ a = a+\lambda_a(x+\lambda_x(a)) = a+ \lambda_a(x) + \lambda_a\lambda_x(a) \\
        &= a+\lambda_a(x)+\lambda_a\lambda_x\lambda_a\lambda^{-1}_a(a)\\
        &=a+\lambda_a(x)+\lambda_{a\circ x\circ a}\lambda^{-1}_a(a)\\
        &=a+\lambda_a(x)+\lambda_a\lambda^{-1}_a(a) &\text{(since $a\circ x \circ a = a$)}\\
        &=a+\lambda_a(x)+a.
    \end{align*}
    Hence $a+\lambda_a(x)+a= a$ and $a$ is regular in $(A,+)$.
\end{proof}
\end{lemma}

\begin{proposition}
Let $(A,+,\circ,\lambda,\cm)$ be a \ourstructure. If $E_{+}(A)=E_{\circ}(A)$ and $(A,\circ)$ is an inverse semigroup, then $(A,+)$ is an inverse semigroup.
\begin{proof}
Recall that $e_1\circ e_2 = e_1 + e_2$ when $e_1 \in E_{\circ}(A)$. Hence, if $(A,\circ)$ is an inverse semigroup then, for $e_1,e_2 \in E_{\circ}(A)$, we have $e_2+e_1 =e_2 \circ e_1 = e_1 \circ e_2 = e_1+e_2$. Hence  $E_{+}(A)$ is a commutative semigroup and thus, by Lemma~\ref{invsemigroup},  the result follows.
\end{proof}
\end{proposition}

\subsection{Matched product of solutions and \ourstructures}\label{sec_matchedproduct}\hfill

In this section we construct from two \ourstructures a new \ourstructure. In particular we give examples 
of \ourstructures that are not as in the examples in \cref{sec_leftcancellative}. We do this by defining a matched product of \ourstructures, a construction technique that generalizes the matched product of left cancellative semi-braces \cite{CCS2017}, of skew left braces \cite{SmVe18} and braces \cite{Ba18}. Note that if $B_1$ and $B_2$ are left cancellative semi-braces (or skew left braces or braces), then the multiplicative group of the matched product of $B_1$ and $B_2$ is the Zappa-Szép product of the groups $(B_1,\circ)$ and $(B_2,\circ)$ \cite{CaCoSt20a}, i.e. $(B_1,\circ)$ and $(B_2,\circ)$ form a matched pair of groups \cite{Ta03}. In the same manner, the multiplicative structure of the matched product of two unital \ourstructures forms a matched pair of monoids. Matched pairs of monoids were successfully used by Gateva-Ivanova and Majid in \cite{GIMa08} to construct new solutions. 
Moreover, we prove that the solution associated to the matched product of two \ourstructures is indeed the matched product of the solutions associated to those two \ourstructures.

\begin{definition}
Let $(A_1, +,\circ, \lambda^{(1)},\cm^{(1)})$, $(A_2, +,\circ, \lambda^{(2)},\cm^{(2)})$ be \ourstructures, $\alpha: (A_2,\circ)\to \Aut(A_1,+)$, $\beta: (A_1, \circ)\to \Aut(A_2,+)$ semigroups morphisms such that the following hold 
\begin{align}
    &\lambda^{(1)}_a\alpha_{\beta^{-1}_a(u)} =\alpha_{u}\lambda^{(1)}_{\alpha^{-1}_u(a)}, \qquad && \lambda^{(2)}_u\beta_{\alpha^{-1}_u(a)} =\beta_{a}\lambda^{(2)}_{\beta^{-1}_a(u)},\label{eq:alphabetalambda}\\
   &\alpha_u\cm^{(1)}_a = \cm^{(1)}_{\alpha_u(a)}\alpha_u, &&\beta_a\cm^{(2)}_u = \cm^{(2)}_{\beta_a(u)}\beta_a,\label{eq:alphabetac}
\end{align}
for all $a\in A_1$, $u\in A_2$. The quadruple $(A_1,A_2,\alpha,\beta)$ is called a \emph{matched product system of \ourstructures}.
\end{definition}

In order to show that the matched product of  a matched product system of \ourstructures is again a \ourstructure we need the following lemma.

\begin{lemma}\label{lemma:semigroup}
Let $(A,+)$ be a semigroup and $\lambda: A\to \End(A,+):a \mapsto \lambda_a$. Define $a\circ b = a+\lambda_a(b)$, for any $a,b\in A$. If $\lambda_{a+\lambda_a(b)}= \lambda_a\lambda_b$ for any $a,b\in A$, then $(A,+,\circ, \lambda)$ is a left semitruss.
\begin{proof} First we show that $(A,\circ)$ is associative and thus a semigroup.
For  $a,b,d \in A$ we have
\begin{align*}
    (a\circ b) \circ d &= (a + \lambda_a(b))\circ d
    = a+\lambda_a(b) +\lambda_{a+\lambda_a(b)}(d) \\
    &=a +\lambda_a(b)+\lambda_a\lambda_b(d)
    =a +\lambda_a(b+\lambda_b(d))\\
    &=a\circ (b\circ d).
\end{align*}
Hence, $(A,\circ)$ is a semigroup.
Furthermore,
\begin{align*}
    a\circ(b+d) = a +\lambda_a(b+d) = (a+ \lambda_a(b)) +\lambda_a(d) 
    =a\circ b + \lambda_a(d).
\end{align*}
Hence, $(A,+,\circ, \lambda)$ is a left semitruss.
\end{proof}
\end{lemma}

\begin{theorem}\label{thm:matchedproduct}
 Let $(A_1,A_2,\alpha,\beta)$ a matched product system of \ourstructures. Define, on the cartesian product $A_1\times A_2$, the following
 \begin{align}
     (a,u)+(b,v)&=(a+b,u+v),\label{eq:matchedsum}\\
     (a,u)\circ (b,v) &= (\alpha_u(\alpha^{-1}_u(a)\circ b),\beta_a(\beta^{-1}_a(u)\circ v)),\label{eq:matchedcirc}\\
     \lambda_{(a,u)}(b,v) &= (\lambda^{(1)}_a\alpha_{\beta^{-1}_a(u)}(b), \lambda^{(2)}_u\beta_{\alpha^{-1}_u(a)}(v)),\label{eq:matchedlambda}\\
     \cm_{(a,u)}(b,v)&= (\cm_a(b),\cm_u(v))\label{eq:matchedc}.
 \end{align}
 Then $(A_1\times A_2, +,\circ, \lambda, \cm)$ is a \ourstructure, called \emph{the matched product of $A_1$ and $A_2$ (via $\alpha$ and $\beta$)} and denoted by $A_1\bowtie A_2$.
 \begin{proof}
    Let $a,b,d \in A_1$ and $u,v,w\in A_2$. 
    Denote by $\bar{a}=\alpha^{-1}_u(a)$, $\bar{u} = \beta^{-1}_a(u)$, $\bar{b}=\alpha^{-1}_v(b)$, $\bar{v} = \beta^{-1}_b(v)$, 
    First note that by \eqref{eq:alphabetalambda}, \eqref{eq:sumcirc} and \eqref{eq:matchedsum} we get
    \begin{align*}
        (a,u)+\lambda_{(a,u)}(b,v) 
        &= (a+\lambda^{(1)}_a\alpha_{\bar{u}}(b), u + \lambda^{(2)}_u\beta_{\bar{a}}(v))=(a+\alpha_u\lambda^{(1)}_{\bar{a}}(b), u+\beta_a\lambda^{(2)}_{\bar{u}}(v))\\
        &=(\alpha_u(\alpha^{-1}_u(a)+ \lambda^{(1)}_{\bar{a}}(b)), \beta_a(\beta_a^{-1}(u)+\lambda^{(2)}_{\bar{u}}(v)))\\
        &=(\alpha_u(\bar{a}\circ b), \beta_a(\bar{u}\circ v))
        = (a,u)\circ (b,v).
    \end{align*}
    Hence, $A_1\times A_2$ satisfies condition \eqref{eq:sumcirc}. Moreover, it is a routine computation to verify that $\lambda_{(a,u)}(b,v)+\lambda_{(a,u)}(d,w) = \lambda_{(a,u)}((b,v)+(d,w))$, i.e.
    $\lambda:A_1\times A_2\to \End(A_1\times A_2, +)$.
    Furthermore, note that $\alpha_u(\bar{a}\circ b) =a\circ \alpha_{\bar{u}}(b)$ and $\beta^{-1}_{\alpha_{\bar{u}}(b)}(\bar{u}\circ v)= \beta^{-1}_{\alpha_{\bar{u}}(b)}(\bar{u})\circ \beta^{-1}_{b}(v)$, that yields
    \begin{align*}
        \lambda^{(1)}_{\alpha_u(\bar{a}\circ b)}\alpha_{\beta^{-1}_{\alpha_u(\bar{a}\circ b)}\beta_a(\bar{u}\circ v)}(d) 
        &=\lambda^{(1)}_{a\circ \alpha_{\bar{u}}(b)} \alpha_{\beta^{-1}_{a\circ \alpha_{\bar{u}}(b)}\beta_a(\bar{u}\circ v)}(d)\\
        &= \lambda^{(1)}_{a}\lambda^{(1)}_{\alpha_{\bar{u}}(b)}
        \alpha_{\beta^{-1}_{\alpha_{\bar{u}}(b)}(\bar{u}\circ v)}(d)\\
        &=\lambda^{(1)}_{a}\lambda^{(1)}_{\alpha_{\bar{u}}(b)}\alpha_{\beta^{-1}_{\alpha_{\bar{u}}(b)}(\bar{u})\circ \beta^{-1}_{b}(v)}(d)\\
        &=\lambda^{(1)}_{a}\lambda^{(1)}_{\alpha_{\bar{u}}(b)}\alpha_{\beta^{-1}_{\alpha_{\bar{u}}(b)}(\bar{u})}\alpha_{ \beta^{-1}_{b}(v)}(d)\\
        &=\lambda^{(1)}_{a}\alpha_{\bar{u}}\lambda^{(1)}_{\alpha^{-1}_{\bar{u}}\alpha_{\bar{u}}(b)}\alpha_{ \bar{v}}(d)\\
        &=\lambda^{(1)}_{a}\alpha_{\bar{u}}\lambda^{(1)}_{b}\alpha_{\bar{v}}(d),
    \end{align*}
    In the same way, we get $\lambda^{(2)}_{\beta_a(\bar{u}\circ v)}\beta_{\alpha^{-1}_{\beta_a(\bar{u}\circ v)}\alpha_u(\bar{a}\circ b)}(w) = \lambda^{(2)}_u\beta_{\bar{a}}\lambda^{(2)}_v\beta_{\bar{b}}(w)$.
    It follows $\lambda_{(a,u)\circ (b,v)}(d,w) 
        =\lambda_{(\alpha_u(\bar{a}\circ b), \beta_a(\bar{u}\circ v))}(d,w)=\lambda_{(a,u)}\lambda_{(b,v)}(d,w)$. Hence, \cref{lemma:semigroup}  yields that  $(A_1\times A_2, +,\circ, \lambda)$ is a left semitruss. Moreover, by \eqref{eq:matchedc} and \eqref{eq:csum}, we obtain that 
    $A_1\times A_2$ satisfies \eqref{eq:csum}. In addition, condition \eqref{eq:condc} and \eqref{eq:clambda} for $A_1\times A_2$ follow from the respective conditions in  $A_1$ and $A_2$ and from \eqref{eq:matchedc} and \eqref{eq:alphabetac}. 
    Since $\lambda_{(a,u)}$ is bijective with inverse given by $\lambda^{-1}_{(a,u)}(b,v) = (\alpha^{-1}_{\beta^{-1}_{a}(u)}{\lambda_a^{(1)}}^{-1}(b), \beta^{-1}_{\alpha^{-1}_u(a)}{\lambda_u^{(2)}}^{-1}(v))$, we conclude that $(A_1\times A_2,+,\circ, \lambda, \cm)$ is a \ourstructure.
 \end{proof}
\end{theorem}

Via the matched product we now give an example of a \ourstructure that is not a left cancellative semi-brace, and where the additive and multiplicative operations are different.

\begin{example}\label{MoreThenOne}
   Let $A=\{1,\ldots,n\}$ and $B=C_n=\langle \xi \rangle$ the cyclic group of $n$-elements.
   Let $a+b =a\circ b = b$, $\lambda_a=\id_A$ and $\cm_a(b)=a$, for any $a,b\in A$, then $(A,+,\circ, \lambda, \cm)$ is a \ourstructure, and let $B$ the trivial skew left brace with $(B,+)=(B,\circ)=C_n$, i.e. a \ourstructure with $\lambda_x=\cm_x=\id_B$.
   Consider the natural action of $C_n$ over $A$: $\alpha: B \to A$ defined by $\alpha(\xi^{i})(j) = (1\ 2 \ \ldots n)^{i}(j)$.
   Note that since $\Aut(A,+)=\Sym(A)$ we have that $\alpha:(B,\circ)\to \Aut(A,+)$. It is easy to check that setting $\beta_a=\id$ for any $a\in A$ we get  that $(A,B,\alpha, \beta)$ is a matched product system of \ourstructures.
\end{example}

As in the semi-braces' case we establish a connection between the matched product of \ourstructures and the matched product of their solutions.

For completeness' sake we recall the definition of matched product of solutions given in \cite{CaCoSt20a}.
Let $(S,r_S)$ and $(T,r_T)$ be solutions, $\alpha: T \to \Sym\left(S\right)$ and $\beta: S \to \Sym\left(T\right)$ maps. Put $\alpha_u=\alpha\left(u\right)$, for every $u \in T$, and $\beta_a =\beta\left(a\right)$, for every $a\in S$, then the quadruple
	$\left(r_S,r_T, \alpha,\beta\right)$ is said to be a \emph{matched product system of solutions} if the following conditions hold
	
	{\scriptsize
		\begin{center}
			\begin{minipage}[b]{.5\textwidth}
				\vspace{-\baselineskip}
				\begin{align}\label{eq:s1}\tag{s1}
					\alpha_u\alpha_v = \alpha_{\lambda_u\left(v\right)}\alpha_{\rho_{v}\left(u\right)},
				\end{align}
			\end{minipage}%
			\hfill\hfill\hfill
			\begin{minipage}[b]{.5\textwidth}
				\vspace{-\baselineskip}
				\begin{align}\label{eq:s2}\tag{s2}
					\beta_a\beta_b=\beta_{\lambda_a\left(b\right)}\beta_{\rho_b\left(a\right)},
				\end{align}
			\end{minipage}
		\end{center}

		\begin{center}
			\begin{minipage}[b]{.5\textwidth}
				\vspace{-\baselineskip}
				\begin{align}\label{eq:s3}\tag{s3}
					\rho_{\alpha^{-1}_u\!\left(b\right)}\alpha^{-1}_{\beta_a\left(u\right)}\left(a\right) = \alpha^{-1}_{\beta_{\rho_b\left(a\right)}\beta^{-1}_b\left(u\right)}\rho_b\left(a\right),
				\end{align}
			\end{minipage}%
			\hfill\hfill
			\begin{minipage}[b]{.5\textwidth}
				\vspace{-\baselineskip}
				\begin{align}\label{eq:s4}\tag{s4}
					\rho_{\beta^{-1}_a\!\left(v\right)}\beta^{-1}_{\alpha_u\left(a\right)}\left(u\right) = \beta^{-1}_{\alpha_{\rho_v\left(u\right)}\alpha^{-1}_v\left(a\right)}\rho_v\left(u\right),
				\end{align}
			\end{minipage}
		\end{center}
	
		\begin{center}
			\begin{minipage}[b]{.5\textwidth}
				\vspace{-\baselineskip}
				\begin{align}\label{eq:s5}\tag{s5}
					\lambda_a\alpha_{\beta^{-1}_{a}{\left(u\right)}}{}= \alpha_{u}{\lambda_{\alpha^{-1}_{u}{\left(a\right)}}{}},
				\end{align}
			\end{minipage}%
			\hfill\hfill
			\begin{minipage}[b]{.5\textwidth}
				\vspace{-\baselineskip}
				\begin{align}\label{eq:s6}\tag{s6}
					\lambda_{u}{\beta_{\alpha^{-1}_{u}{\left(a\right)}}{}}=\beta_{a}{\lambda_{\beta^{-1}_{a}{\left(u\right)}}{}},
				\end{align}
			\end{minipage}
		\end{center}
}
\noindent for all $a,b \in S$ and $u,v \in T$.
	
	As shown in {\cite[Theorem 1]{CaCoSt20a}}, any matched product system of solutions determines a new solution on the set $S\times T$. Specifically, if $\left(r_S,r_T, \alpha,\beta\right)$ is a matched product system of solutions, then the map $r:S{ \times} T\times S{ \times} T \to S{ \times} T\times S{ \times} T$ defined by
		\begin{align*}
			&r\left(\left(a, u\right), 
			\left(b, v\right)\right) = 
			\left(\left(\alpha_{u}{\lambda_{\bar{a}}{\left(b\right)}},\, \beta_{a}{\lambda_{\bar{u}}{\left(v\right)}}\right),\ \left(\alpha^{-1}_{\overline{U}}{\rho_{\alpha_{\bar{u}}{\left(b\right)}}{\left(a\right)}},\,  \beta^{-1}_{\overline{A}}{\rho_{\beta_{\bar{a}}{\left(v\right)}}{\left(u\right)}} \right) \right),
		\end{align*}
		where
		\begin{align*}
			\bar{a}=\alpha^{-1}_{u}{\left(a\right)}, \ \bar{u}= \beta^{-1}_{a}{\left(u\right)},\  
			A=\alpha_{u}\lambda_{\bar{a}}{\left(b\right)}, \ U=\beta_{a} \lambda_{\bar{u}}{\left(v\right)}, \  \overline{A}=\alpha^{-1}_{U}{\left(A\right)}, \ \overline{U}= \beta^{-1}_{A}{\left(U\right)},
     \end{align*}
		for all $\left(a,u\right),\left(b,v\right)\in S\times T$, is a solution. This solution is called the \emph{matched product of the solutions} $r_S$ and $r_T$ (via $\alpha$ and $\beta$) and it is denoted by $r_S\bowtie r_T$.

It is a routine computation to prove the following result (see for example \cite[Theorem 9]{CaCoSt20a} \cite[Theorem 6]{CaCoSt20b} in the context of (left cancellative) semi-braces).
\begin{proposition}
    Let $(A_1,A_2,\alpha,\beta)$ a matched product system of \ourstructures. Then $r_{A_1\bowtie A_2}= r_{A_1}\bowtie r_{A_2}$.
\end{proposition}

\section{\ourstructures with associated  bijective solutions}\label{sec_bijective}

The aim of this section is to prove the following result.

\begin{theorem}\label{maintheorembijective}
If $(X,r)$ is a finite left non-degenerate solution, then  $r$ is bijective
if and only if $(X,r)$ is right non-degenerate.
\end{theorem}

Castelli, Catino and Stefanelli in \cite{CaCaSt21} have proven the necessity of the condition. Their proof is  based on the notion of  $q$-cycle sets introduced by Rump in \cite{Ru19}.  We  will prove the sufficiency as well and we reprove the necessity in terms of  \ourstructures. 

Recall that the solution $(A,r_A)$ associated to a \ourstructure $A$ is always left non-degenerate. If the solution  also is right non-degenerate, we call the \ourstructure a {\it non-degenerate \ourstructure}. An example is the structure \ourstructure $S(X,r)$ associated to a non-degenerate solution $(X,r)$.

Let
  $$\mathfrak{q}: A \rightarrow  A : a\mapsto \lambda_a^{-1}(a) ,$$ 
it is called the \emph{diagonal map} on a \ourstructure $A$.

In \cite[Corollary 2 of Proposition 8]{Ru19}, Rump proved that for any bijective solution $(X,r)$,  $\mathfrak{q}$ is bijective if and only if $(X,r)$ is non-degenerate. In the following lemma  we show that one does  not  need the bijectiveness of $r_A$ to prove that $\mathfrak{q}$ is injective. So for any $\mathbb{N}$-graded  non-degenerate \ourstructure with all homogeneous components $A_n$ finite, we get that $\mathfrak{q}$ is bijective.

\begin{lemma}\label{lemma:diagonal}
    Let $(A,+,\circ, \lambda, \cm)$ be a non-degenerate \ourstructure. Then the diagonal map $\mathfrak{q}$ is injective.
    \begin{proof}
        Let $a,d\in A$ and assume that  $\mathfrak{q}(a) =\mathfrak{q}(d)$. Put  say $t=\mathfrak{q}(a)$. It follows that
        \begin{align*}
            \rho_t(a)&= \lambda^{-1}_{\lambda_a(t)}\cm_{\lambda_a(t)}(a)
            =\lambda^{-1}_{\lambda_a\lambda^{-1}_a(a)}\cm_{\lambda_a(t)}(a)
            \\ &= \lambda^{-1}_a\cm_{\lambda_a(t)}(a)
            =\cm_{\lambda^{-1}_a\lambda_a(t)}\lambda^{-1}_a(a)
            = \cm_t(t).
        \end{align*}
        Similarly,  $\rho_t(d) = \cm_t(t)$. Hence, $\rho_t(a)=\rho_t(d)$.
        Since, by assumption,  $\rho_t$ is bijective,  it follows that $a=d$. So,  $\mathfrak{q}$ is injective.
    \end{proof}
\end{lemma}

Let $X$ be a set with $ |X|\geq 2$. Consider the idempotent solution on $X$ given by $r(x,y) = (y,y)$ (see Example~\ref{exidempotent}). Clearly $r$ is neither bijective nor right non-degenerate. However,
the solution is left non-degenerate (as each $\lambda_x=\id_X$) and  
$\mathfrak{q}$ is bijective (since $\mathfrak{q}(x)=\lambda^{-1}_x(x) = x$, for any $x\in X$). This shows that for an arbitrary solution $(X,r)$ one can have $\mathfrak{q}$ bijective even if $r$ is neither  bijective nor right non-degenerate; and thus the bijectiveness assumption is essential in Rump's result.

\begin{lemma} \label{prop:bijdiagonal}
    Let $(A,+,\circ, \lambda,\cm)$ be a non-degenerate \ourstructure. If the diagonal map $\mathfrak{q}$ is bijective, then  $r_A$ is bijective. 
    \begin{proof}
    By \cref{prop:semitruss}, $r_A$ is bijective if and only if $\cm_a$ is bijective for any $a\in A$. From \eqref{eq:clambda} and \eqref{eq:YBE1} we get that, for $a,b\in A$, 
    \begin{align*}
        \rho_b(a) = \lambda^{-1}_{\lambda_a(b)}\cm_{\lambda_a(b)}(a) = \lambda^{-1}_{\lambda_a(b)}\lambda_a\cm_b\lambda^{-1}_a(a)= \lambda_{\rho_b(a)}\lambda^{-1}_b\cm_b\mathfrak{q}(a).
    \end{align*}
    Since, by assumption,  $\rho_b$ is bijective it yields $\lambda_a\lambda^{-1}_b\cm_b\mathfrak{q}\rho^{-1}_b(a) =a$ and thus
    \begin{align*}
        \mathfrak{q}(a) = \lambda^{-1}_b\cm_b\mathfrak{q}\rho^{-1}_b(a).
    \end{align*}
    Now, assume $\mathfrak{q}$ is bijective. Then
    \begin{align*}
        \cm_b(a) = \lambda_b\mathfrak{q}\rho_b\mathfrak{q}^{-1}(a),
    \end{align*}
    i.e. $\cm_b$ is bijective. So $r_A$ is bijective.
    \end{proof}
\end{lemma}

To prove the main result of this section we need one more lemma.

\begin{lemma}\label{lem:rewritediag}
 Let $(A,+,\circ,\lambda,\cm)$ be a \ourstructure with $\cm_a$ bijective for any $a\in A$ and let $\mathfrak{q}$ be the diagonal map. For any $a,b\in A$, it holds
 \begin{align*}
     \lambda^{-1}_b\mathfrak{q}(a)= \mathfrak{q}(\lambda^{-1}_{\cm^{-1}_a\lambda_a(b)}(a)).
 \end{align*}
\begin{proof}
By \eqref{eq:YBE1}, 
we get
$$\lambda^{-1}_{\rho_{\lambda^{-1}_a(b)}(a)}\lambda^{-1}_b = \lambda^{-1}_{\rho_{\lambda^{-1}_a(b)}(a)}\lambda^{-1}_{\lambda_a(\lambda^{-1}_a(b))} = \lambda^{-1}_{\lambda^{-1}_a(b)}\lambda^{-1}_a.$$
Next, put $a=\cm^{-1}_b\lambda_b(x)$, to obtain
\begin{equation}\label{eq:rho}
\lambda^{-1}_{\rho_{\lambda^{-1}_{\cm^{-1}_b\lambda_b(x)}(b)}(\cm^{-1}_b\lambda_b(x))}\lambda^{-1}_b(b) =  \lambda^{-1}_{\lambda^{-1}_{\cm^{-1}_b\lambda_b(x)}(b)}\lambda^{-1}_{\cm^{-1}_b\lambda_b(x)} (b).
\end{equation}
By definition, $\rho_b(a) = \lambda^{-1}_{\lambda_a(b)}\cm_{\lambda_a(b)}(a)$, and thus $$\rho_{\lambda^{-1}_{\cm^{-1}_b\lambda_b(x)}(b)}(\cm^{-1}_b\lambda_b(x)) = \lambda^{-1}_b \cm_b(\cm^{-1}_b\lambda_b(x)) = x.$$
Hence, \eqref{eq:rho} is equivalent to
$$\lambda^{-1}_{x}\mathfrak{q}(b) = \mathfrak{q}(\lambda^{-1}_{\cm^{-1}_b\lambda_b(x)}(b)),$$
as desired.
\end{proof}
\end{lemma}

The sufficiency of the next proposition was proven by Castelli, Catino and Stefanelli in \cite{CaCaSt21}. We translated their proof in the language of \ourstructures and added it for completeness' sake.

\begin{proposition}\label{bijectivenondeg1}
    Let $(A,+,\circ, \lambda,\cm)$ be an $\mathbb{N}$-graded  \ourstructure with all homogeneous components $A_n$ finite and with  $(A,r_A)$
    its associated left non-degenerate solution.
    Then,  $(A,r_A)$ is right non-degenerate if and only if $r_A$ is bijective.
\end{proposition}
\begin{proof}  Assume first that the solution $(A,r_A)$ is non-degenerate, i.e. $(A,+,\circ, \lambda, \cm)$ is a non-degenerate \ourstructure. By Lemma \ref{lemma:diagonal} the diagonal map $\mathfrak{q}$ is injective. As $\mathfrak{q}$ is degree preserving and $A$ is an $\mathbb{N}$-graded  \ourstructure with all homogeneous components $A_n$ finite, it follows that $\mathfrak{q}$ is bijective. By  \cref{prop:bijdiagonal}, we conclude that $r_A$ is bijective.

To prove the converse, assume that $(A,r_A)$ is bijective (i.e. $\cm_a$ is bijective for all $a\in A$). First, we prove that the diagonal map $\mathfrak{q}$ is bijective.
Let $n\in \mathbb{N}$. Without loss of generality, we focus on the restrictions of $\lambda,\rho,\cm$ to $A_n$ for a fixed $n\in \mathbb{N}$.
Since $A_n$ is finite, there exists $m\in \mathbb{N}$ such that $(\lambda^{-1}_z)^{m}=\id_{A_n}$ for any $z\in A_n$. Now, if $m=1$, then $\mathfrak{q}=\id_{A_n}$ and in particular $\mathfrak{q}$ is bijective. If $m\geq 2$, then by \cref{lem:rewritediag}  for $a\in A_n$,
\begin{align*}
    a=(\lambda_a^{-1})^{m}(a)&=(\lambda^{-1}_a)^{m-2}\lambda^{-1}_a\mathfrak{q}(a)\\
    &=(\lambda^{-1}_a)^{m-2}\mathfrak{q}\lambda^{-1}_{\cm_a^{-1}\lambda_a(a)}(a)\\
    &=(\lambda^{-1}_a)^{m-3}\lambda^{-1}_a\mathfrak{q}\lambda^{-1}_{\cm_a^{-1}\lambda_a(a)}(a)\\
    &=(\lambda^{-1}_a)^{m-3}\mathfrak{q}\lambda_{\cm^{-1}_{\lambda^{-1}_{\cm_a^{-1}\lambda_a(a)}(a)}\lambda_{\lambda^{-1}_{\cm_a^{-1}\lambda_a(a)}(a)}(a)}\lambda^{-1}_{\cm_a^{-1}\lambda_a(a)}(a),
\end{align*}
and so on. So there exists $t\in A_n$ such that $a=\mathfrak{q}(t)$. Therefore $\mathfrak{q}$ is surjective and since $A_n$ is finite, $\mathfrak{q}$ is bijective. Now we will prove that $(A,r_A)$ is right non-degenerate. As the $\rho$-map is degree preserving, it is enough to prove that $\rho_a$ is injective, for all $a\in A$. Assume $\rho_a(b) = \rho_a(d)$, i.e. $\lambda^{-1}_{\lambda_b(a)}\cm_{\lambda_b(a)}(b) = \lambda^{-1}_{\lambda_d(a)}\cm_{\lambda_d(a)}(d)$. Therefore \eqref{eq:clambda} yields
\begin{align*}
     \lambda^{-1}_{\lambda_b(a)}\lambda_b \cm_a\mathfrak{q}(b)=\lambda^{-1}_{\lambda_b(a)}\lambda_b \cm_a\lambda^{-1}_b(b)=\lambda^{-1}_{\lambda_d(a)}\lambda_d \cm_a\lambda^{-1}_d(d)=\lambda^{-1}_{\lambda_d(a)}\lambda_d\cm_a\mathfrak{q}(d).
\end{align*}
Hence, since
$\lambda^{-1}_{\lambda_b(a)}\lambda_b =\lambda_{\rho_a(b)}\lambda^{-1}_a=\lambda_{\rho_a(d)}\lambda^{-1}_a=\lambda^{-1}_{\lambda_d(a)}\lambda_d$
we get $\cm_a\mathfrak{q}(b)=\cm_a\mathfrak{q}(d)$. As $\cm_a$ is bijective, this implies $\mathfrak{q}(b)=\mathfrak{q}(d)$, i.e. $b=d$ and thus $\rho_a$ is injective.
\end{proof}

We are now in a position to prove the main result of this section.

\begin{proof}[Proof of Theorem~\ref{maintheorembijective}]
Let $(X,r)$ be a finite left non-degenerate solution.
From Proposition \ref{SolFromSemiTruss}, we know that we can associate to $(X,r)$ its unital structure \ourstructure $M=(M(X,r),+,\circ,\lambda,\cm)$ and a left non-degenerate solution $(M,r_M)$. Furthermore, $M$ is an $\mathbb{N}$-graded \ourstructure with all homogeneous components finite. From \cref{bijectivenondeg1} it follows that $(M,r_M)$ is right non-degenerate if and only if $r_M$ is bijective. Since $(M,r_M)$ being right non-degenerate is equivalent to $(X,r)$ being right non-degenerate, and $r_M$ being bijective is equivalent to $r$ being bijective, we get that $(X,r)$ is right non-degenerate if and only if $r$ is bijective.
\end{proof}

In Lemma~\ref{prop:bijdiagonal} it was shown that the bijectiveness of $\mathfrak{q}$ is an important property to prove that $r$ is bijective. In \cite[Lemma  4.4]{CJV2020} the bijectiveness of $\mathfrak{q}$ is proven in case the solution is irretractable. The following proposition slightly extends this result and the proof follows the same lines.

\begin{proposition}
  Let $(A,+,\circ, \lambda, \cm)$ be a non-degenerate \ourstructure. 
 Assume the following property holds for all $a,b\in A$:
 $(\lambda_a, \rho_a)=(\lambda_b, \rho_b)$ implies $a=b$. Then the  diagonal map $\mathfrak{q}$ is bijective and its inverse is given by $a\mapsto \rho^{-1}_a(a)$.
\end{proposition}
\begin{proof}

By \eqref{eq:YBE1},
\begin{align*}
 \lambda_{\lambda_{\rho^{-1}_a(a)}(a)} \lambda_a
&=\lambda_{\lambda_{\rho^{-1}_a(a)}(a)}\lambda_{\rho_a(\rho^{-1}_a(a))}
=\lambda_{\rho^{-1}_a(a)}\lambda_{a}.
\end{align*}
Because $\lambda_a$ is bijective, we thus get
\begin{align}\label{eq:irr1}
\lambda_{\lambda_{\rho^{-1}_a(a)}(a)}
&=\lambda_{\rho^{-1}_a(a)}.
\end{align}
Furthermore, by \eqref{eq:YBE3},
\begin{align*}
\rho_a \rho_{\lambda_{\rho^{-1}_a(a)}(a)} 
&=\rho_{\rho_a(\rho^{-1}_a(a))}\rho_{\lambda_{\rho^{-1}_a(a)}(a)}
=\rho_{a}\rho_{\rho^{-1}_a(a)}.
\end{align*}
Thus, since, by assumption, $\rho_a$ is bijective, 
\begin{align}\label{eq:irr2}
\rho_{\lambda_{\rho^{-1}_a(a)}(a)}
&=\rho_{\rho^{-1}_a(a)}.
\end{align}
By \eqref{eq:irr1}, \eqref{eq:irr2} and the assumption we get
$
\lambda_{\rho^{-1}_a(a)}(a)
=\rho^{-1}_a(a),
$ and therefore
$a =\lambda^{-1}_{\rho^{-1}_a(a)}(\rho^{-1}_a(a))=\mathfrak{q}(\rho^{-1}_a(a))$.
It follows that the diagonal map $\mathfrak{q}$ is bijective and its inverse is given by $a\mapsto \rho^{-1}_a(a)$.
\end{proof}

\section{Non-degenerate \ourstructures and the retract relation}\label{sec_retraction}

In the previous section we showed that non-degenerate solutions on a finite set are bijective. In this section we prove this for a much wider class of non-degenerate solutions. To do so we introduce, in analogy with skew left braces (see for example \cite{ESS99, LeVe2019}), i.e. for non-degenerate bijective solutions (see also \cite{CJKVAV2021} for various equivalent definitions), the retract of a non-degenerate \ourstructure. 

Let $(A,+,\circ , \lambda, \cm )$ be a \ourstructure. Consider the following set
   $$\mathcal{G}(A)=\{ f(a):=(\cm_a , \lambda_a , \rho_a ) \mid a\in A\}.$$

\begin{lemma}\label{PermSem}
Let $a,a',b,b'\in A$.
\begin{enumerate}
    \item  If $f(a)=f(a')$ then $\cm_{b\circ a} =\cm_{b\circ a'}$, $\lambda_{a\circ b}=\lambda_{a'\circ b}$,  $\rho_{b\circ a}=\rho_{b\circ a'} $, and $\lambda_{a+b}=\lambda_{a'+b}$, $\rho_{a+b}=\rho_{a'+b}$, $\cm_{a+b}=\cm_{a'+b}$.
    \item if  $f(b)=f(b')$ then
$\lambda_{a\circ b} =\lambda_{a\circ b'}$, $\rho_{b\circ a} =\rho_{b'\circ a}$, $\cm_{b\circ a} =\cm_{b'\circ a}$, and
$\lambda_{a+ b} =\lambda_{a+ b'}$, $\rho_{a+b} =\rho_{a+ b'}$, $\cm_{a+ b} =\cm_{a+ b'}$.
\end{enumerate}
\end{lemma}
\begin{proof}
(1)
We have  that $\lambda_{a\circ b}=\lambda_{a}\lambda_{b}=\lambda_{a'}\lambda_b=\lambda_{a'\circ b}$ and similarly
$\rho_{b\circ a}=\rho_{b\circ a'}$. 
Further, $\cm_{\lambda_b(a)}=\lambda_b \cm_a\lambda^{-1}_b = \lambda_b \cm_{a'}\lambda^{-1}_b =\cm_{\lambda_b(a')}$ yields $$\cm_{b\circ a} = \cm_{b+\lambda_b(a)} = \cm_{\lambda_b(a)}\cm_b = \cm_{\lambda_b(a')}\cm_{b} =\cm_{b\circ a'}.$$
Now,
$\lambda_{a+b} = \lambda_{a\circ\lambda^{-1}_a(b)}= \lambda_a\lambda_{\lambda^{-1}_a(b)} = \lambda_{a'}\lambda_{\lambda^{-1}_{a'}(b)}=\lambda_{a'+b}$, 
$\rho_{a+b}=\rho_{a'+b}$ and $\cm_{a+b} =\cm_b\cm_a=\cm_b\cm_{a'} = \cm_{a'+b}$. 

(2)
We have  $\lambda_{a\circ b} =\lambda_{a\circ b'}$, $\rho_{b\circ a}=\rho_{b'\circ a}$. 
Furthermore,  $$\cm_{b\circ a} = \cm_{b +\lambda_b(a)} = \cm_{\lambda_b(a)}\cm_b = \cm_{\lambda_{b'}(a)}\cm_{b'} =\cm_{b'\circ a}.$$
Moreover, $a+b = b+\cm_b(a) = b\circ\lambda_b \cm_b(a)$, then $\lambda_{a+b} = \lambda_{b\circ\lambda_b \cm_b(a)} = \lambda_b\lambda_{\lambda_b \cm_b(a)} =\lambda_{b'}\lambda_{\lambda_{b'} \cm_{b'}(a)} = \lambda_{a+b'}$, similarly $\rho_{a+b} =\rho_{a+b'}$. Finally, $\cm_{a+b}=\cm_b\cm_a =\cm_{b'}\cm_{a} = \cm_{a+b'}$.
\end{proof}

As a consequence of Lemma~\ref{PermSem} we have the following (well-defined) additive and multiplicative operations on $\mathcal{G}(A)$:
$$f(a)+f(b) := (\cm_{a+b}, \lambda_{a+b},\rho_{a+b})=f(a+b),$$
and
$$f(a)\circ f(b):= (\cm_{a\circ b},\lambda_{a\circ b},\rho_{a\circ b})=f(a\circ b).$$
Hence,
  $$f:A \rightarrow \mathcal{G}(A): a \mapsto f(a)=(\cm_a, \lambda_a , \rho_a),$$
is a semigroup homomorphism for both the additive and multiplicative operations.

We will show that   $\mathcal{G}(A)$ is a \ourstructure provided $A$ is non-degenerate.
First we introduce a $\lambda$-map on $\mathcal{G}(A)$. 

\begin{lemma}\label{GammaSurjective}
The mapping $\lambda :\mathcal{G}(A) \rightarrow \Map (\mathcal{G}(A), \mathcal{G}(A)): a\mapsto \lambda_{f(a)}$
given by $\lambda_{f(a)} (f(b)) = f(\lambda_a (b))$, is well-defined.
Furthermore, each $\lambda_{f(a)}$ is surjective.
\end{lemma}
\begin{proof}
To prove that the map  $\lambda$ is well-defined,
we  first  note that  if $f(a)=f(a')$ then $\lambda_{f(a)} =\lambda_{f(a')}$. 
Next assume $f(b)=f(b')$. Then $\lambda_{f(a)}(f(b)) = f(\lambda_a(b))$ and $\lambda_{f(a)}(f(b'))=f(\lambda_a(b'))$.

By \eqref{eq:YBE1},  $\lambda_{\lambda_a(b)} = \lambda_a\lambda_b\lambda^{-1}_{\rho_b(a)} = \lambda_a\lambda_{b'}\lambda^{-1}_{\rho_{b'}(a)} = \lambda_{\lambda_a(b')}$. By \eqref{eq:clambda}  $\cm_{\lambda_a(b)} =\lambda_a \cm_b \lambda^{-1}_a = \lambda_a \cm_{b'} \lambda^{-1}_a = \cm_{\lambda_a(b')}$
and by \eqref{eq:clambda} and \eqref{eq:YBE1}, for all $x\in A$,
\begin{align*}
\rho_{\lambda_a(b)}(x)
&= \lambda^{-1}_{\lambda_x\lambda_a(b)}\cm_{\lambda_x\lambda_a(b)}(x) \\
&= \lambda^{-1}_{\lambda_{x\circ a}(b)}\lambda_x\lambda_a\cm_{b}\lambda^{-1}_a\lambda_x^{-1}(x) 
\\&= \lambda_{\rho_b(x\circ a)}\lambda^{-1}_b\lambda^{-1}_{x\circ a}\lambda_x\lambda_a\cm_{b}\lambda^{-1}_a\lambda_x^{-1}(x)
\\&=\lambda_{\rho_{b'}(x\circ a)}\lambda^{-1}_{b'}\lambda^{-1}_{x\circ a}\lambda_x\lambda_a\cm_{b'}\lambda^{-1}_a\lambda_x^{-1}(x)\\& = \rho_{\lambda_a(b')}(x).
\end{align*}
This shows  that $\lambda_{f(a)}(f(b))=\lambda_{f(a)}(f(b'))$ and thus   the map $\lambda$  is  well-defined.

To prove that the map $\lambda_{f(a)}$ is surjective, let   $f(d) \in \im(f)$. Then there exists $b \in A$ such that $\lambda_a(b) = d$. Hence,
$\lambda_{f(a)}(f(b)) = (\cm_{\lambda_a(b)},\lambda_{\lambda_{a}(b)},\rho_{\lambda_{a}(b)}) = (\cm_d,\lambda_d,\rho_d) = f(d)$,
as desired. 
\end{proof}

Next we introduce a $\cm$-map in $\mathcal{G}(A)$ as follows
 $$\cm: \mathcal{G}(A) \rightarrow \Map (\mathcal{G}(A), \mathcal{G}(A)): f(b)\mapsto \cm_{f(b)},$$
 with $$\cm_{f(b)}(f(a)) = f(\cm_{b}(a)).$$
The well-definedness of this map will follow from  Lemma~\ref{ManyClaims}, but first we need another proposition.

\begin{proposition} \label{prop:c_a=c_b}
Let $(A,+,\circ, \lambda , \cm)$ be a non-degenerate \ourstructure.
Let $a,b\in A$. If  $\lambda_a=\lambda_b$ and $\rho_a=\rho_b$ then $\cm_a=\cm_b$.
Furthermore, if $\lambda_a =\rho_a =\id$ then $\cm_a=\id$.
\begin{proof}
First note that, for $x,y,z\in A$, because of Proposition~\ref{prop:semitruss} and \eqref{eq:YBE2},
\begin{align*}
    \lambda_{\rho_{\lambda_y\lambda^{-1}_x(z)}\rho^{-1}_y(x)}\rho_{\lambda^{-1}_x(z)}(y)
    &=
    \rho_{\lambda_{\rho_y\rho^{-1}_y(x)}\lambda^{-1}_x(z)}\lambda_{\rho^{-1}_y(x)}(y) \\
    &=\rho_z\lambda_{\rho^{-1}_y(x)}(y),
\end{align*}
yields 
\begin{align}\label{claim1}
    \rho_{\lambda^{-1}_x(z)}(y) =\lambda^{-1}_{\rho_{\lambda_y\lambda^{-1}_x(z)}\rho^{-1}_y(x)}\rho_z\lambda_{\rho^{-1}_y(x)}(y).
\end{align}

Let $a,b\in A$ and assume  $\lambda_a=\lambda_b$ and $\rho_a=\rho_b$. For any $x\in A$, using \eqref{claim1} it follows that
\begin{align*}
    \rho_{\lambda^{-1}_x(a)}(x) &= \lambda^{-1}_{\rho_{\lambda_x\lambda^{-1}_x(a)}\rho^{-1}_x(x)}\rho_a\lambda_{\rho^{-1}_x(x)}(x)\\
    &=\lambda^{-1}_{\rho_{a}\rho^{-1}_x(x)}\rho_a\lambda_{\rho^{-1}_x(x)}(x)\\
    &=\lambda^{-1}_{\rho_{b}\rho^{-1}_x(x)}\rho_b\lambda_{\rho^{-1}_x(x)}(x)\\
    &=\rho_{\lambda^{-1}_x(b)}(x).
\end{align*}
So,
\begin{align}\label{claimextra}
    \cm_a(x) &=\lambda_a\rho_{\lambda^{-1}_x(a)}(x) 
    =\lambda_b\rho_{\lambda^{-1}_x(b)}(x)  =\cm_b(x).
\end{align}
This proves the first part of the statement.

To prove the second part, 
 assume $\lambda_a =\rho_a=\id$. By (\ref{claim1}) we get that  $\rho_{\lambda^{-1}_x(a)}(x) =x$.
Then, as in (\ref{claimextra}),  $\cm_a(x)=\lambda_a\rho_{\lambda^{-1}_x(a)}(x) =x$, as desired.
\end{proof}
\end{proposition}

\begin{lemma}\label{ManyClaims}
Let $(A,+,\circ ,\lambda , \cm)$ be a \ourstructure.
Let $a,a',x\in A$.
\begin{enumerate}
    \item If $\lambda_a =\lambda_{a'}$ then $\lambda_{\rho_x(a)}=\lambda_{\rho_x(a')}$.
    \item if $\lambda_a =\lambda_{a'}$ and $\rho_a =\rho _{a'}$ then  $\lambda_{\lambda_x(a)}=\lambda_{\lambda_x(a')}$.
\end{enumerate}
If, furthermore, $A$ is non-degenerate then $\lambda_a =\lambda_{a'}$ and $\rho_a =\rho_{a'}$  implies
\begin{enumerate}
    \item[(3)]
    $\lambda_{\lambda_x^{-1} (a)} =\lambda_{\lambda_x^{-1}(a')}$
and $\rho_{\lambda_x^{-1} (a)} =\rho_{\lambda_x^{-1}(a')}$.
    \item[(4)] 
    $\rho_{\rho_x(a)} =\rho_{\rho_x(a')}$
    and  $\rho_{\lambda_x(a)} =\rho_{\lambda_x(a')}$.
    \item[(5)] 
    $\lambda_{\cm_{x}(a)}=\lambda_{\cm_x(a')}$ and
     $\rho_{\cm_{x}(a)}=\rho_{\cm_x(a')}$.
    \item[(6)]  
    $\cm_{\cm_{x}(a)} =\cm_{\cm_x(a')}$.
\end{enumerate}
\end{lemma}
\begin{proof}
(1) Since $\lambda_{\lambda_a(x)} \lambda_{\rho_x(a)} =\lambda_a \lambda_x =\lambda_{a'}\lambda_x =\lambda_{\lambda_{a'}(x)}\lambda_{\rho_x (a')}=\lambda_{\lambda_{a}(x)}\lambda_{\rho_x (a')}$ and, because $\lambda_{\lambda_a(x)}$ is bijective, we get that $\lambda_{\rho_x(a)}=\lambda_{\rho_x(a')}$.
 
(2) Since 
$\lambda_{\lambda_x(a)} \lambda_{\rho_a(x)} =\lambda_x \lambda_a =\lambda_{x}\lambda_{a'} =\lambda_{\lambda_{x}(a')}\lambda_{\rho_{a'} (x)}=\lambda_{\lambda_{x}(a')}\lambda_{\rho_a (x)}$ and, because $\lambda_{\rho_a(x)}$ is bijective, we get that  $\lambda_{\lambda_x(a)}=\lambda_{\lambda_x(a')}$.

Assume now that $A$ is non-degenerate.

(3)
First we claim that $\lambda_x \lambda_{\lambda^{-1}_x(a)} = \lambda_a \lambda_{\rho_{\lambda^{-1}_x(a)}(x)}$. Indeed, by \eqref{eq:YBE1},
\begin{equation}\label{eq:Claim2'}
    \lambda_x \lambda_{\lambda^{-1}_x(a)} = \lambda_{\lambda_x(\lambda^{-1}_x(a))} \lambda_{\rho_{\lambda^{-1}_x(a)}(x)} = \lambda_a \lambda_{\rho_{\lambda^{-1}_x(a)}(x)}.
\end{equation}
Using \eqref{eq:YBE2}, we also get
\begin{align*}
    \rho_{\lambda^{-1}_x(a)}(x) &= \lambda^{-1}_{\rho_{\lambda_x(\lambda^{-1}_x(a))} \rho^{-1}_x(x)} \rho_{\lambda_{\rho_x(\rho^{-1}_x(x))} \lambda^{-1}_x(a)} \lambda_{\rho^{-1}_x(x)}(x)
    \\ &= \lambda^{-1}_{\rho_{a} \rho^{-1}_x(x)} \rho_{\lambda_{x} \lambda^{-1}_x(a)} \lambda_{\rho^{-1}_x(x)}(x)
    \\ &= \lambda^{-1}_{\rho_{a} \rho^{-1}_x(x)} \rho_{a} \lambda_{\rho^{-1}_x(x)}(x),
\end{align*}
so as $\rho_a = \rho_{a'}$ this implies
\begin{align*}
    \rho_{\lambda^{-1}_x(a)}(x) &=  \lambda^{-1}_{\rho_{a} \rho^{-1}_x(x)} \rho_{a} \lambda_{\rho^{-1}_x(x)}(x)
    \\ &= \lambda^{-1}_{\rho_{a'} \rho^{-1}_x(x)} \rho_{a'} \lambda_{\rho^{-1}_x(x)}(x)
    \\ &= \rho_{\lambda^{-1}_x(a')}(x).
\end{align*}
Hence,
\begin{align*}
    \lambda_x \lambda_{\lambda^{-1}_x(a)} 
    &\stackrel{\eqref{eq:Claim2'}}{=} \lambda_a \lambda_{\rho_{\lambda^{-1}_x(a)}(x)}
    = \lambda_a \lambda_{\rho_{\lambda^{-1}_x(a')}(x)}
    = \lambda_{a'} \lambda_{\rho_{\lambda^{-1}_x(a')}(x)}
    \stackrel{\eqref{eq:Claim2'}}{=} \lambda_x \lambda_{\lambda^{-1}_x(a')}.
\end{align*}
Since $\lambda_x$ is bijective, we conclude that $\lambda_{\lambda^{-1}_x(a)} = \lambda_{\lambda^{-1}_x(a')}$. Similarly, one can prove that $\rho_{\lambda^{-1}_x(a)} = \rho_{\lambda^{-1}_x(a')}$.

(4)
Because $\rho_{\rho_x(a)} \rho_{\lambda_a(x)}= \rho_x \rho_a =\rho_x \rho_{a'}= \rho_{\rho_x(a')}\rho_{\lambda_{a'}(x)}=
\rho_{\rho_x (a')}\rho_{\lambda_a(x)}$ and, because of the assumption that $\rho_{\lambda_a(x)}$ is bijective, we get that 
$\rho_{\rho_x(a)} =\rho_{\rho_x(a')}$.

Because $\rho_{\rho_a(x)} \rho_{\lambda_x(a)}= \rho_a \rho_x =\rho_{a'} \rho_{x}= \rho_{\rho_{a'}(x)}\rho_{\lambda_{x}(a')}=
\rho_{\rho_a (x)}\rho_{\lambda_x(a')}$ and, because of the assumption that $\rho_{\rho_a(x)}$ is bijective, we get that 
$\rho_{\lambda_x(a)} =\rho_{\lambda_{x}(a')}$.

(5) Recall that $\cm_x(a) =\lambda_x \rho_{\lambda^{-1}_{a}(x)}(a)$
Hence, $\lambda_{\cm_x(a)}=\lambda_{\lambda_x \rho_{\lambda^{-1}_{a}(x)}(a)}$. Substitute $e=\lambda_a^{-1}(x)=\lambda^{-1}_{a'}(x)$ then we get
$\lambda_{\cm_x(a)}=\lambda_{\lambda_{\lambda_a(e)} \rho_e(a)}$.
Hence statement (5)  is equivalent with 
$\lambda_{\lambda_{\lambda_a(e)} \rho_e(a)} = \lambda_{\lambda_{\lambda_{a'}(e)} \rho_e(a')}= \lambda_{\lambda_{\lambda_a(e)} \rho_e(a')}$
for all $e\in A$.
Now, from part (1) we know that $\lambda_{\rho_{e}(a)}=\lambda_{\rho_e(a')}$ and from part (4) we also get that $\rho_{\rho_e(a)}=\rho_{\rho_e(a')}$. Therefore, by part (2), $\lambda_{\lambda_{\lambda_a(e)}\rho_{e}(a)} =\lambda_{\lambda_{\lambda_{a}(e)} \rho_{e}(a')}$, as desired.

Repeat the steps in the previous, but replace part (2)  by part  (4), i.e. if $\rho_x$ is bijective and $\lambda_a=\lambda_{a'}, \rho_a=\rho_{a'}$ then $\rho_{\lambda_x(a)} = \rho_{\lambda_x(a')}$.

(6) 
Note that because of part  (5)  and Proposition~\ref{prop:c_a=c_b} we immediately obtain that $\cm_{\cm_b(a)} = \cm_{\cm_b(a')}$.
\end{proof}

\begin{theorem}\label{retractassoc}
Let $(A,+,\circ, \lambda, \cm)$ be a non-degenerate \emph{\ourstructure}. Then 
$ \mathcal{G}(A)=\{ f(a)= (\cm_a,\lambda_a, \rho_a)\mid a\in A\}$ is a \ourstructure
for the operations
$$f(a)\circ f(b)= (\cm_{a\circ b},\lambda_{a\circ b},\rho_{a\circ b}), \quad
f(a)+f(b) = (\cm_{a+b}, \lambda_{a+b},\rho_{a+b}),$$   $\lambda$-map
 $$\lambda_{f(a)}f(b) =(\cm_{\lambda_a(b)},\lambda_{\lambda_{a}(b)},\rho_{\lambda_{a}(b)})=f(\lambda_a (b)).$$
 and $\cm$-map
  $$\cm_{f(b)}f(a) = (\cm_{\cm_{b}(a)}, \lambda_{\cm_{b}(a)}, \rho_{\cm_{b}(a)})=f(\cm_b(a)).$$

The map $f: A\rightarrow \mathcal{G}(A): a\mapsto f(a)$ is a \ourstructure epimorphism
and  the associated solution of $\mathcal{G}(A)$ is non-degenerate.

Note that because of Proposition~\ref{prop:c_a=c_b} the natural  mapping 
$\mathcal{G}(A)\rightarrow \{ (\lambda_a, \rho_a)\mid a\in A\}$ is bijective and thus the latter can be considered as a \ourstructure.
\end{theorem}
\begin{proof}
That the mentioned operations are well-defined has been proven in the beginning of this section in Lemma \ref{PermSem}. That the $\lambda$-map is well-defined follows from Lemma~\ref{GammaSurjective} and that the $\cm$-map is well-defined follows from Lemma~\ref{ManyClaims}.

Note that, for each $a\in A$ we  know from Lemma~\ref{GammaSurjective}  that $\lambda_{f(a)}$ is surjective. To prove it also is injective assume that $\lambda_{f(a)} (f(b)) =\lambda_{f(a)}(f(b'))$. Then $\lambda_{\lambda_a (b)} =\lambda_{\lambda_{a}(b')}$ and $\rho_{\lambda_a (b)}=\rho_{\lambda_a (b')}$ and thus it follows from Lemma~\ref{ManyClaims}.(2) that $\lambda_b =\lambda_{\lambda_a^{-1} \lambda_a  (b)} =\lambda_{\lambda_a^{-1} \lambda_a (b')}=\lambda_{b'}$.
 Similarly,  we get that $\rho_b=\rho_{b'}$.
From Proposition~\ref{prop:c_a=c_b} it follows that also $\cm_b =\cm_{b'}$. Hence, $f(b)=f(b')$, as desired.

It is now straightforward to verify that all required conditions for $\mathcal{G}(A)$  to be a \ourstructure are satisfied
(in particular, $\lambda : \mathcal{G}(A)\rightarrow \Aut (\mathcal{G}(A),+)$) and that $f: A \rightarrow \mathcal{G}(A)$ is an epimorphism of \ourstructures.

It remains to show that the associated solution of $\mathcal{G}(A)$ is right non-degenerate. 
We prove this via Lemma~\ref{homomorphismqsemitruss}. 
Hence, because of Proposition~\ref{prop:c_a=c_b}, we need to prove the following: if $f(a)=f(a')$, $f(b)=f(b')$ then 
$\lambda_{\rho_a^{-1}(b)} =\lambda_{\rho_{a'}^{-1}(b')}$ and $\rho_{\rho_a^{-1}(b)}= \rho_{\rho_{a'}^{-1}(b')}$.
Clearly, without loss of generality, we may assume that $a=a'$.
Hence we need to prove Lemma~\ref{ManyClaims}.(3) for $\rho $ and $\lambda$ interchanged.
For this notice that
if $r:X\times X \rightarrow X \times X: (x,y)\mapsto (\lambda_x (y), \rho_y(x))$ is a solution then so is
$\tau r \tau :X^{2} \times X^{2}: (x,y)\mapsto (\rho_x (y) , \lambda_y (x))$, where $\tau (x,y)=(y,x)$.
Indeed, $r_{12}r_{23}r_{12}=r_{23}r_{12}r_{23}$ if and only if $(\tau r_{12}\tau)(\tau r_{23} \tau) (\tau r_{12}\tau)=
(\tau r_{23}\tau) (\tau r_{12}\tau) (\tau r_{23}\tau)$. Hence we indeed may interchange $\lambda$ and $\rho$ and the result follows.
\end{proof}

\begin{corollary}\label{permutationyb}
Let $(A,+,\circ , \lambda, \cm)$ be a non-degenerate \ourstructure. Then $(\mathcal{G}(A), + , \circ ,\lambda , \cm)$ is a \ourstructure with $(\mathcal{G}(A),\circ)$ a cancellative monoid (and thus also $(\mathcal{G}(A),+)$ is left cancellative) that satisfies the left and right Ore condition.
\end{corollary}
\begin{proof}
To prove left cancellativity,  suppose $f(a)\circ f(b) =f(a) \circ f(b')$. Then $\lambda_{a}\lambda_b=\lambda_{a}\lambda_{b'}$ and $\rho_a \rho_b =\rho_a \rho_{b'}$. As $\lambda_a$ and $\rho_a$ are bijective we obtain that $\lambda_b =\lambda_{b'}$ and $\rho_b =\rho_{b'}$. Because of Proposition~\ref{prop:c_a=c_b} we then also have $\cm_b =\cm_{b'}$. Hence $f(b)=f(b')$, as desired. So both semigroups $(\mathcal{G}(A),\circ)$ and $(\mathcal{G}(A),+)$ are left cancellative.
Similarly one proves that $(\mathcal{G}(A),\circ) $ is right cancellative. Indeed, suppose $f(a)\circ f(b) =f(a')\circ f(b)$. Then,
$\lambda_a \lambda_b =\lambda_{a'} \lambda_b$ and $\rho_a \rho_b = \rho_{a'} \rho_b$. As $\lambda_b$ and $\rho_b$ are bijective it follows that $\lambda_a =\lambda_{a'}$ and $\rho_a =\rho_{a'}$. Hence, also $\cm_a=\cm_{a'}$ and thus $f(a)=f(a')$, as desired.

Since $\mathcal{G}(A)$ is a non-degenerate \ourstructure, we obtain from  Proposition~\ref{epitruss}   that $f(a)\circ f(b) =\lambda_{f(a)} (f(b)) \circ \rho_{f(b)}(f(a))$, where $\rho$ denotes the $\rho$-map for $\mathcal{G}(A)$. As each $\lambda_{f(a)}$ and $\rho_{f(b)}$ is  bijective, the left and right conditions follow at once.
\end{proof}

We are now in a position to extend \cref{bijectivenondeg1} to a much larger class of solutions.
For this, in the non-degenerate case, we consider $\{ (\lambda_a , \rho_a) \mid a\in A\}$ as a subsemigroup of  the direct product of
$(\Sym (A),\circ)$ and its opposite group $(\Sym (A),\circ^{op})$, where $\circ$ is the composition of functions.

\begin{corollary}\label{bijectivenondeg2}
Let $(A,+,\circ, \lambda, \cm)$ be a non-degenerate \ourstructure. If   for every $a\in A$, there exists $b\in A$ such that
$\lambda_ a \lambda_b =\id$ and $\rho_{a\circ b} =\rho_b \rho_a =\id$ (for example if all $\lambda_a$ and $\rho_a$ are of finite order),
then the associated solution $r_A$ is bijective and $\mathcal{G(A)}$ is a skew left brace.
\end{corollary}
\begin{proof} Note that it follows from Proposition~\ref{prop:c_a=c_b} that if $\{ (\lambda_a , \rho_a) \mid a\in A\}$ is a group then also $(\mathcal{G}(A),\circ)$ is a group. 
Indeed, if $a\in A$ then, by assumption, $(\lambda_a ,\rho_a) (\lambda_b , \rho_b) =(\id ,\id)$ for some $b\in A$. Hence $\lambda_{a\circ b} =\rho_{a\circ b}=\id$. So, by Proposition~\ref{prop:c_a=c_b}, 
$\cm_{a\circ b}=\id$. Hence $f(a) \circ f(b)=f(a\circ b)=\id$. It follows that $(\mathcal{G}(A),\circ)$
is indeed a group.
In particular, all maps $\cm_a$ are bijective and thus, by Proposition~\ref{prop:semitruss}, $r_A$ is bijective.

The last part of the result now  follows at once from Example~\ref{exsb} and Corollary~\ref{permutationyb}.
\end{proof}

\begin{corollary}
Let $(X,r)$ be a non-degenerate solution.  If the subsemigroup $\langle (\lambda_x, \rho_x)\mid x\in X\rangle $
of  the group $(\Sym (X),\circ) \times (\Sym (X),\circ^{\text{op}})$ is a group itself then $r$ is bijective.
\end{corollary}
\begin{proof}
As before, the structure semigroup $S(X,r)$ is a non-degenerate \ourstructure and, because of the assumptions, $\mathcal{G}(S(X,r)) $ is a group.
Hence, by Corollary~\ref{bijectivenondeg2}, $r_{S(X,r)}$ is bijective. Since this map, restricted to $X^2$ is $r$, we obtain that $r$ is bijective.
\end{proof}

\section{Algebraic structure of \ourstructures}\label{sec_algebraic}
\subsection{Additive structure of \ourstructure and the structure algebra}\label{sec_structurealgebra}\hfill

In this section we investigate the algebra $K[(A,\circ)]$ of a \ourstructure $(A,+,\circ, \lambda,\cm)$. Our main result states that this algebra is left Noetherian and satisfies a polynomial identity
provided $A$ is a unital strongly $\mathbb{N}$-graded  \ourstructure  with  $A_1$ finite, $A_0=\{ 1\}$  and the diagonal map is  bijective.
As such an algebra is an epimorphic image of the structure algebra $K[M(X,r)]$, with $(X,r)$ a left non-degenerate finite solution, it is sufficient to deal with the unital structure \ourstructure $M(X,r)$. To prove this result we first determine the algebraic structure of the derived monoid $A(X,r)$.

Put
$$A=A(X,r) =\langle x\in X \mid x +y =y + \cm_y (x)\rangle,$$
the additive monoid of the unital structure \ourstructure $M(X,r)$.
Recall from Lemma~\ref{lemma:unital} that for a unital \ourstructure we have  $0=1$, where $0$ and $1$ are the identities of the additive and the multiplicative monoid respectively.
For  any $a,b\in A$,
   \begin{eqnarray}
     a+b =b+\cm_b(a), \label{leftdiv}
   \end{eqnarray}
where  $\cm: (A,+) \rightarrow \End (A,+): a\mapsto \cm_a$,
is a monoid anti-homomorphism.
In particular,  each right ideal $b+A$ of $(A,+)$ is a two-sided ideal of $(A,+)$.
One says that $a\in A$ is  \emph{divisible} by $b$ if $a=e+b+d$ for some $e,d\in A$. If $e=0$ then one says that $a$ is \emph{left divisible} by $b$.
Because of \eqref{leftdiv},  $a$ is divisible by $b$ if and only if $a$ is left divisible by $b$.

Consider the subsemigroup 
 $$\mathcal{C}=\mathcal{C}(A):=\{ \cm_a \mid a\in A\}=\langle \cm_x \mid x\in X\rangle,$$
of $\End(A,+)$. Furthermore, for $a,b\in A$,
\begin{eqnarray}
 \cm_a \cm_b =\cm_{b+a} =\cm_{a+\cm_a (b)} =\cm_{\cm_a(b)} \cm_a, \label{cnormal}
 \end{eqnarray}
 and thus
 \begin{eqnarray}
  \cm_a  \mathcal{C} \subseteq \mathcal{C} \cm_a.\label{leftsimple}
  \end{eqnarray}
 So, every left ideal of the semigroup $\mathcal{C}$ is a two-sided ideal.
If $\mathcal{C}$ is finite (for example if $X$ is finite), there exists
a positive integer, say $v$, so that 
  $$\cm_x^{v}=\cm_{vx}=\cm_{vx}^2 \mbox{ is an idempotent},$$  
for each $x\in X$. 
 
From now on $X=\{ x_1 , \ldots , x_n\}$ is a finite set
and thus any element of $A=A(X,r)$ is divisible by only finitely many elements.
Let $a\in A=\langle x_1 , \dots , x_n\rangle$ and  let $m_1$ be the maximal non-negative integer so that $a$ is left divisible by $m_1x_1$, that is  $a=m_1 x_1 + b$ and $b\in A$ can not be written as $x_1 + d$ for some $d\in A$. 
Repeat this  argument on $b$ and consider the maximal non-negative integer $m_2$ so that $b=m_2x_2 +d$,
for some $d\in A$. After at most $n$ steps we get that 
  $$a=m_1 x_1+ m_2 x_2 + \cdots + m_n x_n,$$ 
for some non-negative integers $m_i$.
Notice that  $vx_i+vx_j=vx_j+ \cm_{vx_j}(vx_i)= vx_j +v\cm_{vx_j}(x_i)$
and $\cm_{vx_j}(x_i)\in X$.
Hence,  
  $$B(v)=\{ m_1 vx_{1}+ \cdots + m_n vx_n \mid m_1, \ldots , m_n\geq 0\},$$ 
is a submonoid of $(A,+)$.
Furthermore
\begin{eqnarray}\label{finitemodule}
  A=B(v)+F(v), \: 
\mbox{ with }F(v)=\{ m_1 x_{1} +\cdots + m_n x_n \mid v>m_1, \ldots , m_n\geq 0\} .
\end{eqnarray}
Thus $A$ is a finite module over the finitely generated submonoid $B(v)$.
For $1\leq i \leq n$,
put 
  $$y_i=vx_i \quad \mbox{ and } \quad X(v)=\{ y_1 , \dots , y_n\}.$$
The solution $s:X^2 \rightarrow X^2: (x,y)\mapsto (y,\cm_y (x))$
induces a solution $s_{X(v)}: X(v)^{2} \rightarrow X(v)^{2}: (y_i,y_j) \mapsto (y_j, \cm_{y_j}(y_i))$.
Clearly,
  $$B(v)=A(X(v),s_{X(v)})=\langle y_1, \ldots , y_n\rangle.$$
Put 
    $$\cm_j =\cm_{y_j},\quad  \mbox{ for } 1\leq j \leq n.$$
Because of \eqref{cnormal} and since each element $\cm_y $, for $y\in Y$ is idempotent,  any non-identity element of $\mathcal{C}(B(v))$
can be written as $\cm_{j_1} \cdots  \cm_{j_k}$ for some $j_1, \ldots , j_k$ 
and some $k\leq n$. 
If $v$ is clear from the context we simply write $B(v)$ as $B$.

Let 
$\cm_{j_1} \cm_{j_2} \cdots \cm_{j_k}\in \mathcal{C}(B)$
with $J=\{ j_1 ,\cdots , j_k\}$ a subset of $\{ 1, \ldots , n\}$.
Then, for any $j_l\in J$,  we have 
\begin{eqnarray*}
   \cm_{j_l} \cdots \cm_{j_k} &=& \cm_{\cm_{j_l}(y_{j_{l+1}})} \cm_{j_l} \cm_{j_{l+2}} \cdots \cm_{j_{k}}\\
   &=& \cm_{\cm_{j_l}(y_{j_{l+1}})}  \cm_{\cm_{j_l}(y_{j_{l+2}})} \cdots \cm_{\cm_{j_l}(y_{j_{k}})} \cm_{j_l}.
 \end{eqnarray*}
Because  $\cm_{j_l}$ and $\cm_{j_k}$ are  idempotent we get 
 \begin{eqnarray}
   \cm_{j_l} \cdots \cm_{j_k} \cm_{j_{l} }
   &=& \cm_{\cm_{j_l}(y_{j_{l+1}})}  \cm_{\cm_{j_l}(y_{j_{l+2}})} \cdots \cm_{\cm_{j_l}(y_{j_{k}})} \cm_{j_l}\cm_{j_{l}}
   \nonumber\\
    &=& \cm_{\cm_{j_l}(y_{j_{l+1}})}  \cm_{\cm_{j_l}(y_{j_{l+2}})} \cdots \cm_{\cm_{j_l}(y_{j_{k}})} \cm_{j_l} \nonumber\\
    &=&  \cm_{j_l} \cdots \cm_{j_k} \label{rightC}
 \end{eqnarray}
Hence it follows from \eqref{rightC}  that each $ \cm_{j_1} \cdots \cm_{j_k} $ is an idempotent and 
we have shown the following property.

\begin{lemma} \label{Cband}
The semigroup $\mathcal{C}(B(v)) =\{ \cm_{j_1} \cm_{j_2} \cdots  \cm_{j_k} \mid 1 \leq j_1 ,\cdots , j_k \leq n \}\cup \{ \id \}
=\{ \cm_{j_1} \cm_{j_2} \cdots  \cm_{j_k} \mid 1 \leq j_1 ,\cdots , j_k \leq n, \mbox{ all distinct} \}\cup \{ \id \}$
is a band, i.e. a semigroup consisting of idempotents. 
\end{lemma}

 For $1\leq k\leq n$  and $\kappa\in \Sym_n$, the symmetric group of degree $n$, put 
 \begin{align*}
     t_k &=y_1 + \cdots +y_k\in A, \\
   \kappa (t_k)&=y_{\kappa (1)} + \cdots + y_{\kappa (k)}\in A, \\
   B_{\kappa (t_k)} &= \langle y_{\kappa (i)} \mid 1\leq i \leq k\rangle,
 \end{align*}
and put 
   $$T=\{ \kappa (t_k) \mid 1 \leq k \leq n,\; \kappa \in \Sym_n\}.$$ 
Let $0\neq w\in B$. Assume $w=w'+ y_k$  for some $k\leq n$.  Let  $m$ be the maximal non-negative integer so that $w'$ is left divisible by $m y_k$. Then  $w=my_k + w''+y_k$ for some $w''\in B$ that is  not left divisible by $y_k$. 
If $0\neq w''$ then repeat the argument on $w''$ and after at most $n$  steps we get that
  $$w=w_t + t,$$
for some $t\in T$ and $w_t\in B_t$.
Since $\cm_{w}=\cm_t \cm_{w_t}$, it follows from 
\eqref{rightC}  that 
   $$\cm_w =\cm_t.$$ 
  Hence,
 for $a,b\in B_t$
 \begin{eqnarray}
 a+b+t =a+t +\cm_t (b) = a+t + \cm_{a+t} (b) =b+a+t,\label{interchange}
 \end{eqnarray}
 in particular
  $$a+t+b+t =b+t +a+t,$$
that is, each  $B_t+t$ is an  abelian semigroup. 
With the notation introduced above we thus  obtain the following property.
\begin{lemma}
Let $(X,r)$ be a finite left non-degenerate solution and $A=A(X,r)$.
Then $A=B(v) +F(v)$ and $B(v)=\{ 0\} \cup \bigcup_{t\in T} (B(v)_{t} +t)$, a finite union of abelian subsemigroups.
\end{lemma}

We also have that for $\kappa \in \Sym_n$, $1\leq k\leq n$ and $1\leq l\leq k$
  \begin{eqnarray}\label{idealR}
   \kappa (t_k) +y_{\kappa (l)} &= &y_{\eta (l')} +\eta (t_k)
  \end{eqnarray}
for some $\eta \in \Sym_n$ and $1\leq l'\leq k$. 

To prove this it is sufficient to deal with the case that $\kappa =\id$. Hence, we need to show that for $1\leq l\leq k\leq n$ one can rewrite 
$t_k + y_l = y_1 + y_2 + \cdots + y_k + y_l$ as  $y_{\eta(l')}+ \eta (t_k)$ for some $\eta \in \Sym_n$ and $1 \leq l' \leq k$.
If $l=1$ then this is obvious. Assume $k>l>1$. Then
$$t_k + y_l  = y_l  + \cdots +y_k + y_l  + \cm_{y_l  + \cdots +y_k+ y_l} (y_1+ \cdots + y_{l-1}). $$
Hence, since  $\cm$ is an anithomomorphism and because of \eqref{rightC},
\begin{eqnarray*}
t_k + y_l  &=& y_l  + y_{l+1}+ \cdots +y_k + y_l  + \cm_{y_{l+1}  + \cdots +y_k+ y_l} (y_1+ \cdots + y_{l-1})\\
    &=& y_l + y_1+ \cdots + y_{l-1}+  y_{l+1}  + \cdots +y_k+ y_l \\
    &=& y_l +\eta (t_k) 
    \\&=& y_{\eta(l')} +\eta (t_k),
\end{eqnarray*}
for some $\eta$ and $l'=k$, as desired.
It remains to deal with the case $k=l$.
In this case
 $$t_k + y_k = y_k +y_k + \cm_{y_k+y_k} (y_1 + \cdots +y_{k-1}),$$
and thus, because $\cm_{y_k}$ is an idempotent and $\cm$ is an anti-homomorphism,
  \begin{eqnarray*}
   t_k + y_k &=& y_k +y_k + \cm_{y_k} (y_1 + \cdots +y_{k-1})\\
        &=& y_k + y_1 +\cdots +y_{k-1} +y_k\\
        &=& y_k +t_k,
  \end{eqnarray*}
again as desired.

Let $K$ be a field. Each
  $K[ B_t+t]$
is a left $K[B_t]$-module for the module action
 $K[B_t] \times K[B_t+t] \rightarrow K[B_t+t]$
defined by the linear extension of 
  $ (a, d+t) \mapsto a+d+t$,
where $a,d\in B_t$ and $t\in T$.
Because of \eqref{interchange},  $K[B_t+t]$ is a cyclic left module over the 
finitely generated commutative $K$-algebra $K[B_t]^{ab}=K[B_t]/[K[B_t],K[B_t]]$, i.e.  the abelianization of the algebra $K[B_t]$.
Hence the commutative (non-unital) algebra $K[B_t+t]$ is a Noetherian left $K[B_t]$-module.

Obviously, for $\kappa \in \Sym_n$, we have that  $B_{\kappa (t_n)}+\kappa (t_n) =B+\kappa (t_n)$ is a left ideal of $B$.
Hence, $K[B_{t_n} + t_n]$ is a left  ideal of $K[B]$ and thus it  is a Noetherian left
$K[B]$-module   and, as ring, it is commutative. Put
 $$R_n= \sum_{\kappa \in \Sym_n} K[B_{\kappa (t_n)}+\kappa (t_n)].$$ 
Clearly this is a Noetherian left $K[B]$-module. Furthermore, from \eqref{idealR} it follows that 
$R_n$  is an ideal  of $K[B]$. 
As a ring it is  a sum of left ideals that are  PI-algebras (actually each of the left ideals is commutative as a ring) and  thus,
by a result of Rowen \cite{Rowen1976} (or more general a result of K\c{e}pczyk \cite{Kcepczyk}, which states that a ring which is a sum of two  PI rings is again a PI ring), $R_n$ is a PI-algebra.

Next consider the $K$-algebra  $K[B]/R_n$. For any $\kappa \in \Sym_n$, 
it easily is verified that  $K[B_{\kappa( t_{n-1})}+\kappa(t_{n-1})]/R_n$ is a left $K[B]/R_n$-submodule of $K[B]/R_n$. 
With a similar reasoning  as above, we obtain that this is a Noetherian cyclic left module over $K[B_{\kappa(t_{n-1})}]$
and thus  $K[B_{\kappa( t_{n-1})}+\kappa(t_{n-1})]/R_n$ is a Noetherian left module over $K[B]/R_n$. Furthermore, as a ring it is commutative.
Put
   $$R_{n-1} =\sum_{\kappa \in \Sym_n}  K[B_{\kappa (t_{n-1})}+\kappa (t_{n-1})].$$
By \eqref{idealR},  $R_{n-1}+R_n/R_n$ is an ideal in $K[B]/R_n$ and it is  a Noetherian left $K[B]/R_n$-module. 
Consequently $R_{n-1}+R_n$ is a Noetherian left $K[B]$-module.
One can repeat this process in an obvious manner making use of the sets
   $$R_{k} =\sum_{\kappa \in \Sym_n}  K[B_{\kappa (t_k)}+\kappa (t_k)],$$
where $1\leq k\leq n$. Put $R_0=K$.  It follows, in particular the algebra  $K[B]=R_0+ R_{1}+ \cdots +R_n$ is left  Noetherian.
Since a $K$-algebra $V$  is PI  if it has an ideal $I$ so that both algebras $I$ and $V/I$ are PI, we obtain 
 that $K[B]$  also is a  PI-algebra.
By \eqref{finitemodule},  $K[A]=\sum_{f\in F} K[B]+f$. So the $K$-algebra $K[A]$ is a finite module over the subalgebra $K[B]$
and thus
$K[A]$ also is a left Noetherian PI algebra (actually it is a Noetherian left $K[B]$-module).
We thus have proven the following result. 

\begin{theorem} \label{ThmNoetherianPI}
Let $(X,r)$ be a finite left non-degenerate solution and put $X=\{ x_1 , \ldots , x_n\}$. 
Let $v$ be a positive integer so that $\cm_x^v$ is an idempotent endomorphism for each $x\in X$.
Then, with the notation introduced above, the derived monoid $A=A(X,r)$ satisfies the following properties.
\begin{enumerate}
    \item $A(X,r) =B(v) +F(v)$, i.e. $A(X,r)$ is a finite module over $B(v)=A(Y,s_Y)=\langle y_1=vx_1, \ldots , y_n=vx_n\rangle.$
    \item $\mathcal{C}(B(v))$ is a band;
    \item  for $1\leq k \leq n$, $t_k =y_1 + \cdots y_k\in B$ and
    $\kappa \in \Sym_n$,  each  $B_{\kappa(t_{k})}+ \kappa (t_k)$ is a commutative semigroup, where 
    $B_{\kappa(t_{k})}=  \langle y_{\kappa (i)} \mid 1\leq i \leq k\rangle$.
    \item for $1\leq k \leq n$  and $B_k = \bigcup_{\kappa \in \Sym_n} B_{\kappa(t_{k})} +\kappa (t_k)$,
       $$B_n \subseteq B_n \cup B_{n-1} \subseteq  \cdots \subseteq B_n \cup \cdots \cup B_2 \cup B_1 \subseteq B,$$
       is an ideal chain of $B(v)$,
    \item  each Rees factor semigroup of the ideal chain  
     is a union of left ideals
      $$(B_n\cup B_{n-1} \cup \cdots \cup B_{i-1} \cup 
        B_{\kappa(t_{i})} +\kappa (t_{i}))/(B_n \cup B_{n-1} \cup \cdots \cup B_{i-1}),$$
        with $\kappa \in \Sym_n.$
     \item $A$ satisfies the ascending chain condition on left ideals.
     \item for any field $K$, each factor of the ideal chain 
      $$\{0\} \subseteq K[B_n] \subseteq K[B_n]+ K[B_{n-1}] \subseteq \cdots \subseteq K[B_n ] +\cdots + K[B_1] \subseteq K[B(v)],$$
      of $K[B(v)]$ 
      is  a Noetherian left $K[B(v)]$-module that is a finite sum of commutative rings.
\end{enumerate}
In particular, $K[A]$ is a Noetherian left $K[B(v)]$-module, and  $K[A]$ is a left Noetherian PI-algebra of finite Gelfand-Kirillov dimension bounded by $n$.
\end{theorem}

In general the algebra $K[A]$ is not right Noetherian.
Indeed, again as in Example~\ref{exidempotent}, consider  solutions $(X,r)$ with $r(x,y)=(y,y)$ for $x,y\in X$. So $r$ is idempotent and left non-degenerate. For simplicity take $X=\{ x,y\}$.
To avoid confusion of the operations in the structure algebra, we will write $A=A(X,r)$ multiplicatively, i.e.
$A=A(X,r) =\langle x,y \mid xy=yy,\; yx=xx \rangle =\{ x^{n}, y^{n} \mid n\geq 0\}$.
In $K[A]$ we have $(x^{n}-y^{n})a=0$ for any $a\in A$. Hence 
$\sum_{n>0}K(x^n-y^n)$ is a right ideal of $K[A]$ that obviously  is not finitely generated as a right ideal.
 
As an application of Theorem~\ref{ThmNoetherianPI} we claim that the structure algebra $K[M(X,r)]$ is left Noetherian in case $(X,r)$ is a finite left non-degenerate solution, provided the diagonal map $\mathfrak{q}$ is surjective  (i.e. $X\rightarrow X: x\mapsto \lambda_x^{-1}(x)$ is surjective). With a standard length and induction argument one can easily show that the latter is equivalent with the diagonal map   $\mathfrak{q} : M(X,r) \rightarrow M(X,r) : a\mapsto \lambda_a^{-1}(a)$ being surjective.

From Section~\ref{sec_StructureMonoid}, we know that
 $M(X,r)$ is a submonoid of the semidirect product $A(X,r) \rtimes \im \lambda$ and $\im \lambda=\{\lambda_a \mid a\in M(X,r)\}$ is a finite group. From Theorem~\ref{ThmNoetherianPI} we get that the algebra $K[A(X,r) \rtimes \im \lambda]$ is a 
 Noetherian left $K[B(v)]$-module.
 Hence, to prove 
 that $K[M(X,r)]$ is left Noetherian  
 it is sufficient to show that one can choose the positive integer $v$ in the statement of Theorem~\ref{ThmNoetherianPI} so that $B(v)\subseteq M(X,r)$, i.e. for all $x\in X$ we also may assume that $\lambda_{vx}=\id$, and thus $M(X,r)$ is a finite 
 module over $B(v)$.
This is what will be shown in the following result.

Also recall that an algebra over a field  is said to be  \emph{representable} if it can be embedded into a  matrix algebra over some field.
A well-known result of Ananin \cite{Ananin} states that  any finitely generated  left Noetherian PI-algebra over a field is representable. Conversely, any representable algebra is a PI-algebra.
 
\begin{corollary} \label{corollary:pinoetherian}
 Let $(X,r)$ be a finite left non-degenerate solution and let $K$ be a field. If the diagonal map $\mathfrak{q}:X\rightarrow X$ is bijective,
then there exists a  positive integer  $v$  so that $\cm_x^v$ is an idempotent endomorphism for each $x\in X$ and $B(v)\subseteq M(X,r)$.

Consequently, $M(X,r) =B(v)+F$, for some finite subset $F$ of $M(X,r)$
and  
the structure algebra $K[M(X,r)]$ is a connected 
graded left Noetherian representable algebra.
In particular, so is the algebra $K[(A,\circ)]$ for any unital strongly $\mathbb{N}$-graded  \ourstructure $(A,+,\circ, \lambda,\cm)$ with  $A_1$ finite, $A_0=\{ 1\}$  and with  bijective diagonal map.
\end{corollary}
 \begin{proof}
  From Theorem~\ref{ThmNoetherianPI}
we know that $A=A(X,r)$ is a finite module over the left Noetherian ring $K[B]$, where $B=B(v)$ and $v$ is any positive integer so that each $\cm_x^{v}$ is an idempotent, for $x\in X$. Without loss of generality, replacing $v$ by a multiple if needed, we also may assume that $\lambda_a^{v}=\id$, for each $a\in M(X,r)$.
 
 Now, for $x\in X$, define $x^{(1)} =x$ and recursively $x^{(i+1)} =\mathfrak{q} (x^{(i)})=\lambda_{x^{(i)}}^{-1}(x^{(i)})$.
Since, by assumption, $\mathfrak{q}$ is a bijective map and thus it is of finite order, i.e. $\mathfrak{q}^m=\id$, for some positive integer $m$. 
So $x^{(m+1)}=x^{(1)}=x$. Since for any positive integer $k$ we have that $kx=x^{(1)}\circ x^{(2)} \circ \cdots \circ  x^{(k)}$, we obtain that
$vmx=(x^{(1)}\circ x^{(2)} \circ \cdots \circ x^{(m)})^{v}$ and thus $\lambda_{mvx} =(\lambda_x \lambda_{x^{(1)}} \cdots \lambda_{x^{(m)}})^v=\id$. So $B(mv) \subseteq M(X,r)$, as desired.
 
 Because of Corollary~\ref{CorEpiGraded}, for any strongly  $\mathbb{N}$-graded unital \ourstructure   $(A,+,\circ, \lambda , \cm)$ with bijective diagonal map  and
 with  $A_1$ finite and $A_0=\{ 1\}$,
 we have that the algebra $K[(A,\circ)]$ is a graded epimorphic image of the structure algebra $K[M(A_1,r_{A_1})]$ for the left non-degenerate finite solution $(A_1,r_{A_1})$ with bijective diagonal map. Hence, the last statement of the result follows from the first one.
 \end{proof}

In case $(X,r)$ is a finite left non-degenerate bijective solution (so a finite non-degenerate bijective solution, by \cref{maintheorembijective}), then each element of $A(X,r)$ is a normal element, as each $\cm_a$ is bijective, and thus  $\mathcal{C}(A)$ is a group with $\id$ as the only idempotent. Hence, for an appropriate choice of $v$ we get that  $B(v)\subseteq (Z(A)\cap M(X,r))$. It follows that $w\circ B(v)=B(v)\circ w$ for each element $w\in M(X,r)$. Hence one obtains Theorem 4 
in \cite{JeKuVA19Cor} (see also \cite{JeKuVA19}); in particular 
$K[M(X,r)]$ is a module-finite normal extension of the commutative affine subalgebra $K[B(v)]$ and thus
$K[M(X,r)]$ is a left and right Noetherian PI-algebra.  Note that the bijective left non-degenerate assumption implies that the diagonal map is bijective, because of Theorem~\ref{maintheorembijective} and Lemma~\ref{lemma:diagonal}.

\subsection{\ourstructures with a left simple semigroup of endomorphisms and their solutions}\label{sec_simple}\hfill

Let $A=(A,+,\circ , \lambda ,\cm)$ be a \ourstructure and assume 
 $\mathcal{C}=\{ \cm_a \mid a\in A\}$
 is a finite subsemigroup of $\End(A,+)$. We know from \eqref{leftsimple} that  every left ideal of $\mathcal{C}$ is a two-sided ideal.
It is then well known that this semigroup has an ideal chain
  $$C_0 \subseteq C_1 \cdots \subseteq C_{n-1} \subseteq C_n=\mathcal{C}, $$
with $C_0=\{ \theta\}$ or $C_0$ the empty set, and all factors $C_i/C_{i-1}$ are either a null semigroup or completely ($\theta$)-simple. In the latter case we know that such a semigroup is of the form $\mathcal{M}^{0}(G,k,l,P)$, with $G$ a maximal 
subgroup of $\mathcal{C}$ and $P$ a regular sandwich matrix. Since every element of such a semigroup has an idempotent as left identity, it is easy to see that $R\cup C_{i-1}$, where $R$ is a ``column'' of this semigroup, is a left ideal of $\mathcal{C}$. But since this must be a two-sided ideal, we get that $l=1$.

So the principal factors that are completely ($\theta$)-simple (or possibly simple in case of $C_1$) are of the form
$\mathcal{M}^{0}(G,k,1,P)$, and after normalising, we also may assume that each entry of the row matrix $P$ is the identity of $G$.
So, this semigroup is a finite union of finite groups and the idempotents form a subsemigroup so that $ef =e$ for any two non-zero idempotents.

In this subsection we will deal with the case that $\mathcal{C}$ is left simple (for example if $(A,+)$ is a right simple semigroup),  i.e. $\mathcal{C}=C_1$ and $C_0 =\emptyset$ and show that $(A,+)$
is determined by skew left braces (actually racks).

So, from now on, assume $\mathcal{C}$ is a finite  left simple semigroup and thus 
 $$\mathcal{C}= G_1 \cup \cdots \cup G_k,$$ 
a disjoint union of finite groups, say with identity $1_1, \ldots, 1_k$ respectively. We have 
   \begin{eqnarray}
       G_iG_j \subseteq G_i \quad \mbox{and} \quad  1_i1_j=1_i.\label{multiplication}
 \end{eqnarray}
Consider the  anti-homomorphism
 $\cm:(A,+) \rightarrow \End(A,+)$ and 
define the following left ideals  of $(A,+)$, for $1\leq i\leq n$,
    $$A_i=\cm^{-1} (G_i).$$ 
For $a\in A_j$ and $b\in A_i$, since  $\cm_b \cm_a = \cm_{\cm_b (a)}\cm_b$, and thus, by  \eqref{multiplication},  we get  that 
$\cm_{\cm_b (a)} \in G_i$, i.e. $\{ \cm_b(a) \mid \cm_b \in G_i,\; a \in A \} \subseteq A_i$.
Because $1_i \sigma_b =\sigma_b$ the set $\{ \cm_b(a) \mid \cm_b \in G_i,\; a \in A \}$ is a subsemigroup of $(A,+)$. We will denote it as $G_i(A)$.
For a subset $T$ of $A$ we denote by $G_i(T)$ the  set $\{ \cm_b(t) \mid \cm_b \in G_i,\; t \in T \} $.
Note that $G_i(A_i)\subseteq G_i(A)=G_i 1_i (A) \subseteq G_i (A_i)$ and 
thus    $$G_i(A)=G_i(A_i).$$
We also have that 
 \begin{eqnarray}
 a+b=b+\cm_b (a)=\cm_b (a) + \cm_{\cm_b(a)}(b).\label{sumoftwo}
 \end{eqnarray}
Hence it follows that  $a+b \in G_i (A_i)$ and that 
  each  $G_i(A_i)$ is a   left ideal of the semigroup $(A,+)$.
Furthermore, we get the  following restrictions of the derived solution of $A$,
 $$\cm_{j,i} : A_j \times A_i \rightarrow A_i \times A_i: (a,b) \mapsto (b, \cm_b(a)).$$
In particular, 
each
$$\cm_{i,i} : A_i \times A_i \rightarrow A_i \times A_i: (a,b) \mapsto (b,\cm_b(a)),$$
is a solution.
Clearly this map restricts to solutions 
 $$\widetilde{\cm}_{i,i} : G_i (A_i) \times G_i (A_i) \rightarrow G_i (A_i) \times G_i (A_i) .$$

As $\mathcal{C}$ is finite, there exists a positive integer, say  $v$, so that if $b\in A_i$ then  $\cm_b^v = 1_i$. 
Hence, it follows that $$\cm_b = 1_i\cm_b = \cm_b^{v+1} = \cm_{\cm_b^v(b)}\cm_b^v = \cm_{1_i(b)}1_i = \cm_{1_i(b)}.$$
So 
$$\cm_{j,i} : A_j \times A_i \rightarrow A_i \times A_i: (a,b) \mapsto (b, \cm_{1_i(b)}(a)).$$

Now for $a\in A_i$ we have that $\cm_a (G_i (A_i)) \subseteq  G_i (A_i)$ and $1_i$ acts as the identity on $G_i(A_i)$.
As $G_i$ is a group we get that each map
 $\cm_{a} : G_i(A_i) \rightarrow G_i(A_i)$
is bijective (and of finite order, as $\mathcal{C}$ is finite). Hence the restriction
 $$\widetilde{\cm}_{i,i} : G_i (A_i) \times G_i (A_i) \rightarrow G_i (A_i) \times G_i (A_i), $$
is a bijective non-degenerate solution and thus $G_i(A_i)$ is a normalising semigroup, i.e. $G_i(A_i) +a =a+G_i(A_i)$ for all $a\in A_i$.

Furthermore, for $a \in A_i$ and $b \in A$, we have that $\sigma_a(b) = \sigma_a(1_i(b))$. Hence, the solution $\cm_{i,i}$ on $A_i$ is determined by its bijective non-degenerate subsolution on $G_i(A_i)$ and the projections $1_i$.

To conclude, we have proved the  first part of the following result.

 \begin{theorem}\label{descrleftsimple}
 Let $A=(A,+,\circ, \lambda , \cm)$  be a \ourstructure. Assume $\mathcal{C}$ is a finite left simple semigroup, let $E=E(\mathcal{C})$ be the subsemigroup consisting of the idempotents and $G_e$ the maximal subgroup of $\mathcal{C}$ containing $e \in E$. Then, $A=\cup_{e\in E} A_e$, a disjoint union of left ideals $A_e=\cm^{-1}(G_e)$, so that the derived solution $s_A$ associated to $A$ restricts to a bijective non-degenerate solution on $G_e(A_e)$, say $s_{A_e}$
 and for  $a_e\in A_e$, $a_f\in A_f$ we have $s_A(a_e,a_f) =s_{A_f}(f(a_e),a_f) =( a_f, \cm_{f(a_f)}(f(a_e)))$, 
 i.e. the derived solution is determined by bijective non-degenerate solutions.
 
 If, furthermore, $A$ is strongly $\mathbb{N}$-graded with  $A_1$ finite if  $A_0=\emptyset$ then, 
for each $e\in E(\mathcal{C})$, 
  $A_e \setminus G_e(A_e)$ is finite and $G_e(A_e)$ also is strongly $\mathbb{N}$-graded.
 \end{theorem}
 \begin{proof}
 It only remains to prove the second part of the statement. So, assume $A$ is strongly
 $\mathbb{N}$-graded with  $A_1$ finite if  $A_0=\emptyset$.
 Let $e\in E$.
 Because of \eqref{sumoftwo} each element of $A_e\setminus (A_0\cup A_1)$ belongs to $G_e(A_e)$.
 Clearly $A_e$ is strongly $\mathbb{N}$-graded and hence so is  $G_e(A_e)$.
 Hence the result follows.
 \end{proof}

\section*{Acknowledgement}
The first named author is supported by the Engineering and Physical Sciences Research Council [grant number EP/V005995/1].
The second author was supported in part by Onderzoeksraad of Vrije Universiteit Brussel and Fonds voor Wetenschappelijk Onderzoek
(Flanders), grant G016117. The third author is supported by Fonds voor Wetenschappelijk Onderzoek (Flanders), via an FWO post-doctoral fellowship, grant 12ZG221N.
The fourth author is supported by Fonds voor Wetenschappelijk Onderzoek (Flanders), via an FWO Aspirant-fellowship, grant 11C9421N.

For the purpose of open access, the author has applied a CC BY public copyright licence  to any Author Accepted Manuscript version arising.

Data Access Statement: Data sharing is not applicable to this article as no datasets were generated or analysed in this research.

\bibliographystyle{abbrv}
\bibliography{bibliography}
\end{document}